%% file: TAC-Viale-Pierobon-submission-apr2026.tex
\documentclass[11pt]{amsart}

\usepackage[a4paper]{geometry}
\usepackage{times}
\usepackage[italian,english]{babel}
\usepackage[latin1]{inputenc}
\usepackage[pdftex]{color, graphicx}
\usepackage{fancyhdr}
\usepackage{mathtools}

\usepackage[autostyle, italian=guillemets]{csquotes}

\usepackage{amsmath,amsthm}
\usepackage{amssymb,amsfonts}
\usepackage{mathrsfs}

\usepackage{tikz}
\usepackage{graphicx}
\usepackage{stmaryrd}
\usepackage{enumerate}
\usepackage{xypic}

\usepackage{microtype}
\usepackage{hyperref}
\usepackage{amscd}

\usepackage[mode=multiuser,status=draft,lang=english]{fixme}
\fxsetup{theme=colorsig}

\FXRegisterAuthor{M}{aM}{MV}
\FXRegisterAuthor{P}{aMP}{MP}

\newcommand{\nphi}{\varphi}

\newcommand{\Mod}{\textrm{Mod}}

\newcommand{\C}{\mathcal{C}}
\newcommand{\LL}{\mathcal{L}}
\newcommand{\down}{\textrm{DOWN}}

\newcommand{\F}{\mathcal{F}}
\newcommand{\G}{\mathcal{G}}

\newcommand{\Id}{\textrm{Id}}
\newcommand{\Hom}{\textrm{Hom}}

\newcommand{\malg}{\bool{MALG}}

\newcommand{\Presh}{\textrm{Presh}}
\newcommand{\SPresh}{\textrm{S-Presh}}
\newcommand{\exPresh}{\textrm{ExPresh}}
\newcommand{\exSh}{\textrm{ExSh}}
\newcommand{\Sh}{\textrm{Sh}}

\newcommand{\modx}{\Mod^{\textrm{bool}}}
\newcommand{\modxx}{\Mod^{\textrm{cba}}}

\newcommand{\exmodxx}{\textrm{ExMod}^{\textrm{cba}}}

\newcommand{\bundle}{\textrm{Bundle}}

\newcommand{\exbundle}{\textrm{ExBundle}}

\newcommand{\bundleopen}{\textrm{Bundle}_{\textrm{open}}}
\newcommand{\stexetale}{\textrm{St-ExEtale}}
\newcommand{\exetale}{\textrm{ExEtale}}

\newcommand{\etale}{\textrm{Etale}}

\title{Boolean valued models, presheaves, and \'etal\'e spaces}

\author{Moreno Pierobon, Matteo Viale}
\thanks{Both the authors acknowledge support from INDAM through GNSAGA. The second author acknowledges support from the project:
\emph{PRIN 2017-2017NWTM8R
Mathematical Logic: models, sets, computability.}
}

\input{macros.tex}

\begin{document}

\begin{abstract}
Boolean valued models for a signature $\mathcal{L}$ are generalizations of $\mathcal{L}$-structures in which we allow the
$\mathcal{L}$-relation symbols to be interpreted by Boolean truth values. For example, for elements $a,b\in\mathcal{M}$ with
$\mathcal{M}$ a $\bool{B}$-valued $\mathcal{L}$-structure for some Boolean algebra $\bool{B}$, 
$(a=b)$ may be neither true nor false, but get 
an intermediate  truth value in $\bool{B}$. 
In this paper we expand and relate the work of Mansfield and others on the semantics of Boolean valued models, 
and of Scott, Fourman, Monro, and others on the adjunctions between $\bool{B}$-valued models and presheaves on the Boolean algebra 
$\bool{B}$.
We start giving a different proof of a result by Monro identifying presheaves on (complete) Boolean algebras with Boolean valued models, and sheaves with Boolean valued models having the \emph{mixing property}. We also give an exact topological 
characterization (the so called \emph{fullness property}) of which Boolean valued models satisfy 
\L o\'s Theorem (i.e. the general form of the Forcing Theorem
which Cohen ---Scott, Solovay, Vopenka--- established for the special case given by the forcing method in set theory).
Next we separate the fullness property from the mixing property, by showing that the latter is strictly stronger.
Then we give an exact categorical characterization of which 
presheaves correspond to full Boolean valued models in terms of the structure of global sections of their associated \'etal\'e space.
Finally we introduce a topological presentation via \'etal\'e bundles of the sheafification process (with respect to the supremum Grothendieck topology) for presheaves on a complete Boolean algebra, linking these topological/category theoretic results to the theory of Boolean valued models.

\end{abstract}

\maketitle

\section*{Summary of main results}
This paper analyzes the correspondences existing between Boolean valued models and presheaves on Boolean algebras. Our aim is to lay foundations for bridging the gap between set theorists and category theorists, and to explore further the deep connections of the Boolean valued approach to forcing 
with sheaf theory.
We start giving a brief summary of the main results (in some cases we are forced to be inaccurate or cryptic as the precise formulation of many of our theorems requires a terminology that cannot be given right away):

\begin{enumerate}
\item \label{equiv:boo}
We show that there exists an adjunction 
between the category of Boolean valued models (see Definitions \ref{def:boolvalmod}, \ref{def:boolvalmorph})
and the category whose objects are
presheaves on Boolean algebras and whose arrows are given by adjoint homomorphisms (Theorem \ref{thm:adjbvmodpresh}); moreover this adjunction specializes
to an equivalence of categories between the family of \emph{extensional} Boolean valued models (Definition
\ref{def:boolvalmod}) and the \emph{separated} presheaves 
in which all local sections are restrictions of global sections (Corollary \ref{thm:adjbvmodpresh}); note that (up to Boolean 
equivalence) every
 Boolean valued model is extensional.  In this correspondence, models with the mixing 
property (Definition \ref{def:mixprop})
are sheaves (with respect to the supremum Grothendieck topology), and conversely (Proposition \ref{chr:mix}). 
Our result is the special case for Boolean algebras of more general results by Fourman, Scott \cite{FOURSCOTT}, Monro \cite{MONRO86, MONROBIS}, and others, which prove the existence of an equivalence of categories between Heyting valued models and presheaves defined on Heyting algebras. Our presentation differs from the the ones findable in the well established literature on the topic for two main reason. First, our definition of Boolean valued model is different from the one natural arising in the study of presheaves on Heyting/Boolean algebras, since it requires the reflexivity of the equality (Remark \ref{rem:transeq}). This choice is grounded in the use of Boolean valued models in set theory (in particular in the forcing method), but it does not imply a plain equivalence with presheaves on a Boolean algebra, thus requiring a careful description of the categories involved in the correspondence. The second reason is that our direct proof in terms of Boolean valued models has the advantage of being accessible to
set theorists as well as to category theorists. 

\item \label{diff:sheaf}
We introduce a topological construction describing the sheafification process for presheaves on a complete Boolean algebra $\bool{B}$.
More precisely, 
we associate to every presheaf $\F$ on $\bool{B}$ a bundle $\Lambda_\F$ over the Stone space $\St(\bool{B})$ (Definition \ref{def:etbund1}, Lemma \ref{lem:JBundef}). We first show that the separated quotient of $\F$ is represented by a presheaf of certain local sections on $\Lambda_\F$ (Propositions \ref{prop:univJsep} and \ref{lem:Jsepbindles1}), then we prove that the sheafification of $\F$ can be represented as the sheaf given by suitably defined Stone sections of the bundle $\Lambda_\F$ (Theorems \ref{mainthm:sheaf} and \ref{2mainthm:sheaf}). 

\item 
We relate the notion of open continuous mapping between topological spaces, to that of complete homomorphism between complete Boolean algebras, and to that of 
adjoint homomorphism between Boolean algebras (see 
Theorem \ref{thm:adj-dual-cbaTDCH} and Propositions \ref{prop:openadjpairs} \ref{thm:openmap-->adjoint} in Section \ref{sec:dual}). 
We also connect our constructions to 
standard completion processes performed in functional analysis/general topology/set theory (Examples \ref{ex:sheaf}, \ref{exm:sheafLR}).

\item \label{full:mix} We investigate the basic features of Boolean valued models and their semantics, separating the mixing property and the fullness property (Definition \ref{def:fullness-wellbehaved}); the latter
identifies those $\bool{B}$-valued models for which \L o\'s Theorem ---equivalently  the Forcing Theorem--- holds. Specifically we show that every model with the mixing property is full (Proposition \ref{ful:mod}), that this implication is strict 
(Example \ref{non:mix}), 
and that not all 
Boolean valued models are full.
We also find a simple topological property characterizing the fullness property (Theorem \ref{thm:charLosThm}).

\item \label{etal:sp}
With the aim of clarifying further the distinction between fullness and mixing property,
we characterize the fullness property elaborating more on our correspondence between $\bool{B}$-valued models and presheaves, 
and taking into account the notion of \'etal\'e space (Theorem \ref{chr:ful}).

\end{enumerate}


\section*{Forcing, Boolean valued models, and sheaves}
Boolean valued models constitute one of the classical approaches to forcing \cite{BELL,JECH}. 
They also give a useful method for producing new first order models of a theory; for example the standard ultrapower construction is 
just a special instance of a more general procedure to construct Boolean valued elementary extensions of a given structure 
(see \cite{VIAUSAX} or \cite{MANS71}). 
Many interesting spaces of functions (for example $L^\infty(\mathbb{R}), C(X)$ etc.) can be naturally endowed of 
the structure of a $\bool{B}$-valued model; for example set $\Qp{f=g}=\qp{\bp{x\in \mathbb{R}: f(x)=g(x)}}_{\bool{MALG}}$ 
for $f,g\in L^\infty(\mathbb{R})$, where $\bool{MALG}$ is the complete Boolean algebra given by 
Lebesgue measurable sets modulo Lebesgue null sets; $\Qp{f=g}$ gives a natural measure of the degree of equality 
subsisting between $f$ and $g$. In particular by means of Boolean valued models one can expect to employ logical methods to analyze 
the properties of these function spaces. This can be fruitful, see for example \cite{VIASCHANUEL} or 
\cite{MR779056,MR759416,VIAVAC15}.\\
The aim of this paper is to systematize these results and connect them to sheaf theory. Indeed, the description of the forcing method in terms of Boolean valued models provides a framework for investigating the forcing machinery within category-theoretic semantics for logics, starting from the seminal work \cite{LAWTIE} by Lawvere and Tierney relating forcing to Grothendieck toposes. The strict correlation between Heyting valued semantics and sheaves for Grothendieck topologies is well established in the literature: see in particular Fourman--Scott's work \cite{FOURSCOTT} and references therein. It turns out that sheaf theory provides an approach to the forcing method equivalent to the one with Heyting valued models (an application can be found in \cite{Four}). The correspondence between Boolean valued models and certain types of presheaves is a corollary of that general theory, and was deeply investigated by Monro \cite{MONROBIS, MONRO86}. It has its own relevance due to the role Boolean valued models have in set theory, but also because the difference between Heyting valued semantics and Boolean valued semantics is exactly the diffeence between intuitionistic logic and classical logic. In the works mentioned above, the notion of Boolean valued models which naturally arise from the study of presheaves on Boolean algebras lacks one property which is always required when defining Boolean valued models for set theory, that is: for every element $\tau$ in the model, $\Qp{\tau=\tau}=1$. Adding this condition in the definition of Boolean valued models changes the family of presheaves to which they correspond. The first step of our work is then providing a direct proof in terms of Boolean valued models of the Boolean case of Forman--Scott's result, as formulated by Monro\footnote{We note that we became aware of Monro's results only after preparing the first version of this paper. We thank Greta Coraglia for bringing Monro's work to our attention.} in \cite[Proposition 2.3]{MONROBIS} (contribution \ref{equiv:boo} above), adapting it to the notion of Boolean valued models of interest in set theory. Then, we expand and develop this result (contributions \ref{diff:sheaf}, \ref{full:mix}, \ref{etal:sp}  above). Our long-term hope is that this is a first step towards a fruitful transfer of techniques, ideas and results arising in set theory in the analysis of the forcing method to other domains where a sheaf-theoretic approach to problems is useful.\\
This paper (while written by scholars whose mathematical background is rooted in set theory) is aimed primarily at readers with some expertise in category theory and familiarity with first order logic. We emphasize that
\emph{no knowledge or familiarity with the forcing method as presented in \cite{BELL,JECH,KUNEN} is required} to follow the proofs and statements of the main results. 
We will sporadically relate our results to the forcing machinery through some examples (see in particular Examples
\ref{ex:fullnotimplmix}, \ref{ex:fullnotimplmix2}, \ref{exm:sheafLR}) and a few comments, 
but that's all. The reader unfamiliar with the forcing method can safely skip them without compromising the comprehension of the main body of this article. \\
Let us spell out more precisely the outcomes of \cite{MR779056,MR759416,VIAVAC15,VIASCHANUEL} as it is related to what we will do here.
\cite{MR779056,MR759416,VIAVAC15} show 
that there is a ``natural'' identification between:
\begin{itemize}
\item the unit ball of commutative real $C^*$-algebras of the form $C(X)$ with 
$X$ compact Hausdorff extremally disconnected;
\item 
the $\RO(X)$-names for the unit interval
$[0,1]$ as computed in the forcing extension $V^{\RO(X)}$ of $V$ given by $\RO(X)$ (where $\RO(X)$ denotes the complete Boolean algebra of regular open subsets of $X$). 
\end{itemize}
In 
\cite{VIASCHANUEL} these results are applied to establish a weak form of Schanuel's conjecture on the transcendence 
property of the exponential function over the complex numbers.
One of the outcomes of the present paper combined with \cite{VIAVAC15}
for example is that the sheafification (according to the supremum Grothendieck topology on the positive elements of $\bool{MALG}$) of the presheaf given by the essentially bounded measurable functions defined on some
measurable subset of $[0,1]$ is exactly given by the sheaf associated to the Boolean valued model describing the real 
numbers in $[0,1]$ which exist in $V^{\bool{MALG}}$ (see Example \ref{exm:sheafLR}).
These results are just samples of the various roles Boolean valued models can play across different mathematical fields, 
and why we strongly believe it can be useful to develop a flexible translation tool to relate concepts expressed in rather 
different terminology in different set-ups. This is more or less all we will say on the relation between the results of the present paper and forcing.

\section*{Structure of the paper}
The paper is organized as follows:
\begin{itemize}
\item Sections \ref{sec:prereqbatopsppreord} and \ref{sec:grothtoppreshetspbun} give prerequisites on topology, Boolean algebras, partial orders, Grothendieck topologies, (pre)sheaves, (\'etal\'e) bundles. 

\item Section \ref{sec:dual} explores the connections/dualities between the notion of open continuous map, that of complete homomorphism between complete Boolean algebras, and that of adjoint homomorphism between Boolean algebras.

\item Section \ref{sec:boolvalmod} gives a rapid presentation of the key properties of Boolean valued models.

\item Section \ref{sec:pres} establishes a functorial correspondence between Boolean valued models and presheaves on complete Boolean algebras, which specializes to a duality between Boolean valued models with the mixing property and sheaves for the supremum (equivalently, dense) Grothendieck topology.

It also gives a categorical characterization of the fullness property for Boolean valued models in terms of properties of the bundles associated to their presheaf structure.

\item Section \ref{sec:sheafposet} presents a topological construction via \'etal\'e bundles of the sheafification operator with respect to the supremum Grothendieck topology for presheaves on a complete Boolean algebra.
\end{itemize}

\tableofcontents

\section{Preliminaries on Boolean algebras, partial orders, topological spaces, compactifications}\label{sec:prereqbatopsppreord}

In this section we gather a number of folklore results on Boolean algebras, partial orders, topological space, sand compactifications, which we will use freely in the remainder of the paper. We include proofs of those non-trivial results for which we are not able to trace a reference in the literature. For non-trivial results whose proof is omitted we give appropriate references.

\subsection{Basics on Boolean algebras, partial orders, topological spaces}
We recall some basic facts and terminology about Boolean algebras, Heyting algebras, partial orders, compact spaces. The missing proofs can be found in \cite{JECH, JohnStone, viale-bookonforcing}.

Given a topological space $(X,\tau)$
and $A$ subset of $X$, $\Reg{A}$ is the interior of the closure of $A$ in $(X,\tau)$. 
$A$ is \emph{regular open} if $A=\Int{\Cl{A}}=\Reg{A}$.  
$\RO(X,\tau)$ (in short $\RO(X)$ if $\tau$ is clear from the context) denotes the family of regular open subsets of $(X,\tau)$.
$\CLOP(X,\tau)$ (or just $\CLOP(X)$ if $\tau$ is clear from the context) denotes the family of clopen subsets of $(X,\tau)$.
Clearly $\CLOP(X)\subseteq\RO(X)$.

By the Stone Representation Theorem (see for instance \cite[Theorem 3.8.2]{viale-bookonforcing} or \cite[Theorem 7.11]{JECH}), every Boolean algebra $\bool{B}$ is isomorphic to the Boolean algebra 
$\CLOP(\St(\bool{B}))$ of the clopen subsets of its \emph{Stone space}
$\St(\bool{B})$.
The latter is the compact Hausdorff $0$-dimensional space
whose points are the ultrafilters on $\bool{B}$, topologized by the topology $\tau_\bool{B}$ obtained by taking as a base the family of sets of the form
$N_b:=\bp{G\in\St(\bool{B}): b\in G}$ for some $b\in\bool{B}$. These sets are actually the clopen sets of $\St(\bool{B})$.
\begin{notation}\label{not:stonespaces}
For the remainder of this paper, when dealing with Stone spaces $X$ (i.e. $0$-dimensional compact Hausdorff spaces) it is convenient to look at points of $X$ at times as elements of their clopen neighborhoods and at times as ultrafilters on $\CLOP(X)$ as the Stone duality naturally identifies each point of $X$ to the ultrafilter on $\CLOP(X)$ given by its clopen neighborhoods. Thus we will write $G\in U$ for $U\in\CLOP(X)$ and $G\in \St(\CLOP(X))$ if we see $G$ as a point in $X$, and $U\in G$ if we see $G$ as an ultrafilter on $\CLOP(X)$.
\end{notation}

Given a topological space $(X,\tau)$,
$\RO(X)$ is endowed with the structure of a complete Boolean algebra by letting
\begin{itemize}
\item
$0_{\RO(X)}:=\emptyset$, $1_{\RO(X)}:=X$,
\item 
 for $U, V\in\RO(X)$
\begin{gather*}
U\vee V:=\Reg{U\cup V}, \\
U\wedge V:=U\cap V, \\
\neg U:=X\setminus\Cl{U}.
\end{gather*}
\item
for any family $\bp{U_i: i\in I}\subseteq\RO(X)$
\begin{gather*}
\bigvee_{i\in I}U_i:=\Reg{\bigcup_{i\in I}U_i}, \\
\bigwedge_{i\in I}U_i:=\Reg{\bigcap_{i\in I}U_i}.
\end{gather*}
\end{itemize}
\begin{fact}
\label{comp:bool}
A Boolean algebra $\bool{B}$ is complete if and only if $\CLOP(\St(\bool{B}))=\RO(\St(\bool{B}))$ (i.e. if and only if the space $\St(\bool{B})$ is extremally disconnected).
\end{fact}
See \cite[Prop. 4.1.5]{viale-bookonforcing} for a proof.
\begin{remark}
\label{rem:supisdense}
Let $\bool{B}$ be a Boolean algebra. If $A\subseteq\bool{B}$, let $\bp{N_a: a\in A}$ be the family of basic open sets in the Stone space $\St(\bool{B})$ corresponding to $A$. Assume $\bigvee A$ exists. Then $\bigcup_{a\in A}N_a$ is a dense open set in a clopen set $N_b$ if and only if $b=\bigvee_{\bool{B}} A$. Observe that, if $A$ is infinite, in general it can be that $\bigcup_{a\in A}N_a\subsetneq N_{\bigvee A}$.
In case $\bool{B}$ is complete, 
\[\Reg{\bigcup_{a\in A}N_a}=N_{\bigvee A}.\]
\end{remark}

Finally recall that a frame is a bounded distributive lattice which is upward complete and satisfies the distributive law $q\wedge\bigvee_{i\in I}p_i=\bigwedge_{i\in I}(q \wedge p_i)$
(see \cite[Chapter IX]{MacLMoer} for details). Examples of frames of interest to us are given by the topology $\tau$ of a topological space $(X,\tau)$ and by complete boolean algebras. Given a frame 
$(P,\wedge,\vee,0_P,1_P)$ we denote by $P^+$ the partial order $P\setminus \bp{0_P}$.

Given a pre-order\footnote{A pre-order $\leq$ on $P$ is a reflexive and transitive binary relation; a pre-order $\leq$ is a partial order if it is also antisymmetric.} $(P, \leq)$, we endow it with the downward topology $\tau_P$:
given $X\subseteq P$,
\[
\downarrow X:=\{p\in P: \textit{ there exists } x\in X \textit{ such that } p\leq x\}.
\] 
$X\subseteq P$ is a \emph{down-set} if $X=\downarrow X$. \\
The family $\down(P)$ of the down-sets of $P$ is a topology for $P$, the \emph{downward topology} $\tau_P$.
\begin{itemize}
\item
A \emph{filter} $F$ on a preorder $P$ is a downward directed and upward closed non-empty subset of $P$.
\item
A \emph{prefilter} is a subset $Q$ of $P$ such that every finite family $q_1, \dots, q_n$ of elements in $Q$ has a \emph{refinement}, i.e. a $p\in P$ such that $p\leq q_i$ for $i=1,\dots n$.
\item $X\subseteq P$ is \emph{dense below $p$} if for all $q\leq p$ there is $r\in X$ refining $q$.
\item
$X\subseteq P$ is\emph{ predense below some $p\in P$} if $\downarrow X$ is dense below $p$.
\item
$X\subseteq P$ is \emph{(pre)dense} if it is (pre)dense below all $p\in P$.
\end{itemize}
\begin{definition}
\label{def:emb}
 A \emph{morphism of pre-orders} $f:P\to Q$ is an order preserving map, i.e. if $p_1\leq p_2$ in $P$ then $f(p_1)\leq f(p_2)$ in $Q$. A morphism $f:P\to Q$ is:
\begin{itemize}
\item \emph{incompatibility-preserving} when, if $p_1\bot p_2$ (i.e. there is no $s\in P$ such that $s\leq p_1$ and $s\leq p_2$), then $f(p_1)\bot f(p_2)$;
\item an \emph{embedding} when $p_1\leq p_2$ in $P$ if and only if $f(p_1)\leq f(p_2)$ in $Q$;
\item \emph{dense} if $\ran(f)$ is dense in $Q$.
\end{itemize}
\end{definition}
Note that an incompatibility-preserving morphism need not to be injective: for example, $U\mapsto\Reg{U}$ defines a non-injective (and surjective) incompatibility-preserving morphism of $\tau$ into $\RO(X,\tau)$ provided $\tau\supsetneq\RO(X,\tau)$.
\begin{definition}
A pre-order $P$ is:
\begin{itemize}
\item 
\emph{separative} if, for every $p, q\in P$, $p\nleq q$ implies that there exists $r\in P$ such that $r\leq p$ and $r\bot q$;
\item 
\emph{atomless} if any $p\in P$ can be refined by two incompatible conditions;
\item
\emph{upward complete} if it admits suprema for all its non-empty subsets.
\end{itemize}
\end{definition}
\begin{remark}
Let $P, Q$ be separative preorders. If a morphism $P\to Q$ is incompatibility-preserving, then it is an embedding. The converse holds if $f$ is dense.
\end{remark}

Observe the following:
\begin{itemize}
\item Any separative pre-order $P$ does not have a minimum (unless it has only one element), and in case it is an upward complete partial order the following holds: for any $p\in P$ and $A\subseteq P$, $\bigvee_P(\downarrow A\cap \downarrow p)=p$ if and only if 
$\downarrow A\cap\downarrow p$ is a dense open subset of $\downarrow p$.
\item Given $(X,\tau)$ a topological space, $\tau$ is an upward complete partial order.
\item If $\bool{B}$ is a Boolean algebra, then $\bool{B}^+:=\bool{B}\setminus\{0_\bool{B}\}$ with the induced order is a separative partially ordered set, and is atomless if $\bool{B}$ has no atoms.
\item A Boolean algebra $\bool{B}$ is complete if and only if $\bool{B}^+$ is an upward complete separative pre-order.
\item If $P$ is a pre-order, we can always surject it onto a separative pre-order via an incompatibility-preserving morphism.
\end{itemize}
\begin{definition}
The \emph{Boolean completion} of a partial order $(P,\leq)$ is a complete Boolean algebra $\bool{B}$ such that
there exists a dense incompatibility-preserving morphism $e:P\to \bool{B}^+$.\\
 The \emph{Boolean completion} of a Boolean algebra $\bool{B}$ is the Boolean completion of the separative order $\bool{B}^+$.
\end{definition}
\begin{theorem}
\label{boo:com}
If $P$ is a partial order, then $\RO(P,\tau_P)$ 
is its Boolean completion, as witnessed by the map 
\[
\begin{aligned}
e:&P\to\RO(P).\\
&p\mapsto \Reg{\downarrow p}.
\end{aligned}
\]
Moreover the Boolean completion of $P$ is unique up to isomorphism.
\end{theorem}
\begin{proof}
See the proof of \cite[Theorem 14.10]{JECH} or \cite[Theorem 4.2.4]{viale-bookonforcing}.
\end{proof}
\begin{remark}
If $P$ is separative, $\downarrow p$ is already regular open for every $p\in P$. Moreover, in this case the map $e:p\mapsto\downarrow p$ is injective.
\end{remark}
\begin{fact}
\label{alg:com}
Every Boolean algebra $\bool{B}$
can be densely embedded in the complete Boolean algebra
$\RO(\St(\bool{B}))$
via the Stone duality map
\[
b\mapsto N_b=\bp{G\in\St(\bool{B}):\, b\in G}
\]
which identifies $\bool{B}$ with $\CLOP(\St(\bool{B}))$.
The image is dense because the clopen sets form a base for the topology on
$\St(\bool{B})$. \\
Moreover $\RO(\bool{B}^+,\tau_{\bool{B}^+})\cong\RO(\St(\bool{B}))$
via the unique extension of the map defined on 
$\bool{B}^+$ by $b\mapsto N_b$. 
\end{fact}

We also recall that a topological space is $T_1$ if any two distinct points can be separated by open sets which contain exactly one of them; it is extremally disconnected if the closure of an open set is open, it is
$0$-dimensional (or totally disconnected) if the clopen sets form a base for the space.

\begin{fact}
Let $(X, \tau)$ be a $T_1$ topological space. Then:
\begin{itemize}
\item if $X$ is $0$-dimensional, it is Hausdorff;
\item if $X$ is Hausdorff, extremally disconnected, it is $0$-dimensional;
\item $X$ is extremally disconnected Hausdorff if and only if $\CLOP(X)=\RO(X)$.
\end{itemize}
\end{fact}

We also need the following characterization of  of the regularization operation (see \cite[Lemma 2.1.4]{viale-bookonforcing}):
\begin{lemma}\label{lem:charreg}
Let $(X,\tau)$ be a topological space and $A\subseteq X$. Then its regularization $\Reg{A}$ is the set
\[
\bp{a\in X: \, \exists\, U\text{ open neighborhood of }a \text{ such that }A\cap U\text{ is dense in }U}.
\]
\end{lemma}

\subsection{Hausdorff compactifications}

We now recall some basic facts about Hausdorff compactifications. 

\begin{definition}
A triple $\cp{Y, \sigma,\iota}$  is a Hausdorff compactification of some topological space $(X,\tau)$ if 
$(Y,\sigma)$ is a compact Hausdorff space and $\iota:X\to Y$ is a topological embedding with a dense image.

For $(X, \tau)$ a $T_0$-topological space, its Stone-\v{C}ech compactification (when it exists) is the unique Hausdorff compactification $\cp{Y,\sigma,\iota}$ of $X$
that has the following universal property: any continuous map $f:X\to K$ where $K$ is compact Hausdorff extends uniquely to a continuous map $\beta(f):Y\to K$ such that $\beta(f)\circ\iota=f$.
\end{definition}

An useful compactification exists for locally compact topological space:
\begin{definition}
\label{def:onepointcomp}
Let $(X,\tau)$ be a topological space. The space $\beta_0(X)=X\cup\bp{\infty}$ has the topology $\beta_0(\tau)$ generated by 
\[
\tau\cup \bp{\beta_0(X)\setminus A: A\in\tau,\, \Cl{A}\text{ compact}}.
\]
\end{definition}
$(\beta_0(X),\beta_0(\tau))$ is always compact. Furthermore:
\begin{itemize}
\item
If $X$ is $T_1$ so is $\beta_0(X)$.
\item
If $X$ is locally compact Hausdorff, $\beta_0(X)$ is Hausdorff.
\end{itemize}

We now address the Stone-\v{C}ech compactification, which exists for a very large class of topological spaces $(X, \tau)$.

\begin{definition}
A topological space $(X,\tau)$ is \emph{Tychonoff} if for any $x\in X$ and $C$ closed subset of $X\setminus\bp{x}$ there is a continuous $f:X\to\mathbb{R}$ which vanishes on $x$ and gets value $1$ constantly on $C$.
\end{definition}

\begin{remark}
Every locally compact Hausdorff space is Tychonoff. Furthermore the Stone-\v{C}ech Compactification Theorem characterizes the spaces which are topologically homeomorphic to subsets of a compact Hausdorff space as the class of Tychonoff spaces. 
\end{remark}

We spell out the details of one possible presentation of the Stone-\v{C}ech compactification.

\begin{notation}
Let $(X,\tau)$ be a topological space. $C\subseteq X$ is a \emph{$0$-set} if there is a continuous
 $f:X\to\mathbb{R}$ such that $C=f^{-1}[\bp{0}]$.\\ 
 We let $\mathcal{O}_0(X)$ denote the open sets which are complements of $0$-sets.
 \end{notation}
 
 \begin{remark}
Note that, for any topological space, the family of its $0$-sets is closed under finite unions and clopen sets are $0$-sets.
Also note that elements of $\mathcal{O}_0(X)$ can be described as the preimage of $[0;a)$
 for some $a<1$ and some continuous $f:X\to [0,1]$. \\
When $(X,\tau)$ is Tychonoff any closed set is the intersection of $0$-sets and also the intersection of open sets which are the complement of $0$-sets.
Hence in a Tychonoff space $\mathcal{O}_0(X)$ forms a base of open sets closed under finite intersections.
\end{remark}

\begin{definition}\label{def:stcechcomp}
Given a topological space $(X,\tau)$, $\beta(X)$ is the set whose elements are the maximal filters of $0$-sets. The topology $\beta(\tau)$ on $\beta(X)$ is generated by 
$\beta(U)=\bp{F\in\beta(X): \,\exists C\in F\, C\subseteq U}$, for $U\in\mathcal{O}_0(X)$.

$\iota_X:X\to\beta(X)$ maps any $x\in X$ to the maximal filter given by $0$-sets containing $x$.
\end{definition}

\begin{remark}\label{rmk:stcechcomp}
The following is well known (see for instance \cite[Appendix D.3]{viale-bookonforcing}) for a topological space $(X,\tau)$:

\begin{itemize}
\item
$\iota_X$ is well-defined and continuous;
\item any continuous $f:X\to K$ with $(K,\sigma)$ a compact Hausdorff space admits a unique continuous extension $\beta(f):\beta(X)\to K$ such that $f=\beta(f)\circ \iota_X$;
\item
$\iota_X$ is injective provided the (singletons of) points of $X$ are the intersection of $0$-sets, and a local homeomorphism exactly when $X$ is Tychonoff;  
\item $\iota_X[X]$ is open in $\beta(X)$ if and only if $X$ is locally compact and Hausdorff.
\end{itemize}
\end{remark}

 The following follows almost immediately from the above observations:
\begin{fact}\label{fac:StCe-iso}
Assume $(X,\tau)$ is Tychonoff. Then $\RO(X,\tau)\cong\RO(\beta(X),\beta(\tau))$ via the map $U\mapsto U\cap X$ for $U\in \RO(\beta(X),\beta(\tau))$.
\end{fact}

\begin{proof}
W.l.o.g. since $\iota_X$ is injective we can assume that $X\subseteq\beta(X)$ and $\iota_X$ is the identity.
We first show that for an open set $U\in\tau$ 
\begin{equation}\label{eqn:keyeqXbetaX}
\mathrm{Reg}_X(U)=\mathrm{Reg}_{\beta(X)}(U)\cap X,
\end{equation} where
$\mathrm{Reg}_Z(U)$ (respectively $\mathrm{Cl}_Z(U)$, $\mathrm{Int}_Z(U)$) denotes the regularization (respectively closure, interior) of $U$ in the topological space $Z$ for $Z$ one among $X,\beta(X)$.

Equation \ref{eqn:keyeqXbetaX} entails that the map $U\mapsto U\cap X$ for $U\in\RO(\beta(X))$ is bijective: it is clearly surjective in view of \ref{eqn:keyeqXbetaX}, and it is also injective, since  any two regular open which have dense intersection coincide (see for example \cite[Fact 2.1.7]{viale-bookonforcing}.

We prove \ref{eqn:keyeqXbetaX}:
note that the family $\beta(Z)$ for $Z$ a $0$-set of $(X,\tau)$ forms a basis of closed sets for $(\beta(X),\beta(\tau))$ and that $\beta(Z)\cap X=Z$ for all such $Z$. Consequentlyfor all $U\subseteq X$:
\begin{align*}
\mathrm{Reg}_{\beta(X)}(U)\cap X&= \mathrm{Int}_{\beta(X)}(\mathrm{Cl}_{\beta(X)}(U))\cap X=\\
&=(\bigcup\bp{W\in\beta(\tau): W\subseteq \mathrm{Cl}_{\beta(X)}(U)})\cap X=\\
&=\bigcup\bp{W\cap X:W\in \beta(\tau)\text{ and } W\subseteq \mathrm{Cl}_{\beta(X)}(U)}=\\
&=\bigcup\bp{W\in\tau: W\subseteq \bigcap \bp{\beta(V): V\text{ is a $0$-set in } X \text{ and }U\subseteq V}}=\\
&=\bigcup\bp{W\in\tau: W\subseteq \bigcap \bp{\beta(V)\cap X: V\text{ is a $0$-set in } X \text{ and }U\subseteq V}}=\\
&=\bigcup\bp{W\in\tau: W\subseteq \bigcap \bp{V: V\text{ is a $0$-set in } X \text{ and }U\subseteq V}}=\\
&=\mathrm{Int}_{X}(\mathrm{Cl}_{X}(U))=\\
&=\mathrm{Reg}_{X}(U),
\end{align*}
where the equality between the third and fourth line follows from the fact that the topology $\tau$ on $X$ is the subspace topology inherited by $X$ as a subspace of $(\beta(X),\beta(\tau))$ (since $X$ is Tychonoff),
while the one between the fifth and the sixth line follows from the fact that $\beta(V)\cap X=V$ for all $0$-sets $V$ for $(X,\tau)$; finally the one between the sixth and the seventh line holds because the $0$-sets form a base for the closed sets of $(X,\tau)$.

%
%
%
%
%
%
%
\end{proof}

\begin{corollary}
\label{corol:betazextrdisc}
Assume $(X,\tau)$ is Tychonoff and extremally disconnected. Then so is 
$(\beta(X),\beta(\tau))$.
\end{corollary}
 \begin{proof}
We have just to prove that $(\beta(X),\beta(\tau))$ is extremally disconnected. This amounts to show that any regular open set in $\beta(\tau)$ is closed. In view of the previous Fact the regular open sets of $(\beta(X),\beta(\tau))$ are of the form $\beta(U)$ for $U\in \RO(X,\tau)$;
now, since $(X,\tau)$ is extremally disconnected, any such $U$ is also clopen for $\tau$;
it can then be checked that $\beta(U)$ is clopen for $\beta(\tau)$ as well (see for details \cite[Proposition D.3.10(7)]{viale-bookonforcing}).
\end{proof}
 \subsection{Maximal filters of closed or open sets and Hausdorff compactifications}
 We now want to relate the points of Hausdorff compact spaces to (maximal) filters of topological sets. This is standard in the literature. However, we prefer to discuss in detail the delicate topological points we want the reader to focus on.
%
%
The guiding example of filter on a partial order $P$ is given by the family $G_x$ of open neighborhoods of a point $x\in X$  for $(X,\tau)$ a topological space and $(P,\leq)$ the partial order $\mathcal{O}(X)^+$ given by non-empty open sets of $(X,\tau)$. 
This filter in general is not a maximal filter (consider the interval $[0,1]$ with the usual euclidean topology; the filter of open neighborhoods of $1/2$ is not maximal; it can be extended to a maximal filter containing $[0;1/2)$ and to another maximal filter containing $(1/2;1]$).
Moreover if $\mathcal{A}$ is a base of non-empty open sets for $X$, $G_x\cap\mathcal{A}$
is also a filter on the partial order $(\mathcal{A},\subseteq)$, but in general a maximal filter $G$ on 
$\mathcal{A}$
may not identify a point of $X$: the intersection of the open sets in $G$ (or even of the closure of the open sets in $G$) may be empty (consider on $\mathbb{R}$ with the euclidean topology a maximal filter of open sets extending 
$\bp{(a,+\infty):a\in\mathbb{R}}$). This identification of maximal filters of open sets to points of $X$ 
is clearly related to the existence of compactifications of $X$.
If $X$ is compact Hausdorff any maximal filter $G$ on $(\mathcal{A},\subseteq)$ determines a unique point 
$x$ of $X$ with $G\supseteq G_x\cap\mathcal{A}$, this point $x$ is the unique
element of the closed set $\bigcap\bp{\Cl{U}:U\in G}$. 
However different maximal filters of open (or closed) sets
could determine the same point of $X$ (as shown by the case outlined above for $1/2$ in $[0,1]$).

We now compare various possible ways to describe points of some compactification of 
a topological space. Typically we can identify a point of a compact Hausdorff space by means of:
\begin{itemize}
\item its open neighborhoods;
\item its regular open neighborhoods;
\item its open neighborhoods which are complements of $0$-sets;
\item the family of closed sets to which it belongs;
\item the family of $0$-sets to which it belongs;
\item the family of clopen sets to which it belongs (in case the space is $0$-dimensional).
\end{itemize}
Note that each of the above families determines a (maximal?) filter in the corresponding partial order under inclusion.
Hence given any topological space $(X,\tau)$ it is tempting to try to construct a Hausdorff compactification of $X$ whose points should be the maximal filters of one of the above families with respect to $\tau$. This is exactly what occurs
in the case of the above construction of the Stone-\v{C}ech compactification.

We are interested in comparing the above types of maximal filters and understand when these maximal filters define the same ``compactification''.
We now show that:
\begin{itemize} 
\item\textbf{For any $(X,\tau)$, maximal filters on $\RO(X)$ naturally correspond to maximal filters of open sets:}
 assume towards a contradiction that some maximal filter $F$ on $\RO(X)$ admits two incompatible extensions $F_0,F_1$ to maximal filters on $\mathcal{O}(X)^+$.
If any elements of $F_0$ intersects any elements of $F_1$, $F_0\cup F_1$ is a prefilter strictly containing both of them. Hence there are $U_i\in F_i$ such that $U_0\cap U_1$ is empty.
This occurs if and only if $\Reg{U_0}\cap\Reg{U_1}$ is empty.
Now $U_i\cap U$ is non-empty for all $U\in F$ and
this occurs if and only if $\Reg{U_i}\cap U$ is non-empty for all $U\in F$. Since $F$ is maximal this can be the case only if $\Reg{U_i}\in  F$ for both $i$, which contradicts $F$ being a prefilter.\\
We note that for $G$ ultrafilter on $\RO(X)$ the unique maximal filter of open sets $\bar{G}$ is given by 
$\bp{U:\Reg{U}\in G}$.
\item\textbf{If $(X,\tau)$ is compact Hausdorff, maximal filters of closed sets correspond to maximal filters of $0$-sets:}
Any closed set $C$ is an intersection of $0$-sets since $(X,\tau)$ is Tychonoff.
Now if $\F$ is a maximal filter of closed sets, let $\F^*$ be the intersection of $\F$ with the $0$-sets.
It suffices to argue that $\F^*$ is a maximal filter of $0$-sets. If not there is a $0$-set $D$ such that
$D\cap E$ is non-empty for all $0$-sets in $\F$ and $D\cap C$ is empty for some closed and non-empty 
non-$0$-set in $\F$. Now if $C=\bigcap_{i\in I}C_i$ with each $C_i$ a $0$-set, we get that each $C_i$ is in $\F$. Hence we have on the one hand that
\(
\emptyset=D\cap C=\bigcap_{i\in I} D\cap C_i,
\)
on the other hand that $\bp{D\cap C_i:i\in I}$ is a family of closed sets with the finite intersection property. Since $X$ is compact the latter implies that $D\cap C$ is non-empty. We reached a contradiction.

Conversely if $\G$ is a maximal filter of $0$-sets, $\bigcap\G$ is a singleton $\bp{x_\G}$
(we use here that the space is compact Hausdorff).
Therefore $\G^*$ is the intersection with the $0$-sets
of the maximal filter of closed sets given by closed sets containing $x_\G$.
Hence we have a bijective correspondence between maximal filters of closed sets and maximal filters of $0$-sets.
%
\item\textbf{If $(X,\tau)$ is extremally disconnected, maximal filters on open sets correspond to maximal filters on $\mathcal{O}_0(X)^+$:} 
In extremally disconnected spaces regular open sets are closed, hence they are $0$-sets, and the complement of $0$-sets. This gives that $\mathcal{O}_0(X)\supseteq\RO(X)$. Hence any maximal filter on $\mathcal{O}_0(X)^+$ induces a maximal filter on $\RO(X)$. 
By the first item these overlaps with maximal filters of open sets.
We also note that whenever $\F$ is a maximal filter of open sets, $\F\cap \mathcal{O}_0(X)^+$ is a
maximal filter on $\mathcal{O}_0(X)^+$: otherwise $\F\cap \RO(X)$ could have two incompatible extensions to a maximal filter on $\mathcal{O}_0(X)^+$, and therefore also to $\mathcal{O}(X)^+$. We already saw that the latter is impossible.
\end{itemize}
Note that if $(X,\tau)$ is extremally disconnected but not Hausdorff, the filters of (regular) open neighborhoods of points of $X$ might not be maximal filters of (regular) opens\footnote{Consider the case of Zariski topology on $\mathbb{R}$: it is extremally disconnected as the closure of any open set is $\mathbb{R}$ which is clopen; there is a unique maximal filter given by the family of all open non-empty sets; $\mathbb{R}$ is the unique regular open of the topology; the filters of open neighborhoods of points are not maximal.}. We have already seen that in compact Hausdorff spaces $X$ maximal filters of closed sets determine points of $X$, hence they may not determine a maximal filter of open sets (in case $X$ is not extremally disconnected).

Wrapping everything together we get the following fact (see for instance \cite[Sections II.4.9 and IV.2.3]{JohnStone} for partial versions of it):
\begin{proposition}
Assume $(X,\tau)$ is extremally disconnected, compact, Hausdorff. Then there are natural identifications of the families of:
\begin{itemize}
\item maximal filters on $\RO(X)$;
\item maximal filters on $\CLOP(X)$;
\item maximal filters on $\mathcal{O}_0(X)^+$;
\item maximal filters of open sets;
\item maximal filters of closed sets;
\item maximal filters of $0$-sets.
\end{itemize}
\end{proposition}

It will be convenient in the next sections to focus on spaces $(X,\tau)$ which admit a base of regular open sets and are $T_0$.
The Tychonoff spaces have this property, as well as separative and atomless partial orders endowed with the downward topology (which are $T_0$, but not $T_1$).
Note that the latter spaces are not \emph{sober} (as defined in \cite[Section IX.3]{MacLMoer}). Our focus will be on topological spaces $(X,\tau)$ with $\tau$ arising from the downward topology induced by a separative atomless partial order on $X$ or $(X,\tau)$ compact Hausdorff. In particular our focus is on a class of frames orthogonal to the sober spaces as in \cite[Chapter IX]{MacLMoer}.

We conclude this section giving a characterization of extremally disconnected compact Hausdorff spaces which we will need in Section \ref{sec:sheafposet}.
 
\begin{proposition}
 \label{comp:exdisc}
 Let $(X,\tau)$ be a compact Hausdorff space. Then $X$ is extremally disconnected if and only if it is the Stone-\v{C}ech compactification of each of its dense subsets.
 \end{proposition}
 
 The Proposition follows by a combination of \cite[Proposition page 284 - Section 10.47, Theorem page 25 - Section 1.46]{RUSSSTCECOMP}; for a direct proof see \cite[Thm D.3.16]{viale-bookonforcing}.

\section{Preliminaries on Grothendieck topologies, (pre)sheaves, and (\'etal\'e) bundles}\label{sec:grothtoppreshetspbun}

In this section we recall the basic notions of sheaf theory we will need in the last two sections of the paper. The reference text for unexplained concepts is \cite{MacLMoer}.

\begin{notation}\label{not:globgrtop}
Fix a pre-order $(P,\leq)$ and a \emph{concrete} category\footnote{A concrete category is a category whose objects are sets equipped with some structure.}  $\mathcal{C}$.
\begin{itemize}
\item
A functor $\F:P^\mathrm{op}\to\mathcal{C}$ is a \emph{$\mathcal{C}$-valued presheaf on $P$}; we denote by $f\restriction c$ the application of 
$\F(c\leq b)$ to some $f\in \F(b)$.
\item
A \emph{matching family} $\bp{f_a:a\in A}$ on $\F$ for $A\subseteq P$ is a function in $\prod_{a\in A}\F(a)$
such that
$f_c\restriction b=f_a\restriction b$ whenever $a, c$ are both in $A$ and $b\leq a,c$.
\item
A \emph{gluing} $f$ for a matching family $\bp{f_a:a\in A}$ on $\F$  is an element of $\F(b)$ for some $b\in P$ which is an upper bound for $A$ such that $f\restriction a=f_a$ for all $a\in A$.
\item
We say that $\F$ is \emph{determined by its global sections} if $P$ has a top element $1_P$ and $\F(c\leq b)$ is surjective for all $c\leq b$ in $P$.
\end{itemize}
\end{notation}

\subsection{The sup and the dense Grothendieck topologies}
The following definition is a necessary recap for those readers familiar with set theory but not with the notion of Grothendieck topology\footnote{We refer to \cite[Chapters II and III]{MacLMoer} for a general introduction to Grothendieck topologies.}.

We start with the simple notion of \emph{topological} sheaf.
\begin{definition}\label{def:globgrtop0}
Let $(X,\tau)$ be a topological space.

A presheaf $\F:\tau^{\mathrm{op}}\to\mathcal{C}$ is a \emph{topological sheaf} if for any matching family $\bp{f_i:i\in I}$ with each $f_i\in  \F(U_i)$ there is a unique gluing $f\in \F(\bigcup_{i\in I}U_i)$.\footnote{For the basic theory of topological sheaves we refer to \cite[Chapter II]{MacLMoer}.}

\end{definition}

Now we generalize this concept to a larger family of partial orders, among which (the positive elements of) complete Boolean algebras. We introduce the dense Grothendieck topology and the sup Grothendieck topology on an (upward complete) preorder.\footnote{We need upward completeness to define properly the sup topology, it is not required for defining correctly the dense topology. We refer the reader to \cite[Chapter III, Section 2, Examples (d), (e) pag. 114-115]{MacLMoer} for details.}

\begin{definition}
Let $(P,\leq)$ be a preorder. 

\begin{itemize}

\item
The \emph{dense} topology on $P$ is the functor
$J^{\mathrm{den},P}:P^{\mathrm{op}}\to\dSet$ which maps any $p \in P$ to the family of dense open subsets of $\downarrow p$; the restriction map on $J^{\mathrm{den},P}(q\leq p)$ assigns $D\cap \downarrow q$ to any $D \in J^{\mathrm{den},P}(p)$.
\item Assume further that $P$ is upward complete (i.e. $P$ admits suprema for all its non-empty subsets). The \emph{sup} topology on $P$ is the functor
$J^{\mathrm{sup},P}:P^{\mathrm{op}}\to\dSet$ which maps any $p \in P$ to the family of subsets of $\downarrow p$ whose supremum in $P$ is exactly $p$. The restriction map for $J^{\mathrm{sup},P}$ is that of $J^{\mathrm{den},P}$.
\end{itemize}

\end{definition}

Note that the two Grothendieck topologies are similar but not overlapping. We will explore now their similarities and distinctions in some detail:

\begin{remark}
\emph{}

\begin{enumerate}
\item The readers familiar with forcing will have recognized the dense Grothendieck topology on a preorder $P$ as the central object of study when analyzing the effects of $P$ -seen as a forcing notion- on the semantics of the generic extensions of $V$ produced by it.
\item Let $(X,\tau)$ be  topological space. Then the sup topology on $\tau$ assigns to each $U\in\tau$ the family consisting of those sets $A\subseteq \pow{U}\cap\tau$ which are downward closed under inclusion (i.e. $V\in A$ and $W\subseteq V$ in $\tau$ entail $W\in A$ as well) and whose union covers $U$ (i.e. $\bigcup A=U$). In retrospect (see Def. \ref{def:globgrtop} below) the sup topology on $\tau$ is exactly the Grothendieck topology  used to define the notion of topological sheaf on $(X,\tau)$ given in Def. \ref{def:globgrtop0}.
\item
Assume $P=\tau^+$ for some topological space $(X,\tau)$. 
Then $(P,\subseteq)$ is upward complete, and:
\begin{itemize} 
\item
the supremum topology on $P$ assigns to each non-empty open set $p$ in $\tau$ the family of its open covers $A$ (i.e. such that $\bigcup A=p$), which are closed under downward inclusion (i.e. such that $q\in A$ and 
$\tau^+\ni r\subseteq q$ entail $r\in A$),
\item
while the dense topology on $P$ assigns to such an open set $p$ the family of its families
of "almost" covers $A\subseteq \tau^+$ which are downward closed and such that $\bigcup A$ is a dense open subset of $p$.
\end{itemize} 
In particular $J^{\mathrm{den},P}(p)\supseteq J^{\mathrm{sup},P}(p)$ for all $p\in P$ (with inclusion being strict in many cases: e.g. let $(X,\tau)$ be the real numbers with Euclidean topology and $P=\tau^+$; the downward closure of $(0;1/2)\cup(1/2;1)$ is in $J^{\mathrm{den},P}((0;1))$ but not in $J^{\mathrm{sup},P}((0;1))$).

\item
Assume on the other hand that $P=\bool{B}^+=\bool{B}\setminus\bp{0_\bool{B}}$ for  a complete Boolean algebra $\bool{B}$. Then $J^{\mathrm{den},P}$ and 
$J^{\mathrm{sup},P}$ coincide: a peculiar property of complete Boolean algebras is that a downward closed subset $A$ of $\downarrow p$ is dense if and only if $\bigvee_\bool{B} A=p$ (see for example \cite[Fact 3.11.7]{viale-bookonforcing}).  

\item
As outlined by the previous item, this is not the case when $P$ is a frame which is not a complete Boolean algebra.

\item
Finally assume $P$ is a frame. Then the dense topology on $P$ is trivial since any non-empty downward closed subset of $P$ has $0_P$ as an element and is therefore dense in $P$, i.e. 
\[
J^{\mathrm{den},P}(p)=\bp{X\subseteq \downarrow p: X \text{ is non-empty and downward closed}}
\]
for all $p\in P$. On the other hand note that $\emptyset$ belongs to $J^{\mathrm{sup},P}(0_P)$, while it does not belong to $J^{\mathrm{sup},P}(p)$ for any $p\in P^+$, as $\bigvee\emptyset=0_P$. This will have non-trivial effects on the properties of
 $J^{\mathrm{sup},P}$-separated presheaves on $P$ to be defined below. 

\end{enumerate}
\end{remark}


\subsection{Sheaves and separated presheaves on a Grothendieck topology}
We now turn to the notion of sheaf and separated presheaf according to these two topologies.

\begin{definition}\label{def:globgrtop}
Let $(P,\leq)$ be a preorder,  $\mathcal{C}$ be a concrete category, $J$ be either the dense topology or the sup topology on $P$ (the latter just in case $P$ is upward complete).
\begin{itemize}
\item A presheaf
$\F:P^\mathrm{op}\to\mathcal{C}$ is a \emph{$J$-separated presheaf} if for all $p\in P$ and  all matching families $\bp{f_d:d\in D}$ with $D\in J(p)$ there is at most one $g\in \F(p)$ such that $f_d=g\restriction d$ for all $d\in D$.
\item 
A presheaf $\F:P^\mathrm{op}\to\mathcal{C}$ is
a \emph{$J$-sheaf} if for all $b\in P$ and  all matching families $\bp{f_d:d\in D}$ with $D\in J(b)$ there is exactly one $g\in \F(b)$ such that $f_d=g\restriction d$ for all $d\in D$.
\end{itemize}
\end{definition}
We note the following:
\begin{itemize}
\item
Let $P$ be a frame and
$\F$ be a presheaf on  $P$.  Then $\F(0_P)$ has at most one element if $\F$ is a separated presheaf for the sup topology, and $\F(0_P)$ has exactly one element when it is a sheaf: the empty set trivially satisfies the notion of matching family,\footnote{One has to verify for $\emptyset=\bp{f_a:a\in\emptyset}$ the universal sentence asserting that for all $a,b\in P$ if $a,b\in\emptyset$, then $f_a\restriction a \wedge b=f_b \restriction a\wedge b$. The premise of the implication is alwasy false for all $a,b\in P$, hence the implication is always true for all $a,b\in P$. } and we already noted that $\emptyset\in J^{\mathrm{sup},P}(0_P)$. This gives that the empty family (which is indexed by the empty set) is a matching family for $\F$. Now any $g\in \F(0_P)$ trivially satisfies the property of being a gluing of the empty family.\footnote{One has to verify that for all $a\in P$, if $a\in\emptyset$, then $g\restriction a=f_a$. This is again a universal quantification on  an implication which is always verified because its premise is false.}. Hence $\F(0_P)$ cannot have more than one gluing of the empty family if $\F$ is a separated presheaf, and must have exactly one gluing if $\F$ is a sheaf. 
\item
As already outlined, when $(X,\tau)$ is a topological space, $\F$ is a sheaf on $\tau$ for the
sup topology if and only if it is a topological sheaf for $(X,\tau)$.
\item When $P$ is a frame, $\F$ is a sheaf on $P$ for the sup topology if and only if $\F\restriction P^+$ is is a sheaf on $P^+$ for the sup topology. This follows easily from the observation that a sheaf $\F$ for the sup topology on a frame $P$ is such that $\F(0_P)$ has exactly one element; consequently, to check the validity of the sheaf condition for the sup topology on $P$ one can restricts the attention to $P^+$.
\item For a complete boolen algebra $\bool{B}$, $\F$ is a sheaf on $\bool{B}$ for the sup topology if and only if $\F\restriction \bool{B}^+$ is a sheaf on $\bool{B}^+$ for the dense topology: by the previous item $\F$ is a sheaf on $\bool{B}$ for the sup topology if and only if $\F\restriction \bool{B}^+$ is a sheaf on $\bool{B}^+$ for the sup topology. But on 
$\bool{B}^+$ we aleady noted that the $\sup$ and the dense topology coincide.
\item We also note that a sheaf $\F$ for the sup topology on a complete boolean algebra $\bool{B}$ must have surjective restriction maps for all $b\leq c$ such that $\F(c)\neq\emptyset$: Assume $f\in \F(b)$ and $g\in\F(c)$ with $b\leq c$, then $\bp{f\restriction d:d\leq b}\cup\bp{g\restriction e: e\leq\neg b\wedge c}$ is a matching family indexed by the set $\downarrow\bp{b,\neg b\wedge c}$, hence it has a unique gluing $h\in\F(c)$. Clearly $h\restriction b=f$. Consequently a sheaf $\F$ such that $\F(1_\bool{B})$ is non-empty is determined by its globl sections.
\end{itemize}

In the remainder of this paper, when dealing with sheaf theoretic topics, we will (almost exclusively)  focus on the sup Grothendieck topology for a complete boolean algebra $\bool{B}$ and on presheaves $\F$ on $\bool{B}$ determined such that $\F(1_\bool{B})$ is non-empty.

\subsection{Bundles and \'etal\'e spaces}
Let us now move to the notion of \'etal\'e space (again we follow closely what is done in \cite[Chapter II]{MacLMoer}).
\begin{definition}
A \emph{bundle} on a topological space $(X,\tau)$ is a triple $(\pi, E, \sigma)$ where $(E, \sigma)$ is a topological space and $\pi:E\to X$ is a surjective continuous function. $X$ is the \emph{base space}.\footnote{We write $\pi:E\to X$ instead $(\pi, E, \sigma)$ when we what to emphasize the base space.}
A \emph{local section} of a bundle $(\pi, E, \sigma)$ is a continuous map $s:U\to E$ where $U\subseteq X$ is open, with $\pi\circ s$ being the identity on $U$.

 A bundle $(\pi, E, \sigma)$ on $X$ gives an \emph{\'etal\'e space} structure on $(E,\sigma)$ if $\pi$ is a \emph{local homeomorphism}, i.e. for some base $\mathcal{B}$ for $(X,\tau)$ the set
 \[
 \bp{s[U]:\,s\text{ is a section on $\pi$, }U\subseteq\dom(s),\, U\in\mathcal{B}}
 \]
 forms a basis of open sets for $(E,\sigma)$.
\end{definition}

Note that $(\pi,E,\sigma)$ is an \'etal\'e bundle on $(X,\tau)$ if and only if every continuous section $s:U\to E$ for $\pi$ with $U\in\tau$ is an open map.

\section{Dualities and adjuctions on topological spaces with open continuous maps}
\label{sec:dual}

In this section we define the type of partial orders/topological spaces over which we consider presheaves. Part of the content of this section recalls standard facts in the theory of frames and locales, other results are (at the best of our knowledge) hardly traceable in the literature. We refer to \cite[Chapter IX]{MacLMoer} and to \cite[Chapter C1]{JohnElephant} for an overview of the literature on the topic.
We need a bit of terminology.

%
%

\begin{definition}\label{def:adjhom}
Let $\bool{B},\bool{C}$ be Boolean algebras.

A map $i:\mathsf{B}\to\mathsf{C}$ is:
\begin{itemize}
\item 
an \emph{homomorphism} if it preserves the Boolean 
operations\footnote{In particular, an homomorphism of Boolean algebras is order preserving.};

\item
a \emph{complete homomorphism} if for all $A\subseteq\bool{B}$ and $a\in \bool{B}$ such that $\bigvee_{\bool{B}}A=a$, we have that $\bigvee_{\bool{C}}i[A]=i(a)$;

\item an
\emph{adjoint homomorphism} if it has a left adjoint $\pi_i:\bool{C}\to\bool{B}$.
\end{itemize}
Define $\pi^*_i:\St(\bool{C})\to\St(\bool{B})$ by $\pi^*_i(G)=i^{-1}[G]$.
\end{definition}

\begin{notation}\label{not:kfpii}
Let $(X,\tau)$, $(Y,\sigma)$ be topological spaces and $f:X\to Y$ be a continuous map.
Define
$\bar{\bar{k}}_f:\sigma\to\tau$ by $\bar{\bar{k}}_f(U)=f^{-1}[U]$\footnote{The map $\bar{\bar{k}}_f$ is denoted as $f^{-1}$ in \cite[Section IX.1]{MacLMoer} and as $f^*$ in \cite[Section II.1]{JohnStone}.} .

Let $\bar{k}_f=\bar{\bar{k}}_f\restriction \RO(Y)$ and $k_f=\bar{k}_f\restriction\CLOP(Y)$.
\end{notation}

The continuity of $f$ grants  that the range of $k_f$ is included in $\CLOP(X)$. 
In principle it is not clear that the range of $\bar{k}_f$ is included in $\RO(X)$; we will see that this is the case when $f$ is an open map.

We will use the above to define contravariant functors from (variations of) the category of topological spaces with open maps to (variations of) the category of Boolean algebras with \emph{adjoint} homomorphisms.
First of all we need notation to define the relevant categories and functors.
\begin{notation}
\emph{}

\begin{itemize}
\item $\bool{Top}_{\mathrm{open}}$ denotes the category of topological spaces with continuous open maps.
\item $\bool{TDCH}_{\mathrm{open}}$ denotes the subcategory with the same arrows but with objects given by the $0$-dimensional compact Hausdorff spaces.
\item $\bool{EDCH}_{\mathrm{open}}$ denotes
the subcategory with the same arrows but with objects given by the extremally disconnected compact Hausdorff spaces.
\item $\bool{BA}_{\mathrm{adj}}$ denotes the category of Boolean algebras with arrows given by \emph{adjoint} homomorphisms.
\item $\bool{CBA}$ denotes the category of complete Boolean algebras with arrows given by complete homomorphisms.
\end{itemize}

\begin{align*}
\mathcal{L}^*: & (\bool{Top}_{\mathrm{open}})^\mathrm{op}\to \bool{CBA}\\
& (X,\tau)\mapsto \RO(X)\\
& (f:X\to Y) \mapsto (\bar{k}_f:\RO(Y)\to\RO(X))
\end{align*}

\begin{align*}
\mathcal{D}: & (\bool{TDCH}_{\mathrm{open}})^\mathrm{op}\to \bool{BA}_{\mathrm{adj}}\\
& (X,\tau)\mapsto \CLOP(X)\\
& (f:X\to Y) \mapsto (k_f:\CLOP(Y)\to\CLOP(X))
\end{align*}

\begin{align*}
\mathcal{R}: & (\bool{BA}_{\mathrm{adj}})^\mathrm{op} \to \bool{TDCH}_{\mathrm{open}}\\
& \bool{B} \mapsto \St(\bool{B})\\
& (i:\bool{B}\to \bool{C}) \mapsto (\pi^*_i:\St(\bool{C})\to\St(\bool{B}))
\end{align*}

\end{notation}

The adjunctions/dualities we want to bring forward are summarized in the following theorem:

\begin{theorem}\label{thm:adj-dual-cbaTDCH}
The following holds:
\begin{enumerate}
\item
$\mathcal{D}$, $\mathcal{L}^*$, $\mathcal{R}$ are well defined functors.
\item
$\mathcal{D}$ implements a duality between $\bool{TDCH}_{\mathrm{open}}$ and 
$\bool{BA}_{\mathrm{adj}}$ with inverse given by $\mathcal{R}$.
\item
$\mathcal{D}\restriction\bool{EDCH}_{\mathrm{open}}$ implements a duality between $\bool{EDCH}_{\mathrm{open}}$ and 
$\bool{CBA}$ with inverse given by $\mathcal{R}\restriction\bool{CBA}$.
\item
Let
$\mathcal{L}=\mathcal{L}^*\restriction  \bool{TDCH}_{\mathrm{open}}$.
Then $(\mathcal{L},\mathcal{R}\restriction\bool{CBA})$ is a contravariant adjunction.
\end{enumerate}
\end{theorem}

The proof is an easy outcome of the following 
propositions linking the properties of (complete injective) homomorphisms to those of 
(open surjective) continuous functions and of adjoint pairs:\footnote{We expect their content to be folklore; however we have not been able to trace in the literature items  (\ref{prop:openadjpairs-4})  and (\ref{prop:openadjpairs-7}) of Proposition \ref{prop:openadjpairs}, and item (\ref{cor:openmapscomplhomI}) of Proposition \ref{thm:openmap-->adjoint}.}

\begin{proposition}\label{prop:openadjpairs}
	Let $i: \bool{B} \rightarrow \bool{C}$ be an homomorphism of Boolean algebras, and
	$f:X\to Y$ be a continuous function between compact $0$-dimensional Hausdorff spaces.

	Then:
	\begin{enumerate}[(i)]
	\item \label{prop:openadjpairs-1}
	$\pi^*_i$ is continuous and closed;
	\item \label{prop:openadjpairs-2}
	$i$ is injective if and only if $\pi^*_i$ is surjective, in which case $\pi^*_i[N_{i(b)}]=N_b$; 
	\item \label{prop:openadjpairs-3}
	$i=k_{\pi^*_i}$ modulo the identification of $\bool{D}$ with $\CLOP(\St(\bool{D}))$ given by 
	$b\mapsto N_b$ for $\bool{D}$ a Boolean algebra;
	\item \label{prop:openadjpairs-4}
	$i$ has a left adjoint $\pi_i$ if and only if $\pi^*_i$ is an open map, in which case $N_{\pi_i(c)}=\pi^*_i[N_c]$, and $i$ is a complete homomorphism;
	\item \label{prop:openadjpairs-5}
	$k_f$ is a homomorphism;
	\item \label{prop:openadjpairs-6}
	$\pi^*_{k_f}=f$ modulo the identification of $Z$ with $\St(\CLOP(Z))$ given by 
	\[
	z\mapsto G_z=\bp{U\in\CLOP(Z): z\in U}
	\] 
	for $Z$ compact $0$-dimensional Hausdorff space;
	\item  \label{prop:openadjpairs-7}
	 $k_f$ has a left adjoint $\pi_f$ if and only if $f$ is an open map, in which case 
	 $N_{\pi_f(c)}=f[N_c]$, and $k_f$ is a complete homomorphism.
	 \end{enumerate}
	 \end{proposition}
	 
	 Furthermore the above can be meshed with the following
	 \begin{proposition} \label{thm:openmap-->adjoint}
	 Let $X,Y$ be arbitrary topological spaces and
	  $f:X\to Y$ be an open continuous map, we have that
	 \begin{enumerate}[(A)]
	 \item\label{cor:openmapscomplhomI}  
	 $\Reg{f^{-1}[U]}=f^{-1}[\Reg{U}]$ for any open set $U$ of $Y$;
	 \item\label{cor:openmapscomplhomII} 
	 $\bar{k}_f:\RO(Y)\to\RO(X)$ given by $U\mapsto f^{-1}[U]$ is a complete homomorphism
	 with right adjoint $\bar{\pi}_f:\RO(X)\to\RO(Y)$ given by $V\mapsto f[V]$\footnote{An analogous result for $\bar{\bar{k}}_f:\mathcal{O}(Y)^+\to\mathcal{O}(X)^+$ is \cite[Chapter C3 Lemma 3.1.10]{JohnElephant}.};
	 \item \label{cor:openmapscomplhomIII} 
	 In case $X,Y$ are $0$-dimensional compact Hausdorff $\bar{k}_f\restriction\CLOP(Y)=k_f$
	 and $\bar{\pi}_f\restriction\CLOP(X)$ is the left adjoint of $k_f$;
	 \item \label{cor:openmapscomplhomIV} 
	 In case $X,Y$ are extremally disconnected compact Hausdorff $\bar{k}_f=k_f$.
	 \end{enumerate}
	 \end{proposition}

\begin{remark}
Let $\bool{B}$ be the Boolean algebra given by finite and cofinite subsets of $\omega$ and
$i:\bool{B}\to\pow{\omega}$ be the inclusion map. It can be checked that $i$ is a complete injective 
homomorphism
and $\pi^*_i$ is not an open map\footnote{This counterexample has been given by D. Monk.}.
On the other hand if $i:\bool{B}\to\bool{C}$ is an homomorphism between complete Boolean algebras, 
$i$ has an adjoint if and only if $i$ is complete.
In particular the notion of complete homomorphisms is weaker than that of 
adjoint homomorphism only when considering non-complete Boolean algebras. 
\end{remark}

First of all we prove Theorem \ref{thm:adj-dual-cbaTDCH} assuming the two propositions:
\begin{proof}
The fact that $\mathcal{L}^*$ is a well defined functor follows from \ref{cor:openmapscomplhomI} of
Proposition \ref{thm:openmap-->adjoint}.
 $\mathcal{D}$ is a duality by Stone duality and by (\ref{prop:openadjpairs-3}) and (\ref{prop:openadjpairs-4}) of Proposition \ref{prop:openadjpairs}. $\mathcal{D}\restriction\bool{EDCH}$ is also a duality by
the known result that the Stone duality maps the object of $\bool{EDCH}$ onto those of $\bool{CBA}$
bijectively up to isomorphism, by (\ref{prop:openadjpairs-3}) and (\ref{prop:openadjpairs-4}) of Proposition \ref{prop:openadjpairs}, and by
Proposition \ref{thm:openmap-->adjoint}. $(\mathcal{L},\mathcal{R}\restriction\bool{CBA})$ is an adjunction by Proposition \ref{thm:openmap-->adjoint} and by (\ref{prop:openadjpairs-3}) and (\ref{prop:openadjpairs-4}) of Proposition \ref{prop:openadjpairs}.
\end{proof}
We start with the proof of Proposition \ref{thm:openmap-->adjoint}.
\begin{proof}
\emph{}

\begin{description}
\item[(\ref{cor:openmapscomplhomI})]
Assume $U$ is open and $x\in \Reg{f^{-1}[U]}$. Then (by Lemma \ref{lem:charreg}) $x$ has an open neighborhood $V$ such that $V\cap f^{-1}[U]$ is dense in $V$. Since $f$ is open, $f[V]$ is an open neighborhood of $f(x)$. It is enough to show that
$U\cap f[V]$ is dense in $f[V]$ (again by Lemma \ref{lem:charreg}). Pick $Z$ open subset of $f[V]$. Then $f^{-1}[Z]$ is an open subset of $V$\footnote{Equivalently, by \cite[Lemma 1.5.3, Chapter C1]{JohnElephant} we have that $U\cap f[V]=f[V\cap f^{-1}[U]]$, and thus $U\cap f[V]$ is dense in $f[V]$ since $V\cap f^{-1}[U]$ is dense in $V$.},
hence it has non empty intersection with $f^{-1}[U]$. If $z$ is in this latter set, $f(z)$ is in $Z\cap U$. 

Conversely assume $x\in f^{-1}[\Reg{U}]$. Then (by Lemma \ref{lem:charreg}) $f(x)$ has an open neighborhood $V$ such that
$U\cap V$ is dense in $V$. It is enough to show that
$f^{-1}[U]\cap f^{-1}[V]$ is dense in $f^{-1}[V]$ (again by Lemma \ref{lem:charreg}). Pick $A$ non-empty  open subset of $f^{-1}[V]$.
Then $f[A]$ is a non-empty open subset of $V$, hence it has non-empty intersection with $U$. This easily gives
that $A\cap f^{-1}[U]$ is non-empty.

\item[(\ref{cor:openmapscomplhomII})]
By (\ref{cor:openmapscomplhomI}) we get that for $U$ regular open subset of $Y$
\[
\Reg{f^{-1}[U]}=f^{-1}[\Reg{U}]=f^{-1}[U],
\]
hence $f^{-1}[U]$ is a regular open subset of $X$ and $\bar{k}_f$ is well defined.

It is immediate to check that $\bar{k}_f$ is order preserving.

Furthermore
if $\bp{A_i:i\in I}$ is a family of regular open subsets of $Y$ and $A=\bigcup_{i\in I}A_i$, we have that
\begin{align*}
\bigvee_{\RO(X)}\bp{\bar{k}_f(A_i):i\in I}=\Reg{\bigcup_{i\in I}f^{-1}[A_i]}=\Reg{f^{-1}[\bigcup_{i\in I}A_i]}=
\Reg{f^{-1}[A]}=\\
=f^{-1}[\Reg{A}]=f^{-1}[\bigvee_{\RO(Y)}\bp{A_i:i\in I}]=\bar{k}_f(\bigvee_{\RO(Y)}\bp{A_i:i\in I}),
\end{align*}
hence $\bar{k}_f$ is a complete morphism of partial orders.

It is well known that a map $k:\bool{B}\to\bool{C}$ is a complete homomorphism of complete Boolean algebras if and only if $k$ is suprema and order preserving. Hence $\bar{k}_f$ is a complete homomorphism.

It is also immediate to check that for $U\in\RO(Y)$ and $V\in\RO(X)$,
$\bar{k}_f(U)=f^{-1}[U]\supseteq V$ if and only if $U\supseteq f[V]=\bar{\pi}_f(V)$, hence 
$(\bar{\pi}_f, \bar{k}_f)$ is an adjoint pair.
\item[(\ref{cor:openmapscomplhomIII}) and (\ref{cor:openmapscomplhomIV})] Trivial.
\end{description}
\end{proof}

We now prove Proposition \ref{prop:openadjpairs}:

\begin{proof}
\emph{}

\begin{description}
\item[(\ref{prop:openadjpairs-1}), (\ref{prop:openadjpairs-2}), (\ref{prop:openadjpairs-3})]
This is part of the standard proofs of Stone duality. 

\item[(\ref{prop:openadjpairs-4})]
If $\pi^*_i$ is open we have that $\pi^*_i[N_c]$ is clopen for all $c\in\bool{C}$. 
Furthermore for all $b\in\bool{B}$ and $c\in\bool{C}$
\[
N_{i(b)}\supseteq N_c\text{ if and only if }(\pi^*_i)^{-1}[N_b]\supseteq N_c
\text{ if and only if } N_b\supseteq \pi^*_i[N_c].
\] 
Hence
the map $\pi_i:N_c\mapsto \pi^*_i[N_c]$ is the left adjoint of $k_{\pi^*_i}$, or equivalently (letting 
$\pi_i(c)=\bigwedge_{\bool{B}}\bp{b:i(b)\geq c}$) $\pi_i$ is well defined and is the right adjoint of $i$.

For the completeness of $\pi_i$, apply  (\ref{cor:openmapscomplhomIII}) of  Proposition \ref{thm:openmap-->adjoint}, since infima and suprema in $\CLOP(X)$ are computed using the same operations of $\RO(X)$ (when they exist).

\item[(\ref{prop:openadjpairs-5})]
Trivial by continuity of $f$.

\item[(\ref{prop:openadjpairs-6})]
See (\ref{prop:openadjpairs-3}) or else:
\begin{align*}
\pi^*_{k_f}(G_x)=\bp{V\in\CLOP(Y): k_f(V)\in G_x}=\bp{V\in\CLOP(Y): f^{-1}[V]\in G_x}=\\
=\bp{V\in\CLOP(Y): x\in f^{-1}[V]}=\bp{V\in\CLOP(Y): f(x)\in V}=G_{f(x)}.
\end{align*}

\item[(\ref{prop:openadjpairs-7})]
Use (\ref{cor:openmapscomplhomIII}) of  Proposition \ref{thm:openmap-->adjoint}.

\end{description}
\end{proof}

\section{Boolean valued models}\label{sec:boolvalmod}
We consider only Boolean valued models for relational languages as (at least for our purposes) there is no lack of generality in doing so, and our proofs are notationally smoother in this set up.
\begin{definition}\label{def:boolvalmod}
Let $\LL=\bp{R_i: i\in I, c_j: j\in J}$ be a relational language (i.e. without functional symbols; constant symbols are allowed) and let $\bool{B}$ be a Boolean algebra. A \emph{$\bool{B}$-valued model} $\M$ for $\LL$ consists of:
\begin{enumerate}
\item a non-empty set $M$, called the \emph{domain} of $\M$;
\item the Boolean value of the equality symbol, which is a function
\[
\begin{aligned}
=^\M:&M^2\to\bool{B}\\
(\sigma,& \tau)\mapsto\Qp{\sigma=\tau}^\M_{\bool{B}}
\end{aligned};
\]
\item for each $n$-ary relational symbol $R\in\LL$, a function
\[
\begin{aligned}
R^\M:&M^n\to\bool{B}\\
(\sigma_1, \dots,& \sigma_n)\mapsto\Qp{R(\sigma_1, \dots, \sigma_n)}^\M_{\bool{B}}
\end{aligned};
\]
\item for each constant symbol $c\in\LL$, an element $c^\M\in M$.
\end{enumerate}
We require the following conditions to hold:
\begin{enumerate}
\item for every $\sigma, \tau, \pi\in M$,
\[
\begin{gathered}
\Qp{\sigma=\sigma}^\M_{\bool{B}}=1_{\bool{B}},\\
\Qp{\sigma=\tau}^\M_{\bool{B}}=\Qp{\tau=\sigma}^\M_{\bool{B}},\\
\Qp{\sigma=\tau}^\M_{\bool{B}}\wedge\Qp{\tau=\pi}^\M_{\bool{B}}\leq\Qp{\sigma=\pi}^\M_{\bool{B}};
\end{gathered}
\]
\item for every $n$-ary relational symbol $R\in\LL$ and for every $\sigma_1, \dots, \sigma_n, \tau_1, \dots, \tau_n\in M$,
\[
\Bigl(\bigwedge_{i=1}^n\Qp{\sigma_i=\tau_i}^\M_{\bool{B}}\Bigr)\wedge\Qp{R(\sigma_1, \dots, \sigma_n)}^\M_{\bool{B}}\leq\Qp{R(\tau_1, \dots, \tau_n)}^\M_{\bool{B}}.
\]
\end{enumerate}
A $\bool{B}$-valued model $\M$ is \emph{extensional} if $\Qp{\sigma=\tau}^\M_{\bool{B}}=1_{\bool{B}}$ entails $\sigma=\tau$ for all $\sigma, \tau\in M$.
\end{definition}
\begin{remark}
\label{rem:transeq}
There is a key difference between the above definition of Boolean valued models and another which is present in the category theoretic literature (see for example Monro's  \cite{MONRO86}): the category theoretic definition allows the existence of "conditional elements"  i.e. $\Qp{\sigma=\sigma}\neq 1_\bool{B}$ is possible for 
some $\sigma\in M$ (the \emph{global elements} are those satisfying the equality). This latter choice is natural from a categorical perspective, since it ensures an exact correspondence between this notion of $\bool{B}$-valued model and the categorical notion of  presheaves on $\bool{B}$. This "category theoretic" definition is furthermore coherent with the standard definition of Heyting valued model (as given for instance in \cite{FOURSCOTT}). Definition \ref{def:boolvalmod} is natural from a model theoretic/set theoretic perspective, as one should naturally expect that elements do not exist conditionally; it is in any case motivated by the practice of (the Boolean valued presentation of) forcing in set theory. In order to establish a natural correspondence with presheaves for the  notion of Boolean valued model as given in Def. \ref{def:boolvalmod} we will have to restrict slightly the class of presheaves of interest for us. 
\end{remark}
If no confusion can arise, we avoid the supscript $\M$ and the subscript $\bool{B}$. Moreover, we write equivalently $\sigma\in M$ or $\sigma\in\M$, identifying a Boolean valued model $\M$ with its underlying set $M$.

We now introduce the notion of Boolean morphism between Boolean valued models:
\begin{definition}\label{def:boolvalmorph}
Fix a language $\LL$. Let $\M_1$ be a $\bool{B}_1$-valued model for $\LL$ and $\M_2$ be a $\bool{B}_2$-valued model for $\LL$, where $\bool{B}_1$ and $\bool{B}_2$ are Boolean algebras. A \emph{morphism of Boolean valued models for $\LL$} is a pair $(\Phi, i)$ where:
\begin{itemize}
\item $i: \bool{B}_1\to \bool{B}_2$ is an adjoint morphism of Boolean algebras; 
\item $\Phi:M_1\to M_2$ is an \emph{$i$-morphism}, that is: for every $n$-ary relational symbol $R$ and every constant symbol $c$ in the language and for every 
$\tau_1,\dots, \tau_n\in M_1$, 
\[\Phi(c^{\M_1})=c^{\M_2},\]
\[ i(\llbracket R(\tau_1, \dots, \tau_n)\rrbracket^{\M_1}_{\bool{B}_1})\leq 
\llbracket R(\Phi(\tau_1), \dots, \Phi(\tau_n))\rrbracket^{\M_2}_{\bool{B}_2},\] 
\[ i(\llbracket \tau_1=\tau_2\rrbracket^{\M_1}_{\bool{B}_1})\leq \llbracket\Phi(\tau_1)=\Phi(\tau_2)\rrbracket^{\M_2}_{\bool{B}_2}.\]
If in both these equations equality holds, we call $\Phi$ an \emph{$i$-embedding}.\end{itemize}
If $i$ is an isomorphism, $\Phi$ is an $i$-embedding, and for all $\tau\in M_2$ there is $\sigma\in M_1$ such that 
$\Qp{\Phi(\sigma)=\tau}_{\bool{B}_2}=1_{\bool{B}_2}$, $\cp{\Phi,i}$ is an \emph{equivalence of Boolean valued models}.
\end{definition}
\begin{remark}
First of all notice that, if $(\Phi, i):(\M_1, \bool{B}_1)\to(\M_2, \bool{B}_2)$ is an isomorphism (which means in particular that both $\Phi$ and $i$ are invertible maps), then $(\Phi, i)$ is an equivalence of Boolean valued models. If both $\M_1$ and $\M_2$ are extensional, the converse holds.\\
Moreover, if $\bool{B}_1=\bool{B}_2=2$ and both $\M_1$ and $\M_2$ are extensional, an $\Id_2$-morphism and an $\Id_2$-embedding are exactly a morphism and an embedding between two
Tarski models, respectively.
\end{remark}
It is worth noting a further istinction between the above notion of morphism and that of Monro's \cite{MONRO86}: there he considers a morphism to be an equivalence class of such functions, where $\Phi$ and $\Psi$ are equivalent whenever $i(\llbracket\tau=\tau\rrbracket)\leq\llbracket\Phi(\tau)=\Psi(\tau)\rrbracket$\footnote{Remember that Monro does not require $\llbracket\tau=\tau\rrbracket=1_{\bool{B}}$.} for every $\tau$ in the first model (see \cite[Definition 5.2]{MONRO86}).

We now define the Boolean valued semantics. By now, we have not required any condition on the Boolean algebra $\bool{B}$. 
 For the definition of the semantic, however, we need to compute some infinite suprema. For this reason, we will always assume in the sequel that formulae are assigned truth values in the Boolean completion $\RO(\bool{B}^+)$ of $\bool{B}$ and we identify 
 $\bool{B}^+$ with a dense subset of the partial order $\RO(\bool{B}^+)^+$. In most cases $\bool{B}$ is rightaway a complete booolean algebra in which case
$\RO(\bool{B}^+)=\bool{B}$ and this complication vanishes.
\begin{definition}
Let $\M=\cp{M, =^\M, R^\M_i: i\in I}$ be a $\bool{B}$-valued model for the relational language $\LL=\bp{R_i: i\in I}$. We evaluate the formulae without free variables in the extended language $\LL_\M:=\LL\cup\bp{c_\sigma: \sigma\in M}$
by maps with values in the Boolean completion $\RO(\bool{B}^+)$ of $\bool{B}$ as follows:
\begin{itemize}
\item $\Qp{c_\sigma=c_\tau}^{\M}_{\RO(\bool{B}^+)}:=\Qp{\sigma=\tau}^\M_\bool{B}$ and 
$\Qp{R(c_{\sigma_1}, \dots, c_{\sigma_n})}^{\M}_{\RO(\bool{B}^+)}:=\Qp{R(\sigma_1, \dots, \sigma_n)}^\M_\bool{B}$;
\item $\Qp{\nphi\wedge\psi}^{\M}_{\RO(\bool{B}^+)}:=
\Qp{\nphi}^\M_{\RO(\bool{B}^+)}\wedge\Qp{\psi}^\M_{\RO(\bool{B}^+)}$;
\item $\Qp{\neg\nphi}^\M_{\RO(\bool{B}^+)}:=\neg\Qp{\nphi}^\M_{\RO(\bool{B}^+)}$;
\item $\Qp{\exists x\nphi(x, c_{\sigma_1}, \dots, c_{\sigma_n})}^\M_{\RO(\bool{B}^+)}:=
\bigvee_{\tau\in M}\Qp{\nphi(c_\tau, c_{\sigma_1}, \dots c_{\sigma_n})}^\M_{\RO(\bool{B}^+)}.$
\end{itemize}
If $\nphi(x_1, \dots, x_n)$ is any $\LL$-formula with free variables $x_1, \dots, x_n$ and $\nu$ is an assignment, we define $\Qp{\nu(\nphi(x_1, \dots, x_n))}^\M_{\RO(\bool{B}^+)}:=
\Qp{\nphi(c_{\nu(x_1)}, \dots, c_{\nu(x_n)})}^\M_{\RO(\bool{B}^+)}$.
\end{definition}
We will often write $\nphi(\sigma_1, \dots, \sigma_n)$ rather than $\nphi(c_{\sigma_1}, \dots, c_{\sigma_n})$
and $\Qp{\nphi(\tau_1, \dots, \tau_n)}$ rather than $\Qp{\nphi(\tau_1, \dots, \tau_n)}^\M_{\RO(\bool{B}^+)}$ if
no confusion can arise.

By induction on the complexity of $\LL$-formulae, it is possible to show the following:
\begin{fact}
For every $\sigma_1, \dots, \sigma_n, \tau_1, \dots, \tau_n\in M$ and for every $\LL$-formula $\nphi(x_1, \dots, x_n)$ with displayed free variables, the following holds:
\begin{equation}
\Bigl(\bigwedge_{i=1}^n\Qp{\sigma_i=\tau_i}\Bigr)\wedge\Qp{\nphi(\sigma_1, \dots, \sigma_n)}
\leq\Qp{\nphi(\tau_1, \dots, \tau_n)}.
\end{equation}
\end{fact}

\begin{remark}
Given a signature $\mathcal{L}$,
and a morphism $(\Theta,i)$ of a $\bool{B}$-valued model $\M$ for $\mathcal{L}$ into a $\bool{C}$-valued model $\N$ for $\mathcal{L}$
it is a natural request that of imposing $i$ to be an adjoint homomorphism, hence complete: otherwise the truth value of $i\cp{\Qp{\exists x\psi(x)}^\M_{\bool{B}}}$ might not be $\bigvee_{\bool{C}}\bp{i\cp{\Qp{\psi(\sigma)}^\M_\bool{B}}: \sigma\in\M}$, in which case the morphism would not preserve the expected meaning of quantifiers.
\end{remark}

The following is the correct generalization to Boolean valued semantics of the notion of elementary embedding between Tarski structures:
\begin{definition}\label{def:elembvm}
Given a signature $\mathcal{L}$,
a morphism $(\Theta,i)$ of a $\bool{B}$-valued model $\M$ for $\mathcal{L}$ into a $\bool{C}$-valued model $\N$
for $\mathcal{L}$ is elementary if for all $a_1,\dots,a_n\in\M$ and $\mathcal{L}$-formulae $\phi(x_1,\dots,x_n)$
\[
i(\Qp{\phi(a_1,\dots,a_n}^\M_\bool{B})=\Qp{\phi(\Theta(a_1),\dots,\Theta(a_n)}^\N_\bool{C}.
\]
\end{definition}

\begin{definition}
A sentence $\nphi$ in the language $\LL$ is \emph{valid} in a $\bool{B}$-valued model $\M$ for $\LL$ if $\Qp{\nphi}=1_{\RO(\bool{B})}$. 
A theory $T$ is \emph{valid} in $\M$ (equivalently, $\M$ is a \emph{$\bool{B}$-model for $T$}) if every axiom of $T$ is valid in $\M$.
\end{definition}
\begin{theorem}[Soundness and Completeness]
Let $\LL$ be a relational language. An $\LL$-formula $\nphi$ is provable syntactically by an $\LL$-theory $T$ if and only if, for every Boolean algebra $\bool{B}$ and for every $\bool{B}$-valued model $\M$ for $T$, $\Qp{\nu(\nphi)}^\M_{\RO(\bool{B})}=1_{\RO(\bool{B})}$ for every assignment $\nu$ taking values in $\M$.
\end{theorem}
\begin{proof}
See for instance \cite[Theorem 6.1.5]{viale-bookonforcing}.
\end{proof}

%

\subsection{\L o\'s Theorem for Boolean valued models and fullness}

\begin{definition}\label{def:quotboolmod}
Let $\bool{B}$ be a Boolean algebra and let $\M=\cp{M, =^\M, R^\M_i: i\in I}$ be a $\bool{B}$-valued model for the relational language $\LL=\bp{R_i: i\in I}$. Let $F$ be a filter of $\bool{B}$. The \emph{quotient} $\M/_F$ is the $\bool{B}/_F$-valued model for $\LL$ defined as follows:
\begin{itemize}
\item the domain is $M/_F:=\bp{[\sigma]_F: \sigma\in M}$, where $[\sigma]_F:=\bp{\tau\in M: \Qp{\tau=\sigma}\in F}$;
\item if $R\in\LL$ is an $n$-ary relational symbol and $[\sigma_1]_F, \dots, [\sigma_n]_F\in M/_{F}$,
\[
\Qp{R([\sigma_1]_F, \dots, [\sigma_n]_F)}^{\M/_F}_{\bool{B}/_F}:=\Bigl[\Qp{R(\sigma_1, \dots, \sigma_n)}^\M_{\bool{B}}\Bigr]_F\in\bool{B}/_F.
\]
\end{itemize}
\end{definition}
It is easy to check that this quotient is well-defined. In particular, if $G$ is an ultrafilter, the quotient 
$\M/_G$ is a two-valued Tarski structure for $\LL$.
On the other hand if $\bool{B}$ is a Boolean algebra and $F$ is a filter on $\bool{B}$, it can happen that 
$\bool{B}$ is complete while $\bool{B}/_F$ is not, and conversely. In particular it is not clear how the semantics of 
formulae with quantifiers is affected by the quotient operation. We now address this problem.
\begin{definition}\label{def:fullness-wellbehaved}
Given a first order signature $\LL$,
a $\bool{B}$-valued model $\M$ for $\LL$ is 
\begin{itemize}
\item
\emph{well behaved} if:
for all $\LL$-formulae $\phi(x_1,\dots,x_n)$
and $\tau_1,\dots,\tau_n\in\M$ $\Qp{\phi(\tau_1,\dots,\tau_n)}^\M_{\RO(\bool{B})}$ is in $\bool{B}$;
\item
\emph{full} if for all ultrafilters $G$ on $\bool{B}$, all $\LL$-formulae $\phi(x_1,\dots,x_n)$
and all $\tau_1,\dots,\tau_n\in\M$
\[
\M/_G\models\phi([\tau_1]_G,\dots,[\tau_n]_G)\qquad\text{ if and only if }\qquad\Qp{\phi(\tau_1,\dots,\tau_n)}^\M\in G.
\]
\end{itemize}
\end{definition}
Notice that if $\bool{B}$ is complete $\bool{B}\cong\RO(\bool{B}^+)$ hence any $\bool{B}$-valued model 
$\M$ is automatically well-behaved. But being well-behaved does not require $\bool{B}$ to be complete, just to be able to compute all the suprema and infima required by the satisfaction clauses for $\exists,\forall$.
On the other hand for atomic $\LL_\M$-formulae the condition expressing fullness is automatic by definition. 
The unique delicate case occurs for $\LL_\M$-formulae  in 
which the principal connective is a quantifier, for example of the form $\exists x\phi(x)$, in which case
the fullness conditions for the formula amounts to ask that $\sup_{\sigma\in\M}\Qp{\phi(\sigma)}^\M$ is actually a finite supremum, as we will see. We will investigate this property in many details in the last two sections of this paper. For the moment letting a 
$\LL_\M$-formula $\phi(x_1,\dots,x_n)$ be \emph{$F$-full} for a filter $F$ if and only if 
\[
\Qp{\phi([\tau_1]_F,\dots,[\tau_n]_F)}^{\M/_F}=1_{\bool{B}/_F}\qquad\Longleftrightarrow\qquad\Qp{\phi(\tau_1,\dots,\tau_n)}^\M\in F,
\]
we automatically have that the family of $F$-full $\LL_\M$-formulae includes the atomic formulae and is closed under conjuctions for any filter $F$, and under Boolean 
combinations for any ultrafilter $F$. 
Another key observation is that if $F$ is a filter and
\[
\bigvee_{i=1}^m\Qp{\phi(\sigma_i)}^{\M}=\Qp{\exists x\,\phi(x)}^{\M}=\bigvee_{\tau\in \M}\Qp{\phi(\tau)}^{\M},
\]
then
\[
\bigvee_{i=1}^m\Qp{\phi([\sigma_i]_F)}^{\M/_F}=\Qp{\exists x\,\phi(x)}^{\M/_F},
\]
since 
\[
\bigvee_{i=1}^m[\Qp{\phi(\sigma_i)}^\M]_F=\Bigl[\bigvee_{i=1}^m\Qp{\phi(\sigma_i)}^{\M}\Bigr]_F\geq_F [\Qp{\phi(\tau)}^{\M}]_F
\]
holds for all $\tau\in\M$.
The notion of full $\bool{B}$-valued model displays its full power in the following result:
\begin{theorem}[\L o\'s Theorem for Boolean valued models] \label{thm:charLosThm}
Let $\M$ be a well behaved $\bool{B}$-valued model for the signature $\LL$.
The following are equivalent:
\begin{itemize}
\item $\M$ is full.
\item for all $\LL_\M$-formulae $\phi(x_0,\dots,x_n)$
and all $\tau_1,\dots,\tau_n\in\M$ there exist $\sigma_1, \dots, \sigma_m\in\M$ such that
\[
\bigvee_{\tau\in \M}\Qp{\phi(\tau,\tau_1,\dots,\tau_n)}=\bigvee_{i=1}^m\Qp{\phi(\sigma_i,\tau_1,\dots,\tau_n)}.
\]
\end{itemize}
\end{theorem}

\begin{proof}
We sketch the proof for the case of existential formulae.\\
First, assume $\M$ to be full:
\[
\M/_G\models\exists x\psi(x, [\tau_1]_G,\dots,[\tau_n]_G)\quad\Leftrightarrow\quad\Qp{\exists x\psi(x, \tau_1,\dots,\tau_n)}^\M\in G.
\]
Thus for every $G$ such that $\Qp{\exists x\psi(x, \tau_1,\dots,\tau_n)}^\M\in G$ there exists $\sigma_G\in\M$ such that $\M/_G\models\psi([\sigma_G]_G, [\tau_1]_G,\dots,[\tau_n]_G).$ \\
Again using the hypothesis, we obtain that for every $G\in N_{\Qp{\exists x\psi(x, \tau_1,\dots,\tau_n)}}$ there exists $\sigma_G\in\M$ such that $G\in N_{\Qp{\psi(\sigma_G, \tau_1,\dots,\tau_n)}}$, that is:
\[
N_{\Qp{\exists x\psi(x, \tau_1,\dots,\tau_n)}}\subseteq\bigcup_{\sigma\in\M}N_{\Qp{\psi(\sigma, \tau_1,\dots,\tau_n)}}.
\]
By compactness, $N_{\Qp{\exists x\psi(x, \tau_1,\dots,\tau_n)}}=\bigcup_{i=1}^mN_{\Qp{\psi(\sigma_i, \tau_1, \dots, \tau_n)}}$, which is our thesis.\\
Conversely, assume that for every formula $\phi$, $\tau_1, \dots, \tau_n\in\M$ there exist $\sigma_1, \dots, \sigma_m\in\M$ such that
\[
\bigvee_{\tau\in \M}\Qp{\phi(\tau,\tau_1,\dots,\tau_n)}=\bigvee_{i=1}^m\Qp{\phi(\sigma_i,\tau_1,\dots,\tau_n)}.
\]
By induction on the complexity of the formulae we prove that $\M$ is full. Let us consider only the non-trivial case: $\psi(x_1,\dots, x_n)=\exists x\phi(x, x_1, \dots, x_n)$. Then
\begin{align*}
\M/G\models\exists x\phi(x, [\tau_1]_G, \dots, [\tau_n]_G)\Leftrightarrow&\M/G\models\phi([\sigma]_G, [\tau_1]_G, \dots, [\tau_n]_G)\\
&\textit{ for some }\sigma\in\M\\
\Leftrightarrow&\Qp{\phi(\sigma,\tau_1, \dots, \tau_n)}\in G\textit{ for some }\sigma\in\M\\
\Rightarrow&\Qp{\exists x\phi(x, \tau_1, \dots, \tau_n)}\in G.
\end{align*}
Conversely, if $\Qp{\exists x\phi(x, \tau_1, \dots, \tau_n)}\in G$, since $\Qp{\exists x\phi(x, \tau_1, \dots, \tau_n)}=\bigvee_{i=1}^m\Qp{\phi(\sigma_i, \tau_1, \dots, \tau_n)}$, there exists $i\in\bp{1, \dots, m}$ such that $\Qp{\phi(\sigma_i, \tau_1, \dots, \tau_n)}\in G$. By inductive hypothesis
\[
\M/G\models\phi([\sigma_i]_G, [\tau_1]_G, \dots, [\tau_n]_G)
\]
and so $\M/G\models\exists x\phi(x, [\tau_1]_G, \dots, [\tau_n]_G)$.
\end{proof}
One recovers the standard \L o\'s Theorem for ultraproducts observing that any product $\prod_{i\in I}N_i$ of $\LL$-structures
is a full $\pow{I}$-valued model for $\LL$, and that the semantic of the quotient structure $\bigl(\prod_{i\in I}N_i\bigr)/G$ for $G$ ultrafilter on $\pow{I}$
is exactly ruled by the conditions set up in the above theorem.\\

Fullness is not automatic (and as we will see is a slight weakening of the notion of sheaf):
\cite[Example 6.3.5]{viale-bookonforcing}
gives a an explicit example of a Boolean valued model which is well behaved but not full. 

\subsection{The mixing property and fullness} 

The mixing property gives a sufficient condition for  fullness  which is, usually, easier to check; as we will see the mixing property characterizes those Boolean valued models which are sheaves for the Sup Grothendieck topology.
\begin{definition}\label{def:mixprop}
Let $\kappa$ be a cardinal,
$\LL$ be a first order language, $\bool{B}$ a 
$\kappa$-complete Boolean algebra,
$\M$ a $\bool{B}$-valued model for $\LL$.
\begin{itemize}
\item
$\M$ satisfies the \emph{$\kappa$-mixing property} if for every antichain $A\subset \bool{B}$
of size at most $\kappa$, and for every subset $\{\tau_a: a\in A\}\subseteq M$, there exists 
$\tau\in M$ such that $a\leq\llbracket\tau=\tau_a\rrbracket$ for every $a\in A$.
\item
$\M$ satisfies the \emph{$<\kappa$-mixing property}
if it satisfies the $\lambda$-mixing property for all
cardinals $\lambda<\kappa$.
\item
$\M$ satisfies the \emph{mixing property}
if it satisfies the $|\bool{B}|$-mixing property.
\end{itemize}
\end{definition}
In \cite{HODG} models with the $<\omega$-mixing property are called models which \emph{admit gluing}.
Whether a $\bool{B}$-valued model $\M$ has the mixing property depends only on the interpretation of the equality symbol by $\Qp{-=-}^\M$.
\begin{proposition}
\label{ful:mod}
Let $\bool{B}$ be a complete Boolean algebra and let $\M$ be a $\bool{B}$-valued model for $\LL$.  Assume that $\M$ satisfies the $\kappa$-mixing property for some $\kappa\geq\min\bp{|\bool{B}|, |M|}$. Then $\M$ is full, according to the following stronger criterion:
\begin{quote}
\emph{
For every formula $\phi(x_0, x_1, \dots, x_n)$ and for every $\tau_1, \dots, \tau_n\in\M$ there exists an element $\tau_0$ such that}
\[
\bigvee_{\sigma_\in\M}\Qp{\phi(\sigma, \tau_1, \dots, \tau_n)}=\Qp{\phi(\tau_0, \tau_1, \dots, \tau_n)}.
\]
\end{quote}
\end{proposition}
For a proof see \cite[Proposition 6.3.14, Remark 6.3.15]{viale-bookonforcing}.
The stronger form of fullness appearing in the above Proposition is the definition of fullness one can find for instance in \cite{JECH}. It is easy to see that this property is true in every full model $\M$ satisfying the $<\omega$-mixing property.

We will now exhibit an example of full Boolean valued model which does not satisfy the mixing property. To do so, we let $V$ 
denote the class of all sets and assume $(V,\in)$ is a model of \ZFC. The reader not familiar with forcing as developed in \cite{viale-bookonforcing} or also in any other classic reference (as
Kunen's \cite{KUNEN}, Jech's \cite{JECH}, or Bell's \cite{BELL} textbooks) can safely skip this example without compromising the comprehension of the remainder of this article.
\begin{example}\label{ex:fullnotimplmix}
\label{non:mix}
We stick in this example to the  terminology on forcing as in \cite{viale-bookonforcing}.
Let $M\in V$ be a countable transitive model of (a sufficiently large fragment of) $\ZFC$ and let $\bool{B}\in M$ be an 
infinite Boolean algebra which $M$ models to be complete. Consider the $\bool{B}$-valued model 
$(M^{\bool{B}},\Qp{-=-}^{M^{\bool{B}}},\Qp{-\in-}^{M^{\bool{B}}})$, which is a definable class in $(M,\in)$ 
(and a countable set in $V$). Now, it can be proved (see \cite[Mixing Lemma 7.2.1]{viale-bookonforcing}) that
\[
M\models (M^{\bool{B}},\Qp{-=-}^{M^{\bool{B}}})\textit{ has the mixing property}
\]
In particular,
\[
M\models (M^{\bool{B}},\Qp{-=-}^{M^{\bool{B}}},\Qp{-\in-}^{M^{\bool{B}}})\textit{ is full}.
\]
The notion of being full is absolute between transitive models of set theory,\footnote{If $M$ proves that $M^{\bool{B}}$ is full, $M$ proves that a certain element $\bigvee_{i=1}^m\Qp{\nphi(a_i)}^{M^{\bool{B}}}\in\bool{B}$ is equal to the element 
$\Qp{\exists x\nphi(x)}^{M^{\bool{B}}}\in\bool{B}$, and this equality between elements of $\bool{B}$ must be true also in $V$.} 
therefore
\[
(V,\in)\models(M^{\bool{B}},\Qp{-=-}^{M^{\bool{B}}},\Qp{-\in-}^{M^{\bool{B}}})\textit{ is full}.
\]
For a complete analysis of $M^\bool{B}$ we refer to \cite[Sections 7.1.2 and 7.2]{viale-bookonforcing} and in particular to \cite[Lemma 7.2.1]{viale-bookonforcing}.

We will now show that, in $V$, $(M^{\bool{B}},\Qp{-=-}^{M^{\bool{B}}},\Qp{-\in-}^{M^{\bool{B}}})$ does not satisfy the 
mixing property. Being $M$ countable, for sure there exists a maximal antichain $A$ of $\bool{B}\subseteq M$ such that 
$A\in V\setminus M$. We claim that this antichain witnesses the fact that the $\bool{B}$-valued model 
$(M^{\bool{B}},\Qp{-=-}^{M^{\bool{B}}},\Qp{-\in-}^{M^{\bool{B}}})$ 
does not have the mixing property.\\
By contradiction, let $\tau\in M^{\bool{B}}$ be such that
\[
\Qp{\tau=\check{a}}\geq a
\]
for every $a\in A$. Then $A$ is definable in $M$ by the formula $\nphi(x,\bool{B})$ where
\[
\nphi(x,\bool{B}):=(x\textit{ is an antichain of $\bool{B}$})\wedge(\bigvee_{\bool{B}}x=1_{\bool{B}})\wedge\forall y(y\in x\lrao(y\in\bool{B}\wedge\Qp{\tau=\check{y}}\geq y)).
\]
The antichain $A$ is the unique solution of the formula $\nphi$. Indeed, assume $C$ to be another solution of $\nphi$. 
Let $c\in C\setminus A$. Then, by the maximality of $A$, there exists $a\in A$ such that $a\wedge c>0$. 
Being $A$ and $C$ solutions of $\nphi$, we have
\[
0_\bool{B}<a\wedge c\leq\Qp{\tau=\check{a}}\wedge\Qp{\tau=\check{c}}\leq\Qp{\check{a}=\check{c}}=0_{\bool{B}},
\]
where $\Qp{\check{a}=\check{c}}=0_\bool{B}$ since $M\models(a\neq c)$ and so 
\[
M\models\Qp{\neg(\check{a}=\check{c})}^{M^{\bool{B}}}_{\bool{B}}=1_\bool{B}.
\]
Being $A$ a subset of $\bool{B}$ definable in $M$, by the Axiom of Comprehension (which holds in the $\ZFC$-model $M$) we have that $A\in M$, against our assumption.
\end{example}

\section{The presheaf structure of a Boolean valued model}
\label{sec:pres}
In this section we describe the correspondence between Boolean valued models and presheaves. We will use this correspondence to describe the mixing property both in terms of sheaves and of sections of \'etal\'e spaces.

Most of the results in this section are special instances of Fourman--Scott's work \cite{FOURSCOTT} and Monro's works \cite{MONROBIS, MONRO86}. We decide to give explicitly the proofs in order to take care of the differences between our definition of Boolean valued models and Monro's one, and to make our presentation self-contained and accessible to readers  (as we are) familiar with set theory.

\subsection{Boolean valued models as presheaves on a Boolean algebra}
We have all the ingredients needed to transfer the notion of Boolean valued model to the sheaf theoretic setting.

Fix a Boolean algebra $\bool{B}$ and a $\bool{B}$-valued model $\M$. 
For every $b\in B^+$, let $F_b$ be the filter generated by $b$, and consider 
the $\bool{B}/_{F_b}$-valued model $\M/_{F_b}$.
Note that if $N_b\subseteq N_c$ (i.e. $b\leq c$), then $F_b\supseteq F_c;$ therefore the
natural map 
\begin{align*}
i^\M_{bc}:& \M/_{F_c}\to \M/_{F_b}\\
& [\tau]_{F_c}\mapsto [\tau]_{F_b}
\end{align*}
defines a morphism from the $\bool{B}/_{F_c}$ valued model $\M/_{F_c}$ onto the $\bool{B}/_{F_b}$-valued 
model $\M/_{F_b}$. Note also that $\bool{B}/_{F_b}$ is isomorphic to $\bool{B}\restriction b$ via the map $[d]_{F_b}\mapsto d\wedge b$ for $d\in \bool{B}$.
Note also that $F_{0_\bool{B}}=\bool{B}$ and $\bool{B}/_{F_{0_\bool{B}}}$ is the trivial Boolean algebra with one element and $\M/_{F_{0_\bool{B}}}$ is the trivial Boolean valued model whose domain has one element and all relations are interpreted with full extension.

It is convenient from now on to look at a Boolean algebra at times as an algebraic structure and at other times as the clopens of its Stone space. We will freely do so in the sequel.


\begin{definition}\label{def:preshbvm}
Given a Boolean algebra $\bool{B}$
and a $\bool{B}$-valued model $\M$ with domain $M$ for a signature $\LL$, its associated presheaf $\F_\M:\bool{B}^{\textrm{op}}\to\dSet$ is defined by:
\begin{itemize}
\item $\F_\M(b):=M/_{F_b}$ for any $b\in\bool{B}$;

\item $\F_\M(b\leq c)$ for $b\leq c\in\bool{B}$ is the map
\begin{align*}
i^\M_{bc}:& \M/_{F_c}\to \M/_{F_b}.\\
& [\tau]_{F_c}\mapsto [\tau]_{F_b}
\end{align*}
\end{itemize}
\end{definition}

\begin{remark}
For any $\bool{B}$-valued model $\M$,
$\F_\M$ is a separated presheaf on  $\bool{B}$  for the supremum Grothendieck topology which furthermore has the property that the restriction maps $\F_\M(b\leq c)$ are always surjective. In particular $\F_\M$ is determined by its global sections (according to notation \ref{not:globgrtop}).
 \end{remark}

\subsection{The mixing property corresponds to the sheaf condition}
We can now relate the notion of mixing model to that of sheaf for the sup Grothendieck topology. This is obtained via the above translation, by means of the following characterization of Boolean valued models with the mixing property (which is the Boolean case of the general result holding for Heyting valued models \cite[Proposition 5.6]{MONRO86} or \cite[Theorem 4.13]{FOURSCOTT}):
\begin{proposition}
\label{chr:mix}
Let $\bool{B}$ be a complete Boolean algebra and
$\M$ be a $\bool{B}$-valued model.
Then $\M$ has the mixing property if and only if the separated presheaf 
$\F_\M$
of Definition \ref{def:preshbvm} is a sheaf.
\end{proposition}
\begin{proof}
Assume that $\M$ has the mixing property.
Let $b\in\bool{B}$ and let $X=\bp{b_i: i\in I}$ with $\bigvee_{i\in I}b_i=b$ and $\bp{f_i:i\in I}$ be a matching family such that $f_i=[\sigma_i]_{F_{b_i}}\in \F_\M(b_i)=\M/_{F_{b_i}}$ for all $i\in I$.

By Zorn's Lemma, we can find $A\subseteq \downarrow X$ maximal antichain. Then $\bigvee A=\bigvee X=b$ (for example by \cite[Fact 3.11.7]{viale-bookonforcing})
and for each $a\in A$ there is some $i_a\in I$ such that $a\leq b_{i_a}$.
%
Now (since $\bp{f_i:i\in I}$ is matching) for all
$i\neq j$ in $I$ 
\begin{equation}
\label{cond:comp}
f_i\restriction b_i\wedge b_j=f_j\restriction b_i\wedge b_j.
\end{equation}
Condition \eqref{cond:comp} can be rewritten as
\[
\Qp{\sigma_i=\sigma_j}\geq b_i\wedge b_j\quad\textit{ for every $i, j\in I$}.
\]
 Let $g_a:=f_{i_a}\restriction a$ for all $a\in A$; then
$g_a=[\sigma_{i_a}]_{F_{a}}$ for all $a\in A$.
Since $\M$ satisfies the mixing property, there exists $\tau\in\M$ such that $\Qp{\tau=\sigma_a}\geq a$ for every $a\in A$.
Then
\[
\Qp{\tau=\tau_i}\geq\bigvee_{a\in A}(\Qp{\tau=\sigma_{i_a}}\wedge\Qp{\sigma_{i_a}=\sigma_i})\geq \bigvee_{a\in A}(a \wedge (b_{i_a}\wedge b_i))=(\bigvee A) \wedge b_i=b\wedge b_i=b_i.
\]
Hence $[\tau]_{F_b}$ is the desired collation of $\bp{f_i:i\in I}$.

This argument shows that $\F_\M$ is a sheaf for the Sup Grothendieck topology on $\bool{B}$.

Conversely, suppose $\F_\M$ is a sheaf for the Sup Grothendieck topology on $\bool{B}$. 
 Let $A$ be an antichain in $\bool{B}$ and let $\sigma_a\in\M$ for every $a\in A$.
 Now for $a\neq a'$ in $A$, since $A$ is an antichain, $a\wedge a'=0$, and so it is clear that 
$\F_\M(a\wedge a'\leq a)([\sigma_a]_{F_a})=\F_\M(a\wedge a'\leq a')([\sigma_{a'}]_{F_{a'}})$, as both are the unique element of $\M/_{F_{0_\bool{B}}}$.

Let $b:=\bigvee A$.
Being $\F_\M$ a sheaf, there exists $\tau\in\M$ such that $\F_\M(a\leq b)([\tau]_{F_b})=[\sigma_a]_{F_a}$ for all $a\in A.$ This is equivalent to say that $\Qp{\tau=\sigma_a}\geq a$ for every $a\in A$. 
Hence $\M$ satisfies the mixing property.
\end{proof}

\subsection{Stalks on presheaves corresponds to Tarski quotients of Boolean valued models}
We do not expand (for the moment) on the transfer of the model theoretic structure associated to $\LL$ on $\M$ to $\F_\M$, but it is to be expected that this can be done "coherently". In order to appreciate why this is the case we show now that this procedure allows to identify the "stalks" of $\F_\M$ with the quotients of $\M$ by an ultrafilter, as shown in the following: 

\begin{notation}\label{not:extRO(X)toO(X)}
Let $\bool{B}$ be a complete boolean algebra.
A  $\F:\bool{B}^{\mathrm{op}}\to\dSet$ is naturally extended to a topological presheaf $\F^*$ defined on the open sets of $\St(\bool{B})$ (as in \cite[Chapter II]{MacLMoer}) by letting, for every open $U\subseteq\St(\bool{B})$, 
and for every open sets $U\subseteq V\subseteq \St(\bool{B})$
\[\F^*(U\subseteq V):=\F(\Reg(U)\subseteq\Reg(V)),\] 
where in both cases we are identifying the regular open subsets of $\St(\bool{B})$ with the elements of $\bool{B}$.
\end{notation}

\begin{remark}\label{rem:stalks=quots}
Fix $\bool{B}$ a complete Boolean algebra and $\M$ a $\bool{B}$-valued model. 
Let us define the "categorical translation" of the interpretation in $\M$ of the symbols of $\LL$ by setting:
\begin{itemize}
\item $\mathcal{F}_\M^R(b)=\bp{([\tau_1]_{F_b},\dots,[\tau_n]_{F_b}):\Qp{R(\tau_1,\dots,\tau_n)}\geq b}$ for all $b\in\bool{B}$ and all $n$-ary relation symbols $R$ of $\LL$;
\item $\mathcal{F}_\M^c(b)=[c^\M]_{F_b}$ for all constants symbols $c$ of $\LL$ and for all $b\in\bool{B}$.
\end{itemize}

For the specific case of the functor $\F_\M$ 
We aim to show that there is a natural identification of the stalks of it extension $\F^*_\M$ (as defined in Notation \ref{not:extRO(X)toO(X)}) at $G\in\St(\bool{B})$ (as defined in \cite[Chapter II]{MacLMoer}) with the Tarski $\LL$-structure $\M/_G$.

We can do this as follows:
First of all, translating  in our setting the definition of stalk (as defined in \cite[Chapter II]{MacLMoer}) at the point $G\in\St(\bool{B})$ of the topological presheaf $\F^*_\M$, we have that 
\[
(\F^*_\M)_G:=\Bigl(\bigsqcup_{b\in G}\M/_{F_b}\Bigr)/_{\sim_G},
\]
 where 
\[
 [\sigma]_{F_b}\sim_G[\tau]_{F_c}\quad\textit{if and only if}\quad[\sigma]_{F_{b\wedge c}}=
 [\tau]_{F_{b\wedge c}}.
 \]
Next we observe that for each $\sigma\in\M$ and $b\in G$
\begin{align*}
[\sigma]_{F_b}&=
\bp{\tau\in\M: \Qp{\sigma=\tau}\geq b}.
\end{align*}
In particular the map $\Theta_G:\sigma\mapsto [[\sigma]_{F_{1_{\bool{B}}}}]_{\sim_G}$ 
defines a surjection of $\M$ onto 
$(\F^*_\M)_G$ with 
\[
\begin{aligned}
\Theta_G(\sigma)=\Theta_G(\tau)&\Longleftrightarrow[[\sigma]_{F_{1_{\bool{B}}}}]_{\sim_G}=[[\tau]_{F_{1_{\bool{B}}}}]_{\sim_G}\\
&\Longleftrightarrow[\sigma]_{F_b}=[\tau]_{F_b}\textit{ for some }b\in G\\
&\Longleftrightarrow\Qp{\sigma=\tau}\geq b\textit{ for some }b\in G\\
&\Longleftrightarrow\Qp{\sigma=\tau}\in G.
\end{aligned}
\]
This shows that $\Theta_G(\sigma)=\Theta_G(\tau)$ if and only if $\M/_G\models [\sigma]_G=[\tau]_G$.\\
In particular we can identify (as models for the pure equality language) the stalk $(\F^*_\M)_G$ and the Tarski quotient $\M/_G$ of $\M$ at $G$ via the map 
$[[\tau]_{F_{1_\bool{B}}}]_{\sim_G}\mapsto [\tau]_G$.

Now we have a natural interpretation of the symbols $-$ of $\LL$ in $(\F^*_\M)_G$ which is either obtained by "passing to the quotient" $\F_\M^{-}$ by $\sim_G$ or by taking their interpretation on $\M/_G$ and moving them to $(\F^*_\M)_G$ via the above map, in either case we obtain that the interpretation of an $n$-ary relation symbol $R$ is given by
\[
[\mathcal{F}_\M^R(1_\bool{B})]_{\sim_G}=\bp{([[\tau_1]_{F_{1_\bool{B}}}]_{\sim_G},\dots,[[\tau_n]_{F_{1_\bool{B}}}]_{\sim_G}):\Qp{R(\tau_1,\dots,\tau_n)}\in G},
\]
and
\[
[\mathcal{F}_\M^c(1_\bool{B})]_{\sim_G}=[[c^\M]_{F_{1_\bool{B}}}]_{\sim_G}.
\]

We have in this manner identified as Tarski $\LL$-structures the stalk of $\F^*_\M$ at $G$ and the  structure $\M/_G$.
\end{remark}

\begin{remark}
Fix a $\bool{B}$-valued model $\M$ with $\bool{B}$ a complete Boolean algebra and its associated topological presheaf $\F^*_\M$ on the open sets of $\St(\bool{B})$. 

The corresponding \'etal\'e space on $\St(\bool{B})$ is obtained as the disjoint union of the fibers $(\F_\M)_G$ for $G\in\St(\bool{B})$, which -by Remark \ref{rem:stalks=quots}- amounts to a disjoint union of the Tarski structures $\M/_G$ as $G$ varies in $\St(\bool{B})$, that is:
\[
E_\M=\bp{[\sigma]_G: \sigma\in M, G\in\St(\bool{B})}.
\]
Equivalently, we can say that
\[
E_\M=\bp{\cp{\sigma, G}: \sigma\in M\textit{ and }G\in\St(\bool{B})}/_{R_\M},
\]
where $R_\M$ is the equivalence relation asserting that $\cp{\sigma, G}R_\M\cp{\tau, H}$ if and only if 
$G=H$ and $\Qp{\sigma=\tau}\in G$.

The projection map $p:E_\M\to\St(\bool{B})$ sends each $\cp{\sigma, G}=[\sigma]_G$ to $G$.

A base for the topology of $E_\M$ is the family
\[
\mathcal{B}:=\bp{\dot{\sigma}[N_b]=\bp{[\cp{\sigma, G}]_{R_\M}: b\in G}: \sigma\in M, b\in\bool{B}}.
\]
This topology is readily Hausdorff. Indeed, if $\cp{\sigma, G}\neq\cp{\tau, H}$, either $G\neq H$ or $G=H$ and 
$\neg\Qp{\sigma=\tau}\in G$. In the first case, $\St(\bool{B})$ being Hausdorff, there exists an open 
neighborhood $N_b$ of $G$ and an open neighborhood $N_c$ of $H$ which are disjoint. Then the basic open 
sets $\dot{\sigma}[N_b]$ and $\dot{\tau}[N_c]$ separate $\cp{\sigma, G}$ from $\cp{\tau, H}$. 
Otherwise, if $G=H$, we have that $[\sigma]_G\neq[\tau]_G$. 
Let $b:=\Qp{\sigma\neq\tau}=\neg\Qp{\sigma=\tau}\in G$. 
Then, $\dot{\sigma}[N_b]$ and $\dot{\tau}[N_b]$ are disjoint open neighborhoods of 
$\cp{\sigma, G}$ and $\cp{\tau, G}$, respectively.
Moreover, for any open $U\subseteq\St(\bool{B})$ and $\sigma\in M$, the local section 
$\dot{\sigma}\restriction {U}$ is open and injective and so it is an homeomorphism on its image.
\end{remark}


\subsection{The adjunction between Boolean valued models and presheaves determined by their global sections}
 So far we have shown the following:

\begin{lemma}

There is a functor:
\begin{align*}
L:\modx\to\Presh^{\textrm{bool}}(\dSet),
\end{align*}
where:
\begin{enumerate}[(A)]
\item $\modx$ is the category of Boolean valued models for the empty language $\bp{=}$ and morphisms between them, with the condition that, if $(\Phi, i):(\M, \bool{B})\to(\N, \bool{C})$ is a morphism, $i$ is an adjoint homomorphism, while the map $\Phi$ need not be injective;
\item\label{item:defmorphpreshboolvalmod}
 $\Presh^{\textrm{bool}}(\dSet)$ is the category:
\begin{itemize}
\item
whose objects are $\dSet$-valued presheaves $\F$ on some Boolean algebra $\bool{B}$ such that $\F(1_\bool{B})$ is non-empty, and 
\item
whose arrows between $\F_0:\bool{B}^{\textrm{op}}\to\dSet, \F_1:\bool{C}^{\textrm{op}}\to\dSet$ are pairs $(\gamma, i)$ where $i:\bool{B}\to\bool{C}$ is an adjoint homomorphism of Boolean algebras and $\gamma=\{\gamma_b:\F_0(b)\to\F_1(i(b))\}_{b\in\bool{B}}$ is a family of maps such that for all $b_1\leq b_2$ the following diagrams commute:
\begin{equation}\label{eqn:diagmorphboolpresh}
\begin{CD}
\F_0(b_2) @>\gamma_{b_2}>>  \F_1(i(b_2))\\
@V{\F_0(b_1\leq b_2)}VV              @VV{\F_1(i(b_1)\leq i(b_2))}V\\
\F(b_1)@>>\gamma_{b_1}>  \F_1(i(b_1))\\
\end{CD}
\end{equation}
\end{itemize}
\item For $\M$ an object in $\modx$, its image under $L$ is the presheaf $L(\M):=\F_\M$.
\item For $(\Phi, i):\M\to \N$ is a morphism in $\modx$ with $i:\bool{B}\to\bool{C}$, its image under $L$ is the morphism $(L(\Phi), i)$ with
$L(\Phi)=\bp{L(\Phi)_b}_{b\in\bool{B}}$ such that 
\[
\begin{aligned}
L(\Phi)_b:=\Phi/_{F_b}:&\M/_{F_b}\to\N/_{F_{i(b)}}\\
&[\tau]_{F_b}\mapsto[\Phi(\tau)]_{F_{i(b)}}
\end{aligned},
\]
where $F_b$ is the filter in $\bool{B}$ generated by $b$ and $F_{i(b)}$ is the filter in $\bool{C}$ generated by $i(b)$.
\end{enumerate}
\end{lemma}

The missing details in the proof of the Lemma are left to the reader; specifically we still need to check that $L(\Phi)$ is well-defined and that the required composition and commutativity laws 
hold for $L$, making $L$ indeed a functor. 
Note that we required $\Presh^{\textrm{bool}}$ to include only the presheaves on some $\bool{B}$ admitting at least one elements on the section given by $1_\bool{B}$, because this is the minimal requirement in order that the converse functor mapping  presheaves on Boolean algebras to Boolean valued models can be meaningfully defined.

\begin{definition}
\[
R:\Presh^{\textrm{bool}}(\dSet)\to\modx
\]
is the functor defined by:
\begin{itemize}
\item For $\F$ any\footnote{Note that $R$ can be defined on arbitrary $\bool{B}$-presheaves such that $\F(1_\bool{B})\neq\emptyset$, however the pairs $\cp{L,R}$ will define an adjunction only when $R$ is restricted to the class of \emph{separated} 
presheaves determined by their global sections and of Boolean valued models for complete Boolean algebras (see Theorem \ref{thm:adjbvmodpresh}).} presheaf with $\F(1_\bool{B})\neq\emptyset$, 
$R(\F):=\M_\F$ is the $\RO(\St(\bool{B}))$-valued model whose domain is $\F(1_\bool{B})$ and such that, for $f, g\in \F(1_\bool{B})$, 
\[
\Qp{f=g}^{\M_\F}:=\Reg{\bigcup\bp{N_b\in\CLOP(\St(\bool{B})): \F(b\leq 1_\bool{B})(f)=\F(b\leq1_\bool{B})(g)}}.
\]
\item For $(\alpha, i):\F\to\G$ a natural transformation of presheaves $\F:\bool{B}^{\textrm{op}}\to\dSet, \G:\bool{C}^{\textrm{op}}\to\dSet$ such that $\F(1_\bool{B})\neq\emptyset\neq\G(1_{\bool{C}})$, and $i:\bool{B}\to\bool{C}$, $R(\alpha, i)=(R(\alpha), i)$, with $R(\alpha):=\alpha_{1_\bool{B}}$, where $\alpha_{1_\bool{B}}:\F(1_\bool{B})\to\G(1_\bool{B})$ is the map defined by
$\alpha$ at $1_\bool{B}$.
\end{itemize}
\end{definition}
The following result corresponds, in our setting, to \cite[Proposition 2.3]{MONROBIS}.
\begin{theorem}\label{thm:adjbvmodpresh}
Let $\SPresh^{\textrm{cba}}(\dSet)$ be the subcategory of $\Presh^{\textrm{bool}}(\dSet)$ given by the subfamily of $J^{\mathrm{\sup},\bool{B}}$-separated presheaves on some $\bool{B}$ which is a complete boolean algebra.

The pair $\cp{L, R}$ is an adjuction between the categories
$\SPresh^{\textrm{cba}}(\dSet)$ and $\modxx$ with $L$ being the 
left adjoint.
\end{theorem}
\begin{proof}
We have only to find the unit $\Theta$ and the counit $\varepsilon$ of the adjunction and then to verify the identities 
\begin{equation}
\label{un:coun}
\Id_R=R\varepsilon\circ\Theta R\textit{ and }\Id_L=\varepsilon L\circ L\Theta.
\end{equation}
To define the unit $\Theta: \Id_{\modxx}\to R\circ L$, we simply take, for $\M\in\modxx$ a $\bool{B}$-valued model, $(\Theta_\M,\Id_{\bool{B}})$, where
\begin{align*}
\Theta_\M:&\M\to\M/_{F_1}.\\
&\tau\mapsto[\tau]_{F_1}
\end{align*}
The fact that $\Theta$ is a natural transformation is left to the reader. Notice that $\Theta_\M$ for $\M$ extensional is the identity morphism.\\
We need now to define the counit $\varepsilon$. Here we need to use crucially the assumption that we restrict 
$R$ to the family of \emph{separated} presheaves. Towards this aim  
if $b\leq c$, we have
that the following diagram is commutative with horizontal lines being bijections:
\begin{equation}\label{comdgm:counitLXRX}
\begin{CD}
(\F(c\leq1))[\F(1)] @>\tau\restriction c\mapsto[\tau]_{c}>>  R(\F)/_{F_c}\\
@V{\tau\restriction c\mapsto\tau\restriction b}VV              @VV{[\tau]_{F_c}\mapsto [\tau]_{F_b}}V\\
(\F(b\leq 1))[\F(1)]@>>\tau\restriction b\mapsto[\tau]_{b}>  R(\F)/_{F_b}\\
\end{CD}
\end{equation}
where, by our convention, $\F(b\leq 1)(\tau)$ is equally denoted as $\tau\restriction b$.\\
The commutativity is automatic by definition of the various maps occurring in the diagram.
To see why the horizontal lines are bijections, observe that
$[\tau]_{F_b}=[\sigma]_{F_b}$ for $\tau,\sigma\in\F(1)$ and $b\in\bool{B}$
if and only if $b\leq\Qp{\tau=\sigma}^{R(\F)}$ if and only if\footnote{The left to right implication of this equivalence uses that
$\F$ is separated, otherwise we can only assert that the set of $d\leq b$ such that $\F(d\leq1)(\sigma)= \F(d\leq 1)(\tau)$
is a dense cover of $b$, but $b$ may not belong to this dense cover if $\F$ is not separated.}
\[
\F(b\leq1)(\sigma)= \F(b\leq 1)(\tau).
\]
In particular we conclude that:
\begin{equation}\label{eqn:LXCIRCRX0}
((L\circ R)(\F))(b):=R(\F)/_{F_b}\cong \F(b\leq 1)[\F(1)]\subseteq\F(b),
\end{equation}
and
\begin{align}\label{eqn:LXCIRCRX1}
((L\circ R)(\F))(b\leq c)&=
\bp{[\tau]_{c}\mapsto [\tau]_{b}: \tau \in \F(1)}\cong\\ \nonumber
&\cong \F(b\leq c)\restriction (\F(b\leq 1)[\F(1)])\\ \nonumber
&\subseteq \F(b\leq c).
\end{align}
Now we let $\varepsilon$ be the natural transformation which in each component $\F:\bool{B}^{\textrm{op}}\to\dSet$ is the morphism $(\epsilon(\F), \Id_{\bool{B}})$, where $\epsilon(\F)$ is defined in the following way:
\begin{itemize}
\item for any $b\in\bool{B}$
\begin{align*}
\varepsilon(\F)(b):\F(1)/_{F_b}&\to\F(b\leq 1)[\F(1)]\\
 [\tau]_b&\mapsto(\tau\restriction b);
\end{align*} 
\item for any $b\leq c\in\bool{B}$,
$\varepsilon(\F)(b\leq c)$ transfers via the horizontal lines of diagram (\ref{comdgm:counitLXRX})
the mapping $\bp{[\tau]_c\mapsto[\tau]_b:\tau\in \F(1)}$ from $\F(1)/_{F_c}$ to $\F(1)/_{F_b}$ to the mapping $\F(b\leq c)\restriction \bp{\tau\restriction c: \tau\in \F(1)}$ (where the latter set is exactly $\F(c\leq 1)[\F(1)]$).
\end{itemize}
The morphism $\varepsilon_\F$ is a well defined natural transformation of $L\circ R$ with 
$\Id_{\SPresh^{\textrm{cba}}(\dSet)}$ due to the commutativity of diagram (\ref{comdgm:counitLXRX}).\\ 
Equations \eqref{un:coun} characterizing the unit and counit properties of $\Theta,\varepsilon$ with respect to $L,R$ 
are satisfied due to the following observations:
\begin{itemize}
\item 
The unit $\Theta$ restricted to the image of the functor $R$ is the identity: 
the target of $R_\bool{B}$ is the class of $\bool{B}$-valued extensional models. 
\item
The counit $\varepsilon$ restricted to the image of the functor 
$L$ is an isomorphism with the identity functor on the image of $L$\footnote{Here considering complete Boolean algebras is essential, otherwise $\F:\bool{B}^{\textrm{op}}\to\dSet$ and $(L\circ R)(\F):\RO(\bool{B})^{\textrm{op}}\to\dSet$ would be defined on differerent Boolean algebras, and thus they could not be isomorphic.}: for any $\bool{B}$-valued model $\M$,
$L(\M)$ is a presheaf $\F_\M:\bool{B}^{\textrm{op}}\to\dSet$ in which $\F_\M(b)=\M/_{F_b}$ is the surjective image of $\F_\M(1)=\M/_{F_1}$ via
$\F_\M(b\leq 1):[\tau]_{F_1}\mapsto[\tau]_{F_b}$; in particular for $\F_\M$ the last inclusions of (\ref{eqn:LXCIRCRX0}) and (\ref{eqn:LXCIRCRX1}) are reinforced to equalities, and we can conclude that $\varepsilon$ is an isomorphism appealing to the fact that the horizontal lines of the commutative diagram (\ref{comdgm:counitLXRX}) are bijections\footnote{Here we use essentially that $R$ is applied to 
a separated presheaf. Otherwise we could have different sections $f,g$ in 
$\F(1)$ such that $\Qp{f=g}_{R_\bool{B}(\F)}=1$; in which case $[f]_{F_1}=[g]_{F_1}$, hence the horizontal lines in (\ref{comdgm:counitLXRX}) are not anymore bijections. It is not clear then how to select canonically the representatives in the equivalence classes of $[-]_{F_1}$ so to maintain the requested naturality properties for the counit given by (\ref{eqn:LXCIRCRX0}) and (\ref{eqn:LXCIRCRX1}).}.
\end{itemize}
\end{proof}
\begin{corollary}\label{cor:adjbvmodpresh}
\label{eq:ex}
Let
\begin{itemize}
\item $\exPresh^{\textrm{cba}}(\dSet)$ be the full subcategory of $\SPresh^{\textrm{cba}}(\dSet)$ 
given by the separated presheaves 
which are determined by their global sections,
\item 
$\exmodxx$ be the full subcategory of $\modxx$ generated by extensional models.
\end{itemize}
The adjunction 
\[
\cp{L:\modxx\to\SPresh^{\textrm{cba}}( \dSet), R:\SPresh^{\textrm{cba}}(\dSet)\to\modxx}
\] 
specializes to an equivalence of categories between 
$\exmodxx$ and $\exPresh^{\textrm{cba}}(\dSet)$.
\end{corollary}
Note that every $\bool{B}$-valued model $\M$ is Boolean equivapent to the extensional 
model $\M/_{F_{1_{\bool{B}}}}$.
Hence the corollary amounts to say that $\bool{B}$-valued models  identify the class of separated
presheaves on $\bool{B}$ in 
which the local sections are always restrictions of global sections.

Note also that any sheaf $\F$ on $\bool{B}$ for the sup Grothendieck topology with $\F(1_\bool{B})\neq \emptyset$ is such that every local section is the restriction 
of a global section: for each $b\in\bool{B}$ let $f_b$ be some section in $\F(b)$;
given any local section $g$ in $\F(c)$ the family $\bp{g,f_{\neg c}}$ is matching, 
and a collation of this family is a global section whose restriction to $c$ is $g$.

\begin{corollary}
The equivalence of categories between $\exmodxx$ and $\exPresh^{\textrm{cba}}(\dSet)$ induces an equivalence of categories between the full subcategory of $\exmodxx$ generated by the models satisfying the mixing property and the full subcategory $\exSh^{\textrm{cba}}(\dSet)$ of $\exPresh^{\textrm{cba}}(\dSet)$ generated by the sheaves $\F:\bool{B}^{\textrm{op}}\to\dSet$ (according to the sup Grothendieck topology) such that $\F(1_\bool{B})\neq \emptyset$.
\end{corollary}

\begin{corollary}\label{char:Gamma0=Gamma1}
Let $\bool{B}$ be a complete Boolean algebra.
A $\bool{B}$-valued model $\M$ has the mixing property if and only if every global section of the \'etal\'e space 
$E_\M$ is a section induced by an element of $\M$.
\end{corollary}


\subsection{A characterization of the fullness property using \'etal\'e spaces}
\label{sec:charfulletalespace}

We can also give a nice characterization of the fullness property for Boolean valued models as in Definition \ref{def:fullness-wellbehaved} in sheaf theoretic terms, along the lines of what we did for the mixing property in Corollary \ref{char:Gamma0=Gamma1}. We will need the equivalent conditions for the fullness property presented in Theorem \ref{thm:charLosThm}.
\begin{remark}A well behaved $\bool{B}$-valued model $\M$ for the language $\LL$ is full if and only if, 
for every $\LL_\M$-formula $\nphi(x_1, \dots, x_n)$ with $n$ free variables $x_1, \dots, x_n$, there exists a finite number $m$ of $n$-tuples 
$\cp{\sigma_1^{(1)}, \dots, \sigma_n^{(1)}}, \dots, \cp{\sigma_1^{(m)}, \dots, \sigma_n^{(m)}}\in M^n$ such that
\[
\Qp{\exists x_1, \dots, \exists x_n\nphi(x_1, \dots, x_n)}=\bigvee_{i=1}^m\Qp{\nphi\cp{\sigma_1^{(i)}, \dots, \sigma_n^{(i)}}}.
\]
Indeed, assume $\M$ to be full. For sake of simplicity we assume $n=2$. Define $\psi(x_1):=\exists x_2\nphi(x_1, x_2)$. Since $\M$ is full, we can find $\sigma^{(1)}_1, \dots, \sigma_1^{(m)}$ such that 
\[
\Qp{\exists x_1\psi(x_1)}=\bigvee\bp{\Qp{\psi\cp{\sigma_1^{(i)}}}: i=1, \dots, m}.
\] 
Now define $\psi_i(x_2):=\nphi\cp{\sigma_1^{(i)}, x_2}$. Again, being $\M$ full, for every $i=1, \dots, m$ we can find $\sigma_2^{(i, 1)}, \sigma_2^{(i, m_i)}$ such that $\Qp{\exists x_2\psi_i(x_2)}=\bigvee\bp{\Qp{\psi_i\cp{\sigma_2^{(i, j)}}}: j=1, \dots, m_i}$. Putting all together,
\[
\Qp{\exists x_1\exists x_2\nphi(x_1, x_2)}=\bigvee_{i=1}^m\Qp{\exists x_2\cp{\sigma_1^{(i)}, x_2}}=\bigvee_{i=1}^m\bigvee_{j=1}^{m_i}\Qp{\nphi\cp{\sigma_1^{(1)}, \sigma_2^{(i, j)}}},
\]
and thus $\cp{\sigma_1^{(i)}, \sigma_2^{(i, j)}}$ for $j=1, \dots m_i$ and $i=1, \dots, m$ are the desired couples.

\end{remark}

Now fix an $\LL_\M$-formula $\nphi(x_1, \dots, x_n)$ with $n$-free variables and let:
\begin{itemize}
\item
 $b_\nphi:=\Qp{\exists x_1, \dots, x_n\nphi(x_1, \dots, x_n)}$.
 \item
 $D_\nphi:=\bp{N_c: \,\exists \sigma_1,\dots,\sigma_n\, 0_\bool{B}<c\leq 
 \Qp{\nphi(\sigma_1, \dots, \sigma_n)}}$
 \item
 $A_\nphi:=\bigcup D_\nphi=\bigcup\bp{N_c:0_\bool{B}<c\leq\Qp{\nphi(\sigma_1, \dots, \sigma_n)},\, \sigma_1,\dots,\sigma_n\in\mathcal{M}}$.
 \end{itemize}
  We define an \'etal\'e space $E^\nphi_\M$ over 
$N_{b_\nphi}$ 
by:
\[
E^\nphi_\M:=\bp{\cp{\sigma_1, \dots, \sigma_n, G}: \sigma_1, \dots, \sigma_n\in M, \Qp{\nphi(\sigma_1, \dots, \sigma_n)}\in G\in \St(\bool{B})}/_R,
\]
where $R$ is the equivalence relation such that $\cp{\sigma_1, \dots, \sigma_n, G}\mathrel{R}\cp{\tau_1, \dots, \tau_n, H}$ if and only if $G=H$ and $\Qp{\sigma_i=\tau_i}\in G$ for every $i=1, \dots, n$.\\
We can equivalently say that
\[
E^\nphi_\M=\bp{\cp{[\sigma_1]_G, \dots, [\sigma_n]_G}: \sigma_1, \dots, \sigma_n\in M, \Qp{\nphi(\sigma_1, \dots, \sigma_n)}\in G\in\St(\bool{B})}.
\]
This set is mapped into $N_{b_\nphi}$ by the function $p_\nphi:\cp{\sigma_1, \dots, \sigma_n, G}\mapsto G$ ( notice that, if there are some $\sigma_1, \dots, \sigma_n\in M$ such that $\Qp{\nphi(\sigma_1, \dots, \sigma_n)}\in G$, then $b_\nphi\in G$).\\
The important observation is that, in general, $p_\nphi$ is not surjective on $N_{b_\nphi}$; 
its image is just the dense open subset $A_\nphi$
of $N_{b_\nphi}$.\footnote{If $b_\nphi$ is a supremum but not a maximum of $\bp{\Qp{\nphi(\sigma_1, \dots, \sigma_n)}: \,\sigma_1,\dots,\sigma_n\in\M}$, there is 
$G\in N_{b_\nphi}\setminus A_\nphi;$
this $G$ is not in the target of $p_\nphi$.}  \\
The topology of $E^\nphi_\M$ is the one obtained as a subspace of $(E_\M)^n$: a base for the topology of $E^\nphi_\M$ is
\[
\mathcal{B}^\nphi:=\bp{\bp{\cp{\sigma_1, \dots, \sigma_n, G}: \Qp{\nphi(\sigma_1, \dots, \sigma_n)}\in G\in N_c}: \sigma_1, \dots, \sigma_n\in M, N_c\subseteq N_{b_\nphi}}.
\]
It is easy to check that $E^\nphi_\M$ equipped with this topology renders continuous the map $p_\nphi$. 
Moreover, $p_\nphi:E^\nphi_\M\to N_{b_\nphi}$ is the bundle (e.g induced by $\Lambda^{\RO(\bigcup D_\nphi)}$ and according to \cite[II]{MacLMoer}) associated to the \'etal\'e space of the presheaf $\G^\nphi_\M:D_{\nphi}^{\textrm{op}}\to\dSet$ defined as follows:
\[
\G^\nphi_\M(c):=\bp{\cp{[\sigma_1]_{F_c}, \dots, [\sigma_n]_{F_c}}: \sigma_1, \dots, \sigma_n\in M\textit{ and }c\leq\Qp{\nphi(\sigma_1, \dots, \sigma_n)}}.
\]
Notice by the way that
 each $\cp{[\sigma_1]_{F_c}, \dots, [\sigma_n]_{F_c}}\in\G^\nphi_\M(N_c)$ defines a continuous open section
\[
(\dot{\sigma}_1\times\dots\times\dot{\sigma}_n): G\mapsto\cp{[\sigma_1]_G, \dots, [\sigma_n]_G, G}
\]
 of $p_{\nphi}$ over $N_c$.\\ 
 We can now give a local characterization (i.e. formula by formula) of the fullness property.
\begin{definition}
Given a $\bool{B}$-valued model $\mathcal{M}$ for $\mathcal{L}$, an
$\mathcal{L}$-formula $\nphi$ is \emph{$\M$-full} if $A_\nphi=N_{b_\nphi}$.
\end{definition}
\begin{theorem}
\label{chr:ful}
Let $\bool{B}$ be a Boolean algebra and let $\M$ be a well behaved $\bool{B}$-valued model for the language $\LL$. The following are equivalent:
\begin{enumerate}
\item \label{chr:ful1} $\M$ is full;
\item\label{chr:ful2} $\nphi$ is $\M$-full
  for every $\LL_\M$-formula $\nphi$;
\item \label{chr:ful3} $A_\nphi$ is closed for every $\LL_\M$-formula $\nphi$;
\item\label{chr:ful4}  for every $\LL_\M$-formula $\nphi$ such that $E^\nphi_\M$ is non-empty, $E^\nphi_\M$ has at least one 
global section.
\end{enumerate}
Moreover if $\M$ has the $<\omega$-mixing property the above are also equivalent to
\begin{enumerate}
\setcounter{enumi}{4}
\item\label{chr:ful5} for every $\LL_\M$-formula $\nphi(x_1, \dots, x_n)$ such that $E^\nphi_\M$ is non-empty, $E^\nphi_\M$ has at least one 
global section of the form $\dot{\sigma}_1\times\dots\times\dot{\sigma}_n$.
\end{enumerate}
\end{theorem}
\begin{proof}
\ref{chr:ful1} implies \ref{chr:ful2}. Indeed, If $\M$ is full, for every $\LL_\M$-formula $\nphi(x_1, \dots, x_n)$ there exist $\sigma_1^{(1)}, \dots, \sigma_n^{(1)}, \dots, \sigma_1^{(m)}, \dots, \sigma_n^{(m)}\in M$ such that
\[
\bigvee_{i=1}^n\Qp{\nphi(\sigma_1^{(i)}, \dots, \sigma_n^{(i)})}=
\Qp{\exists x_1, \dots, x_n\,\nphi(x_1, \dots, x_n)},\quad\text{ i. e. }\quad N_{b_\nphi}=\bigcup_{i=1}^m\,N_{\Qp{\nphi(\sigma_1^{(i)}, \dots, \sigma_n^{(i)})}}\subseteq A_\nphi.
\]
Conversely, assume \ref{chr:ful2}. This means that
\[
A_\nphi=\bigcup_{\sigma_1, \dots, \sigma_n\in\M}N_{\Qp{\nphi(\sigma_1, \dots, \sigma_n)}}=N_{b_\nphi}.
\]
Thus $\{N_{\Qp{\nphi(\sigma_1, \dots, \sigma_n)}}:\sigma_1, \dots, \sigma_n\in\M\}$ is an open cover of the compact space $N_{b_\nphi}$. By compactness, \ref{chr:ful1} holds.\\
The equivalence between \ref{chr:ful2} and \ref{chr:ful3} is immediate. Now we prove that \ref{chr:ful1} implies \ref{chr:ful4}. We have that
\[
N_{b_\nphi}=\bigcup_{i=1}^m\,N_{\Qp{\nphi(\sigma_1^{(i)}, \dots, \sigma_n^{(i)})}}.
\]
Consider
\[
C_1:=N_{\Qp{\nphi(\sigma_1^{(1)}, \dots, \sigma_n^{(1)})}},\qquad C_i:=N_{\Qp{\nphi(\sigma_1^{(i)}, \dots, \sigma_n^{(i)})}}\setminus\bigcup_{j=1}^{i-1}C_j, \quad i=2, \dots, m.
\]
The sets just defined are clopen, pairwise disjoint and they cover $N_{b_\nphi}$. Define $s:N_{b_\nphi}\to E^\nphi_\M$ by
\[s\restriction C_i:=\dot{\sigma}^{(i)}_1\times\dots\times\dot{\sigma}^{(i)}_n.
\]
Thus $s$ is continuous and it is a global section of $p_\nphi$.\\
Finally, assume $s$ is a section as in \ref{chr:ful4}, which means that $p_\nphi\circ s=\Id_{N_{b_\nphi}}$. In particular, $p_\nphi$ has to be surjective, and so \ref{chr:ful2} holds.\\
The missing details are left to the reader.
\end{proof}
The mixing property is strictly stronger than the fullness property. Here is a reformulation of Proposition \ref{ful:mod} in the language of bundles and sheaves:
\begin{corollary}
Assume $\bool{B}$ is a complete Boolean algebra and $\M$ is a $\bool{B}$-valued model with the mixing property. 
Then, for every $\LL_\M$-formula $\nphi(x_1, \dots, x_n)$, $\G^\nphi_\M$ is a sheaf for the sup Grothendieck topology. 
In particular, if $\M$ has the mixing property, every local section of $E^\nphi_\M$ can be extended to a global one. Moreover, each global section is induced by an $n$-tuple of elements of $\M$.
\end{corollary}

\section{Sheafification of presheaves on complete Boolean algebras according to the sup Grothendieck topology}
\label{sec:sheafposet}

The aim of this section is to provide a nice topological description of the sheafification (and also of the separification) process according to the sup Grothendieck topology for presheaves on a complete boolean algebra. This topological description will be performed by explicitly exhibiting an \'etal\'e bundle on $\St(\bool{B})$ associated to a given presheaf on a complete boolean algebra $\bool{B}$, and by taking the family of continuous sections on it. Note that the usual sheafification for the sup topology on frames which are not the open sets of a topological space is given by the double application of the operator $\F\mapsto\F^+$ (see \cite[Chapter III]{MacLMoer}) and does not give an explicit and natural \'etal\'e bundle representation of such sheaves. We use peculiar features of complete boolean algebras $\bool{B}$ to show that the sought \'etal\'e bundle representation exists for sheaves on $\bool{B}$ for the sup topology (specifically on the topological side we use that $\St(\bool{B})$ is the Stone-Cech compactification of any of its dense subsets, on the lattice thoretic side we use that a downward closed subset $D$ of the preorder $\bool{B}^+$ has the same supremum of any of its maximal antichains $A\subseteq D$).



\begin{notation}
Given $\bool{B}$ a complete Boolean algebra, we denote with $\Presh(\bool{B})$ the category of presheaves $\F:\bool{B}^{\mathrm{op}}\to\dSet$ such that $\F(1_{\bool{B}})\neq \emptyset$ and natural transformations between them, and with $\Sh_\mathrm{sup}(\bool{B})$ the full subcategory of $\Presh(\bool{B})$ generated by the presheaves which are sheaves for the supremum Grothendieck topology on $\bool{B}$.
\end{notation}
For the rest of this section we will consider only presheaves on complete Boolean algebras and our focus will be on the notion of sheaf for the supremum Grothendieck topology. In some occasions we will compare this notion to the notion of topological sheaf, as in the following:

\begin{example}\label{ex:sheaf}
 Let 
 $(X,\tau)$ be compact Hausdorff extremally disconnected,  and 
 $\C(X)$ be the set of continuous complex valued functions on $X$.
 $\C(X)$ is  the set of global sections of the \emph{topological sheaf} $\F$ of continuous complex valued functions defined on an open subset of $X$. 
 Notice however that, even though $\F$ is a topological sheaf on the topological space $X$, when restricted to $\bool{B}=\RO(X)$ it is not anymore a sheaf for the dense Grothendieck topology.
 
To see this, let $\bp{U_n:n\in\mathbb{Z}}$ be a maximal antichain of $\RO(X)$ and let $f_n\in \F(U_n)$ be constant with value $n$. Then $\bp{f_n:n\in\mathbb{Z}}$ is a matching family of $\F\restriction \RO(X)$
which cannot be collated by any element of $\F(X)$. 

 
Consider now the set 
\[
\C^+(X):=\bp{f:X\to\beta(\mathbb{C}): f \textit{ is continuous and }f^{-1}[\beta(\mathbb{C})\setminus\mathbb{C}]\textit{ is nowhere dense}}.
\] 
It turns out that $\C^+(X)$ defines the global sections of the sheafification (according to the dense Grothendieck topology) of the presheaf $\C(X)\restriction\bool{B}$:
if $f=\bigcup_{n\in\mathbb{Z}}f_n$, then its unique continuous extension
$\beta(f):X\to\beta(\mathbb{C})$ maps any $G$ outside of $\bigcup_{n\in\mathbb{Z}} U_n$ to some point in $\beta(\mathbb{C})\setminus\mathbb{C}$ and it is the collation of the matching family $\bp{f_n:n\in\mathbb{Z}}$.

More generally, if a continuous $s:X\to \beta(\mathbb{C})$ takes values outside of $\mathbb{C}$ on a nowhere dense set, then we can find a family
$\{V_i: i\in I\}$ of (disjoint) regular open subsets of $X$ such that $\bigcup_{i\in I}V_i$ is open dense in $X$ and $s_i=s\restriction V_i\in\C(V_i)$ has range in $\mathbb{C}$. Then $\bigcup_{i\in I}s_i\in\C(\bigcup_{i\in I}V_i)$ and thus $s\restriction\bigcup_{i\in I}V_i=\bigcup_{i\in I}s_i$. Therefore  $s$ is the unique continuous extension $\beta(\bigcup_{i\in I}s_i):X=\beta(\bigcup_{i\in I}V_i)\to\beta(\mathbb{C})$, and thus it is the collation in $\C^+(X)$ of the matching family $\{s_i: i\in I\}$ for the presheaf $\F\restriction \RO(X)$.

In model theoretic terms $\C(X)$ (seen as a $\RO(X)$-valued model) does not have the mixing property (see Proposition \ref{chr:mix}), while $\C^+(X)$ does. \cite{VIAVAC15} gives a detailed analysis of the Boolean valued model structure of $\C(X)$ and $\C^+(X)$.
Other connections between functional analysis, sheaf theory, and set theory will be given in Example \ref{exm:sheafLR}.
\end{example}

We want to develop a machinery which includes this example as a special case and that gives a concrete presentation of the sheafification process of a given presheaf on a complete boolean algebra according to the sup Grothendieck topology by describing it in terms of local sections of a suitable etal\'e bundle (as is done in \cite[Section II.5]{MacLMoer} for the notion of topological sheaf).

We will outline the similarities between our approach to produce the sheafification of a presheaf according to the sup Grothendieck topology and the one of \cite[Chapter II]{MacLMoer}; for this reason we will adopt a terminology similar to that of \cite[Chapter II]{MacLMoer} in several steps, with the provision to attach the apex $1$ when referring to our construction, in case confusion may arise.

\subsection{The \'etal\'e bundle of a presheaf on a complete boolean algebra}

%
%

  \begin{definition}
  \label{def:etbund1}
Let $\bool{B}$ be a complete boolean algebra.
Given a
 presheaf $\F:\bool{B}^{\textrm{op}}\to\dSet$, we set:
\[
\Lambda^1_\F:=\coprod_{G\in\St(\bool{B})}\F_G
\]
where
$\F_G=\bp{[f]_{G}: f\in\F(b), b\in G}$ and
$[f]_{G}$ is the equivalence class induced by $f\equiv_{G} g$ if and only if $\bigvee\{d\in \bool{B}:f\restriction d=g\restriction d\}\in G$\footnote{Since $\bool{B}$ is complete, this is equivalent to require the existence of $b\in G$ such that $\{d\in \bool{B}^+:f\restriction d=g\restriction d\}$ is dense below $b$. The fact that $\equiv_G$ is an equivalence relation is an easy exercise.}
We define the (surjective) map 
\[
\begin{aligned}
\pi_\F:& \Lambda^1_\F\to \St(\bool{B})\\ 
&[f]_{G}\mapsto G
\end{aligned}.
\]
Each $f\in\F(b)$ determines a map
\[
\begin{aligned}
\dot{f}:&N_b\to\Lambda_\F\\
&G\mapsto[f]_{G}
\end{aligned},
\]
which is injective: if $F\neq G$, $[f]_{F}\neq[f]_{ G}$, since $\F_F,\F_G$ are disjoint.

$\Lambda^1_\F$ is topologized by the topology $\sigma_\F$ generated by the (semi)base\footnote{The topology generated by this semibase gives that this semibase is a base, see below.}  
\[
\mathcal{B}_\F:=\bp{\dot{f}[N_b]: f\in\F(b), \, b\in\bool{B}^+}.
\]
\end{definition}

\begin{remark}
We remark that if $(X,\tau)$ is compact, extremally disconnected, Hausdorff, 
and $\F:\RO(X,\tau)^{\mathrm{op}}\to\mathcal{C}$ is a presheaf, it makes sense to apply our definition to $\F$ and to apply the definition of \cite[Section II.5]{MacLMoer} to $\F^*$ (as defined in Notation \ref{not:extRO(X)toO(X)}). It can be checked that in this specific case $\Lambda^1_\F$ and $\Lambda_{\F^*}$ (as defined in \cite[Section II.5]{MacLMoer}) are exactly the same bundle: the points in $X$ bijectively corresponds to the points of $\St(\RO(X))$ via Stone duality. Furthermore one can check that --- modulo the identification of (the points of) $\Lambda^1_\F$ and $\Lambda_{\F^*}$ ---  the topology we endow $\Lambda^1_\F$ with is exactly the same topology \cite[Section II.5]{MacLMoer} endow $\Lambda_{\F^*}$ with.

In particular, consider a complete Boolean algebra $\bool{B}$ and a $\bool{B}$-valued model $\M$. Then the \'etal\'e space $\Lambda^1_{\F_\M}$ (obtained from the presheaf $\F_\M:\bool{B}^{\textrm{op}}\to\dSet$ defined in Section \ref{sec:pres}) coincides (up to homeomorphism) with the \'etal\'e space $E_\M=\Lambda_{\F^*_\M}$ analyzed in Section \ref{sec:pres} and obtained from the topological presheaf $\F^*_\M$ on (the open sets of) $\St(\bool{B})$ defined there.

\end{remark}


From now on 
all our assertions about separated presheaves or sheaves subsume that we are referring to these concepts as instantiated for the sup Grothendieck topology on a complete boolean algebra.

\begin{lemma}
Let $\bool{B}$ be a complete Boolean algebra.
For every $\F\in\Presh(\bool{B})$, $\mathcal{B}_\F$ is a base for a  topology $\sigma_\F$ on $\Lambda^1_\F$ which makes $(\pi_\F,\Lambda^1_\F,\sigma_\F)$ an \'etal\'e bundle over $\St(\bool{B})$.
\end{lemma}

\begin{proof}
We only show that $\mathcal{B}_\F$ is a base, since from this fact it is almost trivial to prove that $(\pi_\F,\Lambda_\F,\sigma_\F)$ is an \'etal\'e space.

To this end, let $f_1\in\F(b_1),\dots, f_n\in\F(b_n)$ be such that 
$\dot{f}_1[N_{b_1}]\cap\dots\cap\dot{f}_n[N_{b_n}]\neq\emptyset$.
This occurs only there is some
$G\in\St(\bool{B})$ with $b_1,\dots,b_n\in G$ and $[f_i]_{G}=[f_j]_{G}$ for all $i\neq j$. 
This gives that for every $i\neq j$ there is some $c_{ij}\in G$ such that
$D_{ij}:=\{q\leq c_{ij}: f_i\restriction q=f_j\restriction q\}$ is dense below $c_{ij}$.
Now, since $G$ is a filter, there is some $c\in G$ refining $c_{ij}\in G$ for all $i$ and $j$. Then $D:=\bigcap_{i\neq j}D_{ij}\cap\downarrow c$ is a dense subset of $\downarrow c$ such that, for every $d\in D$, $f_j\restriction d=f_i\restriction d$ for all $i,j$.
This gives that, for every $H\in N_c$, we have
\[
[f_1]_{H}=\dots=[f_n]_{H}.
\]
Thus, if we call $V$ the set $\dot{f}_1[N_c]=\dots=\dot{f}_n[N_c]$, then $V\subseteq\bigcap_{j=1}^n\dot{f}_j[N_ {b_j}]$
is an open neighborhood of $[f_1]_{G}=\dots=[f_n]_{G}$; therefore $\bigcap_{j=1}^n\dot{f}_j[N_ {b_j}]$ can be covered by basic open sets.
\end{proof}

\begin{theorem}
\label{fact:lambdahaus} \label{prop:lambdahaus}
Let $\bool{B}$ be a complete Boolean algebra and $\F\in\Presh(\bool{B})$.
The topology $\sigma_\F$ on $\Lambda^1_\F$ is Hausdorff,  locally compact, and extremally disconnected.
 \end{theorem}

 \begin{proof}
\emph{}

\begin{description}
\item[$\Lambda^1_\F$ has the Hausdorff property]
The only non trivial case occurs when the two points we need to separate $[f]_{G}\neq[g]_{H}$ are on the same stalk; we need to show that these two points can be separated by basic open sets. This case occurs when $G=H$ and $f\in\F(b_f)$, $g\in\F(b_g)$ with $N_{b_f},N_{b_g}\in G$. Let
 \[
 N_b:=\bigvee_{\bool{B}}\{N_q: f\restriction q=g\restriction q\}.
 \]
 We observe that $N_b\notin G$: otherwise \[N_c=N_b\cap N_{b_f}\cap N_{b_g}=\bigvee_{\bool{B}}\{N_q: f\restriction q=g\restriction q, q\leq p_f, p_g\}
 \]
 would be in $G$. This would give that $\bp{q\leq N_c: f\restriction q=g\restriction q}$ is dense below $N_c$, yielding $[f]_{G}=[g]_{G}$, a contradiction.
 Now observe that if $H\in N_{\neg b}$, $[f]_{H}\neq [g]_{H}$ as 
 $ f\restriction q=g\restriction q$ holds for no $q\leq \neg b$. We conclude that $\dot{f}[N_{\neg b}]$ and $\dot{g}[N_{\neg b}]$ separate $\dot{f}(G)$ from $\dot{g}(G)$. 
%
%
%
%
%

\item[$\Lambda^1_\F$ is locally compact] 
We note that for $s:N_b\to\Lambda_\F$ local section with open range, $s[N_b]$ is a compact and open subspace of $\Lambda_\F$, being it homeomorphic to the compact Hausdorff space $N_b$.
Being $\Lambda_\F$ Hausdorff, its compact subsets are closed, hence the topology of $\Lambda_\F$ has a base of clopen compact sets and is therefore $0$-dimensional and locally compact. 

\item[$\Lambda^1_\F$ is extremally disconnected] It suffices to prove that the closure of an open set is open.

It is convenient to adopt this piece of notation for $U\subseteq X\subseteq\Lambda_\F$: $\mathrm{Cl}_X(U)$ denotes the closure of $U$ as a subset of $X$ relative to the subspace topology inherited by $X$ from $\Lambda_\F$. Note that $\mathrm{Cl}_X(U)=\mathrm{Cl}_{\Lambda_\F}(U)\cap X$ and that  $\mathrm{Cl}_X(U)=\mathrm{Cl}_{\Lambda_\F}(U)$ when $X$ is a closed subset of $\Lambda_\F$.

Assume $A$ is open for $\Lambda_\F$ and $[f]_G\in\mathrm{Cl}_{\Lambda_\F}(A)$, we show that $[f]_G$ admits an open neighborhood totally contained in $\mathrm{Cl}_{\Lambda_\F}(A)$. 

This will give that $\mathrm{Cl}_{\Lambda_\F}(A)$ is open as well. 

Towards this aim note that for some $b\in G$ we have that $[f]_G\in \dot{f}[N_b]$ and that $\dot{f}$ implements an homeomorphism of $N_b$ with $\dot{f}[N_b]$ which is a clopen subspace of $\Lambda_\F$.

We therefore get that 
\[
Y=\mathrm{Cl}_{\dot{f}[N_b]}(A\cap \dot{f}[N_b])=\mathrm{Cl}_{\Lambda_\F}(A\cap \dot{f}[N_b])
\]
and that $[f]_G\in Y$.

Since $\dot{f}[N_b]$ with subspace topology is homeomorphic to $N_b$ -- which is extremally disconnected -- and $A\cap \dot{f}[N_b]$ is open in this subspace, we get that $Y$ is an open subset of $\dot{f}[N_b]$ in the subspace topology. Since $\dot{f}[N_b]$ is open in $\Lambda_\F$, we get that $Y$ is an open subset of $\Lambda_\F$. Clearly $Y$ is totally contained in $\mathrm{Cl}_{\Lambda_\F}(A)$. 
The proof is completed.
\end{description}
\end{proof}

It is convenient for future uses to characterize the \'etal\'e bundles on $\St(\bool{B})$ which have the above topological properties.

\begin{definition}\label{def:strsepetsp}
Let $(X,\tau)$ be a 
topological space and $(\pi,E,\sigma)$ be an \'etal\'e  bundle on $(X,\tau)$.

$(\pi,E,\sigma)$ is \emph{strongly separated} over $(X,\tau)$ if
for every two local sections $s,t:U\to E$ defined on some open set $U\in \tau$
\[
E_{s,t}=\bp{x\in U:s(u)=t(u)}
\] 
is closed in $U$ for its  subspace topology inherited from $\tau$.
\end{definition}

This fact is almost self-evident, but is -- in any case -- a corollary of the next Lemma:
\begin{fact}
Assume $\bool{B}$ is a complete boolean algebra and $\F$ is a presheaf on $\bool{B}$.
Then $(\pi_\F,\Lambda^1_\F,\sigma_\F)$ is a strongly separated \'etal\'e bundle over 
$(\St(\bool{B}),\tau_\bool{B})$.
\end{fact}

\begin{lemma}\label{lem:strsephauloccompextrdisc}
Let $(X,\tau)$ be an Hausdorff topological space and $(\pi,E,\sigma)$ be an \'etal\'e bundle over $(X,\tau)$.

The following are equivalent:
\begin{enumerate}
\item  \label{lem:strsephauloccompextrdisc-1}
$(\pi,E,\sigma)$ is a strongly separated \'etal\'e bundle;
\item \label{lem:strsephauloccompextrdisc-2}
$(E,\sigma)$ is Hausdorff.
\end{enumerate}
If any of the above equivalent conditions is met by $(\pi,E,\sigma)$, and $(X,\tau)$ is also compact, then $(E,\sigma)$ is locally compact, and if $(X,\tau)$ is also  extremally disconnected, then so is $(E,\sigma)$.
%
\end{lemma}

Note that (according to the notation introduced in Def. \ref{def:strsepetsp}) the very essence of an \'etal\'e space ensures that $E_{s,t}=\bp{x\in U:s(u)=t(u)}$ is always open; hence the Hausdorff property for an \'etal\'e space over an Hausdorff space is strictly related to topological disconnection.

\begin{proof}
\emph{}

\begin{description}
\item[\ref{lem:strsephauloccompextrdisc-1} implies \ref{lem:strsephauloccompextrdisc-2}]
Points that lies on different $\pi$-fibers over $X$ can be separated by the preimages under $\pi$ of disjoint open neighborhoods of their $\pi$-images (which can be found since $(X,\tau)$ is Hausdorff). Now we show that we can also separate points on the same $\pi$-fiber.
Let  $s,t:U\to E$ be continuous section with $s(x)\neq t(x)$ for some $x\in U$ clopen subset of $X$.  Since $(\pi,E,\sigma)$ is strongly separated the set $V$ of $y\in U$ such that
$s(y)\neq t(y)$ is open in $U$ and $x$ is is in this set. Since $(\pi,E,\sigma)$ is an \'etal\'e bundle over $(X,\tau)$, $s,t$ are both local homeomorphisms with their images and are open maps. Hence we get that $s[V]\cap t[V]$ is empty and that both are open subsets of $(E,\sigma)$. These two sets show that $s(x),t(x)$ can be separated by disjoint open neighboorhoods.

\item[\ref{lem:strsephauloccompextrdisc-2} implies \ref{lem:strsephauloccompextrdisc-1}]
 \ref{lem:strsephauloccompextrdisc-2} ensures that for local sections $s,t:U\to E$ defined on some open set $U$
\[
E_{s,t}=\bp{y\in U: s(y)=t(y)}
\]
is closed in the subspace topology inherited by $U$, since $U,E$ are Hausdorff and $s,t$ are continuous. 
The conclusion follows.
\end{description}

The proof that when $(\pi,E,\sigma)$ is strongly separated and $(X,\tau)$ is also  extremally disconnected (respectively compact), then $(E,\sigma)$ is also extremally disconnected  (respectively locally compact) is left to the reader and proceeds along the same lines of the one given for $\Lambda^1_\F$.
\end{proof}

\subsection{Separafication of presheaves on complete boolean algebras in terms of  local sections}

We use $\Lambda^1_\F$ to produce the separification of a presheaf $\F$.
\begin{definition}
\label{def:functxi}
Let $\bool{B}$ be a complete Boolean algebra. 
 $\Gamma^{0,\bool{B}}$ is the functor $\bundle(\bool{B})\to\Presh(\bool{B})$ defined in the following way: 
\begin{itemize}
\item
On objects: it sends a bundle $(\pi,E,\sigma)$ to the presheaf $\Gamma^0_{(\pi,E,\sigma)}:\bool{B}^\mathrm{op}\to\dSet$ such that 
\[\Gamma^0_{(\pi,E,\sigma)}(b)=\{s:N_b\to E: s\text{ is a local section for }(\pi,E,\sigma)\}\] 
with $\Gamma^0_{(\pi,E,\sigma)}(c\leq b)$ the obvious restriction map $s\mapsto s\restriction N_c$.
\item
On morphisms: given a morphism $\psi:E_1\to E_2$, $\Gamma^0(\psi)$ maps $f\in\Gamma^0_{(\pi_1,E_1,\sigma_1)}(b)$ to $\psi\circ f\in\Gamma^0_{(\pi_2,E_2,\sigma_2)}$.
\end{itemize}
\end{definition}\label{def:nattrxi}
We leave to the reader to check that $\Gamma^{0,\bool{B}}$ is well defined and functorial.
\begin{fact}
\label{fact:locsec}
If $\F$ is a presheaf on $\bool{B}$, for every local section $s:N_b\to \Lambda^1_\F$ we can find $D\subseteq \bool{B}$ such that $\bigvee D=b$ and, for every $d\in D$, $s\restriction N_d=\dot{f}$ for some $f\in\F(d)$.
\end{fact}
\begin{proof}
It is enough to notice that, if $G\in N_b$, then a basic open neighborhood of $s(G)$ is of the form $\dot{f}[N_c]$ for some $G\ni N_c\subseteq N_b$. Indeed, by continuity of $s$, there exists $N_d\subseteq N_c$ such that $s[N_d]\subseteq \dot{f}[N_c]$, proving that $s(G)=[f]_G=\dot{f}(G)$ for every $G\in N_d$.
\end{proof}
\begin{definition}\label{rem:isoJ-sep}
Let $\xi^{\F}:\F\to\Xi_\F:=\Gamma^0_{(\pi_\F,\Lambda^1_\F,\sigma_\F)}$ be defined by $\xi^{\F}_b(f)=\dot{f}$ for every $f\in\F(b)$, $b\in \bool{B}$.
\end{definition}
We leave to the reader to check that $\xi^{\F}:\F\to\Xi_\F$ is a natural transformation  acting in the obvious way.

$\bp{\xi^\F}_{\F\in\Presh(\bool{B})}$ separates presheaves, but in general does not sheafify them.
\begin{proposition}\label{prop:univJsep}
For every $\F\in\Presh(\bool{B})$, the presheaf $\Xi_\F$ is separated. Moreover, for every separated presheaf $\G\in\Presh(\bool{B})$, every natural transformation $\delta=\{\delta_b\}_{b\in \bool{B}}:\F\to\G$ factors through $\xi^{\F}$. In particular $\xi^\F$ is an isomorphism if and only if $\F$ is separated.

\end{proposition}
\begin{proof}
To see that $\Xi_\F$ is separated, notice that if $\bigvee D=b$ below $b$ and $f, g\in\F(b)$ are both collations over $b$ of a matching family $\{f_d\in\F(d): d\in D\}$, then for every $G\in N_b$ we have that $[f]_{G}=[g]_{G}$. Hence $\dot{f}$ and $\dot{g}$ are the same object, which is the collation of the matching family $\{\dot{f_d}\in\Xi_\F(d): d\in D\}$.

To prove the universal property, let $\delta=\{\delta_b\}_{b\in \bool{B}}:\F\to\G$ with $\G$ separated. We define $\mu=\{\mu_b\}_{b\in\bool{B}}:\Xi_\F\to\G$ by letting $\mu_b(\dot{f}):=\delta_b(f)$. Notice that it is well defined: indeed, if $f, g\in\F(b)$ coincide on some dense $D\subseteq\downarrow b$ (and thus $\dot{f}=\dot{g}$), then $\delta_b(f)=\delta_b(g)$, since both are the collation of the family $\{\delta_d(f\restriction d): d\in D\}=\{\delta_d(g\restriction d): d\in D\}$ and $\G$ is separated. 

The fact that $\delta_b=\mu_b\circ\xi^{\F}_b$ for every $b\in\bool{B}$ is then straightforward.
\end{proof}
In other words, $\F\mapsto \Xi_\F$ (provided we explain how it acts on the arrows of $\Presh(\bool{B})$, which we will do at a later stage) is up to isomorphism the \emph{separafication functor} for the sup Grothendieck topology on $\bool{B}$ (see for instance \cite[page 43]{VistNotes}). In particular the operator $\F\mapsto\Xi_\F$ differs from the operator $\F\mapsto\F^+$  defined in \cite[Section III.5]{MacLMoer}. Indeed there are cases in which $\F^+$ is a sheaf and $\F$ is not, while $\Xi_\F$ is a sheaf if and only if $\F$ is. The universality properties of the operator $\Xi$ grants that the operator $\F\mapsto\F^+$ factors through $\{\xi^\F\}_{\F\in\Presh(\bool{B})}$.

\subsection{Sheafification of presheaves on complete boolean algebras in terms of Stone sections}

We now develop  the sheafification operator for the sup topology on complete boolean algebras via a presentation of it factoring through \'etal\'e bundles on a compact extremally disconnected Hausdorff space.

The following is essential:

\begin{proposition}
\label{prop:betaEetalebundle}
Assume $(\pi,E,\sigma)$ is a bundle on a compact Hausdorff space $(X,\tau)$. Then:
\begin{enumerate}
\item $(\beta(\pi),\beta(E),\beta(\sigma))$ is still a bundle on $(X,\tau)$;
\item if $(E,\sigma)$ is Tychonoff, $(\pi,E,\sigma)$ is obtained from the restriction to $E$ of $\beta(\pi)$;
\item if $(X, \tau)$ is extremally disconnected and $(\pi, E, \sigma)$ is a Tychonoff strongly separated \'etal\'e bundle, then $(\beta(\pi), \beta(E), \beta(\sigma))$ is \'etal\'e.
\end{enumerate}
\end{proposition}
\begin{proof}
The first two item are almost self-evident. For the third one, remember that, being $(\pi, E, \sigma)$ a strongly separated \'etal\'e bundle and $(X, \tau)$ extremally disconnected, by Lemma \ref{lem:strsephauloccompextrdisc} $(E, \sigma)$ is extremally disconnected. Thus by Corollary \ref{corol:betazextrdisc} $(\beta(\pi), \beta(E), \beta(\sigma))$ is extremally disconnected and by Fact \ref{fac:StCe-iso} its regular open sets are of the form $\beta(U)$ for $U\in\RO(E, \sigma)$. Hence, we have that a base for $\beta(\sigma)$ is
\[
\begin{aligned}
\RO(\beta(E), \beta(\sigma))&=\{\beta(U): U\in\RO(E, \sigma)\}\\
&=\{\beta[s[V]]: s\textit{ local section of $(\pi, E, \sigma)$}, V\in\RO(X,\tau)\}\\
&=\{\beta(s)[V]: s\textit{ local section of $(\pi, E, \sigma)$}, V\in\RO(X,\tau)\}
\end{aligned}
\]
proving that $(\beta(\pi), \beta(E), \beta(\sigma))$ is an \'etal\'e bundle over $(X, \tau)$.
\end{proof}

The following is the cornerstone on which our construction of the sheafification operator relies:

 \begin{proposition}\label{prop:existgluingHaus}
 Let $\bool{B}$ be a complete Boolean algebra, and 
 $(\pi,E,\sigma)$ be a bundle over $(\St(\bool{B}),\tau_\bool{B})$ with $(E,\sigma)$ Tychonoff.  

Then every matching family 
 $\bp{s_U:  U\in D}$ of open sections $s_U:U\to E$ indexed by some $Z\subseteq \tau_\bool{B}$ has exactly one gluing 
 $s:\Reg{\bigcup Z}\to \beta(E)$ which is continuous.
 \end{proposition}
 \begin{proof}
Let $s_0=\bigcup\bp{s_U:U\in D}$ and $V=\bigcup Z$ be the dense open subset  of $\Reg{V}=\Cl{V}$ on which $s_0$ is defined. Note that $\Reg{V}$ is also extremally disconnected compact Hausdorff (with its inherited subspace topology).
Since $E\subseteq\beta(E)$ and the latter is compact Hausdorff, by Proposition \ref{comp:exdisc} applied to $\Reg{V}$, $s_0$ can be uniquely extended to a continuous $s:\Reg{V}\to\beta(E)$.
\end{proof}

This leads naturally to the following definition:

\begin{definition}
Let $(X,\tau)$ be a compact extremally disconnected Hausdorff space.

Given a bundle $(\pi,E,\sigma)$ over $(X,\tau)$ with $(E,\sigma)$ Tychonoff, 
a \emph{Stone section for $E$} on some $U\in\RO(\tau)$ is a continuous map $s:U\to \beta(E)$ such that:
\begin{itemize}
\item $s$ is a continuous section for $\beta(\pi)$,
\item for some dense open $A\subseteq U$ we have
 $s\restriction A$ is a section for $\pi$.
\end{itemize}
We say that $(\pi,E,\sigma)$ is \emph{complete for gluings} if, for every $U\in\RO(\tau)$, every Stone section 
\(
s:U\to \beta(E)
\) takes all its values in $E$.
\end{definition}

We can now define the sheaf we want to associate to $\Lambda^1_\F$. 

\begin{definition}
\label{def:gammapi}
 Let $\bool{B}$ be a complete Boolean algebra and
$(\pi,E,\sigma)$ be a bundle on $\St(\bool{B})$ with $(E,\sigma)$ a Tychonoff space.

Define $\Gamma^1_{(\pi,E,\sigma)}:\bool{B}^\mathrm{op}\to\dSet$ in the following way:
\begin{description}
\item[On objects] For $b\in\bool{B}$
\[
\Gamma^1_{(\pi,E,\sigma)}(b):=\{s:N_b\to \beta(E): s \textit{ is a Stone section for $E$}\};
\]
\item[On morphisms]
for $c\leq b$, $\Gamma^1_{(\pi,E,\sigma)}(c\leq b):s\mapsto s\restriction N_c$.

\end{description}
\end{definition}

\begin{remark}\label{rmrk:xitogamma}
Let $\bool{B}$ be a complete Boolean algebra and let $(\pi,E,\sigma)$ be an Hausdorff \'etal\'e bundle on $\St(\bool{B})$. Clearly the presheaf $\Gamma^0_{(\pi,E,\sigma)}$ of Definition \ref{def:functxi} is a subpresheaf of $\Gamma^1_{(\pi, E, \sigma)}$. Hence, the natural transformation $\xi^\F$ of Definition \ref{rem:isoJ-sep} can be regarded as a natural transformation $\F\to\Gamma^1_{(\pi_\F, \Lambda^1_\F,\sigma_\F)}$.
\end{remark}
 \begin{theorem}\label{mainthm:sheaf}
Let $\bool{B}$ be a complete Boolean algebra 
and $(\pi, E, \sigma)$ be a bundle on $\St(\bool{B})$ with $E$ a Tychonoff space. Then:
\begin{enumerate}
\item $\Gamma^1_{(\pi, E,\sigma)}$ is a sheaf for the supremum Grothendieck topology on $\bool{B}$;
\item if $(\pi, E, \sigma)$ is \'etal\'e then $\Gamma^1_{(\pi, E,\sigma)}$ is determined by its global sections;
\item if $\F\in\Presh(\bool{B})$, every natural transformation $\Theta:\F\to\G$ with $\G$ a sheaf for the supremum Grothendieck topology on $\bool{B}$ factors through\footnote{$\xi^\F$ has been defined in Definition \ref{rem:isoJ-sep}.}
 \[
\xi^\F:\F\mapsto\Gamma^1_{(\pi_\F,\Lambda^1_\F, \sigma_\F)};
\]
\item $\F$ is a sheaf for the sup Grothendieck topology if and only if $\xi^\F:\F\mapsto\Gamma^1_{(\pi_\F,\Lambda^1_\F, \sigma_\F)}$ is an isomorphism.
\end{enumerate}
 \end{theorem}
 \begin{proof}
Let $(\pi, E, \sigma)$ be a bundle over $\St(\bool{B})$ with $(E,\sigma)$ Tychonoff.
\begin{description}
\item[$\Gamma^1_{(\pi, E, \sigma)}$ is separated]
Let $s,t$ be Stone sections in $\Gamma_{(\pi, E, \sigma)}(b)$ which are both gluing of 
some matching family
$\bp{\dot{s}_d:d\in D}$ for some dense $D\subseteq\downarrow b$, so that $s_d\in\F(d)$ for all $d\in D$.
By Proposition \ref{prop:existgluingHaus} $s,t$ coincide.

\item[$\Gamma^1_{(\pi, E, \sigma)}$ is a sheaf]
Let $\bp{s_d:d\in D}$ be a matching family in 
\(\Gamma_{(\pi, E, \sigma)}\)
indexed by some dense $D\subseteq\downarrow b$. 
Then for every $d\in D$ $s_d$ is a Stone section, hence there is some dense $K_d\subseteq\downarrow d$ such that $s_d\restriction N_r=\dot{s}_{r,d}$ for $r\in K_d$ and some $s_{r,d}\in \F(r)$.
This gives that $K=\bigcup_{d\in D}K_d$ is a dense subset of $\downarrow b$. Furthermore for any $r\in K_d$ and $s \in K_e$ and $t$ refining $r,s$, $t$ is in $K_d\cap K_e$, we have that 
\[
\dot{s}_{r,d}\restriction N_{t}=s_d\restriction N_{t}=s_e\restriction N_{t}=\dot{s}_{s,e}\restriction N_{t};
\]
hence the family 
$\bp{\dot{s}_{r,d}:d\in D,\,r\in K_d}$ 
is matching. 
Thus by Proposition \ref{prop:existgluingHaus} it has a unique gluing, which is therefore a Stone section and belongs to $\Gamma^1_{(\pi, E, \sigma)}(b)$. 
\item[If $(\pi, E, \sigma)$ is \'etal\'e $\Gamma^1_{(\pi, E, \sigma)}$ is determined by its global sections]
First of all, for every $b\in\bool{B}$ $\Gamma^1_{(\pi, E, \sigma)}(b)$ is nonempty. Indeed, let $G\in N_b$ and let $e\in\pi^{-1}[\{G\}]$. Being $(\pi, E, \sigma)$ \'etal\'e, we can find $N_c\subseteq N_b$ and a section $s:N_c\to E$ such that $s[N_c]$ is an open neighborhood of $e$. In this way we obtain a family $C=\{c_i: i\in I\}$ with $\bigvee C=b$ and such that $\Gamma^0_{(\pi, E, \sigma)}(c_i)$ is nonempty. Let $s_i$ be an element in $\Gamma^0_{(\pi, E, \sigma)}(c_i)$. Up to a refinement to a maximal antichain $A\subseteq\downarrow C$, we can assume $c_i\wedge c_j=0$ for every $i\neq j$ in $I$. But then $\beta(\bigcup_{i\in I}s_i):N_b\to\beta(E)$ is an element in $\Gamma^1_{(\pi, E, \sigma)}(b)$. 

We can now prove that $\Gamma^1_{(\pi, E, \sigma)}$ is determined by its global sections: let $s\in\Gamma^1_{(\pi, E, \sigma)}(b)$. Take some $t\in\Gamma^1_{(\pi, E, \sigma)}(\neg b)$. Then $s\cup t\in\Gamma^1_{(\pi, E, \sigma)}(1_{\bool{B}})$ and $s=(s\cup t)\restriction b$.
\end{description}
Take now $\F, \G\in\Presh(\bool{B})$ with $\G$ a sheaf, and let $\delta:\F\to\G$ be a morphism. We have to prove that $\delta$ factors through $\xi^\F$.
 Let $s:N_b\to \Lambda_\F$ be a Stone section as witnessed by
$\bp{\dot{s}_d:d\in D}$ for some dense $D\subseteq\downarrow b$ with $s_d\in \F(d)$ for all $d\in D$. Set $\Upsilon_b(s)$ to be the unique element in $\G(p)$ which glues
the matching family $\bp{\delta_d(s_d):d\in D}$. One can check that
$\delta_b(f)=\Upsilon_b(\dot{f})$ for all $b\in\bool{B}$ and $f\in \F(b)$. Furthermore the required diagrams commute.

Now, if $\xi^\F$ is an isomorphism, then $\F$ is clearly a sheaf. Conversely, if $\F$ is a sheaf, each collation of families $\{\dot{f}_d: d\in D\}$ for some dense $D\subseteq\downarrow b$ is of the form $\dot{g}$ for some $g\in\F(b)$, and from this fact the conclusion of the proof is an easy exercise.
\end{proof}

We can obtain the following characterizations of which \'etal\'e bundles describe a sheaf for the sup Grothendieck topology.
\begin{lemma}\label{lem:charbundleissheaf}
Let $\bool{B}$ be a complete boolean algebra. If $\F$ is a sheaf for the sup Grothendieck topology on $\bool{B}$ then $\Lambda^1_\F$ is complete for gluings. Also, whenever $\F$ is separated and determined by its global sections, the converse holds.
\end{lemma}
\begin{proof}
Assume $\F$ is a sheaf and let $s:N_b\to \beta(\Lambda^1_\F)$ be a Stone section. By Fact \ref{fact:locsec} we can find $D\subseteq \bool{B}$ such that $\bigvee D=b$ and for every $d\in D$ there exists $f_d\in\F(d)$ such that $s\restriction N_d=\dot{f_d}$. Take a gluing $g\in\F(b)$ of $\{f_d: d\in D\}$. Then $s$ and $\dot{g}$ coincide on a dense subset of $N_b$ and thus they must be equal. To conclude it is enough to notice that the image of $\dot{g}$ is contained in $\Lambda^1_\F$.

Assume now $\F$ is separated and determined by its global sections. If $\Lambda^1_\F$ is complete for gluings, then $\Gamma^1_{\Lambda^1_\F}=\Gamma^0_{\Lambda^1_\F}$. Since $\F$ is separated, $\Gamma^0_{\Lambda^1_\F}$ is isomorphic to $\F$, and thus $\F$ is a sheaf.
\end{proof}

By Corollary \ref{char:Gamma0=Gamma1} we also have the following.
\begin{corollary}
 Let $\bool{B}$ be a complete Boolean algebra and $\M$ be a $\bool{B}$-valued model.
Then $\M$ has the mixing property if and only if the presheaf $\F_\M$ is isomorphic to the sheaf $\Gamma^1_{\Lambda^1_{\F_\M}}$.
\end{corollary}

\begin{example}\label{ex:fullnotimplmix2}
In Example \ref{non:mix} we are just saying that the continuous extension to the whole $\St(\bool{B})$ of $h:\bigcup_{a\in A}N_a\to \Lambda_{L(M^{\bool{B}})}$ such that \[h(G)=[\check{a}]_G\textit{ if and only if }a\in G\]
is a map $\St(\bool{B})\to\beta(\Lambda_{\F_\M})$ which is not induced by any element of $M^{\bool{B}}$.
\end{example}

\subsection{Mapping the category of presheaves on $\bool{B}$ onto the category of bundles on $\St(\bool{B})$}\label{subsec:adjunctions}

Up to now we have transformed presheaves into bundles and bundles into presheaves, but we have not paid attention to transform arrows between presheaves into arrows between bundle. Now we address this task  (note that the converse direction from bundles to preheaves has been covered in Def. \ref{def:functxi}).

\begin{definition}
Given  a complete Boolean algebra $\bool{B}$:

\begin{itemize}
\item
$\bundle(\bool{B})$ denotes the category whose objects are the bundles $(\pi,E,\sigma)$ over $\St(\bool{B},\tau_{\bool{B}})$ with $(E,\sigma)$ a Tychonoff space, and whose morphisms between $(\pi_1, E_1, \sigma_1)$ and $(\pi_2, E_2, \sigma_2)$ are the continuous maps $\psi:E_1\to E_2$ such that $\pi_1=\pi_2\circ\psi$.
\item 
$\bundleopen(\bool{B})$ denotes the subcategory of $\bundle(\bool{B})$ with same objects but whose morphisms include only the morphisms of $\bundle(\bool{B})$ which are \emph{open} maps.
\item
$\etale(\bool{B})$ is the full subcategory of $\bundle(\bool{B})$ whose objects are \'etal\'e spaces.
%
\item
$\exetale(\bool{B})$ is the full subcategory of $\bundle(\bool{B})$ whose objects are the strongly separated \'etal\'e spaces. 
\end{itemize}
\end{definition}


We need first to make the following observation (which should not come as a surprise in view of the proof of Lemma \ref{lem:JBundef} below):

\begin{fact}
Assume $(\pi_1,E_1,\sigma_1)$ and $(\pi_2,E_2,\sigma_2)$ are \'etal\'e bundles and 
$\psi:E_1\to E_2$ is a morphism between them. Then $\psi$ is an open map.

In particular $\exetale(\bool{B})$ is automatically a full subcategory of $\bundleopen(\bool{B})$.
\end{fact}
\begin{proof}
Since we are dealing with \'etal\'e bundles the images of continuous sections $s:U\to E_i$ defined on an open set $U$ of $\St(\bool{B})$ form a base for $\sigma_i$.
Therefore it suffices to check that $\psi\circ s$ is a local section of $E_2$, whenever $s$ is a local section of $E_1$ defined on an open set. But this follows, since $\pi_2\circ\psi=\pi_1$.
\end{proof}

\begin{definition}\label{def:LambdaB}
Let $\bool{B}$ be a complete Boolean algebra.
 The functor 
\[
\Lambda^{\bool{B}}:\Presh(\bool{B})\to\bundle(\bool{B})
\]
is defined as follows: 
\begin{description}
\item[On objects] It sends a presheaf $\F$ to the bundle $(\pi_\F, \Lambda^1_\F, \sigma_\F)$. 
\item[On morphisms] Let $\delta=\{\delta_b\}_{b\in\bool{B}}:\F_1\to\F_2$ be a natural transformation of presheaves on $\bool{B}$. Its image $\Lambda^{\bool{B}}(\delta):\Lambda^1_{\F_1}\to\Lambda^1_{\F_2}$ is such that, for every $G\in\St(\bool{B})$ and $f\in\F(b)$ with $b\in G$, $\Lambda^\bool{B}\delta([f]_{G}):=[\delta_b(f)]_{G}$.
\end{description}
\end{definition}

\begin{lemma}\label{lem:JBundef}
Let $\bool{B}$ be a complete Boolean algebra.
$\Lambda^{\bool{B}}:\Presh(\bool{B})\to\bundle(\bool{B})$ is a well defined functor whose image is contained in $\exetale(\bool{B})$.
\end{lemma}

\begin{proof}
The fact that $\pi_\F:\Lambda^1_\F\to\St(\bool{B})$  is an \'etal\'e bundle  which is locally compact Hausdorff and extremally disconnected has already been established. 
To complete the proof
we have to show that, if $\delta:\F\to\G$ is a morphism, then $\Lambda^{\bool{B}}(\delta)$ is continuous (and also open).

It is easily seen to be open, as for any $f\in\F(b)$, $\Lambda^{\bool{B}}(\delta)[\dot{f}[N_b]]$ is
-- by definition -- $\dot{\delta_b(f)}[N_b]$, which is open in $\Lambda_\G$.

To see that $\Lambda^{\bool{B}}(\delta)$ is continuous, first observe that $[g]_{H}\in\Lambda^{\bool{B}}(\delta)[\Lambda_\F]$ 
if and only if for some $b\in H$ and for some $f\in \F(b)$ we have
$[g]_{H}=[\delta_b(f)]_{H}$; w.l.o.g. we may further refine $b$ still in $H$ and assume $g,\delta_b(f)$ are both in $\G(b)$ and $D=\bp{d\leq b: g\restriction d=\delta_b(f)\restriction d}$ is dense below $b$. 

Now to show continuity, 
take $[g]_{H}\in\Lambda^{\bool{B}}(\delta)[\Lambda_\F]$ and $\dot{g}[N_c]$ basic open neighborhood of $[g]_H$ given by some some $c\in H$.
By the above remark, we can find $b\in H$ refining $c$, and $f\in \F(b)$, so that 
\[
D=\bp{d\leq b: g\restriction d=\delta_b(f)\restriction d}
\] 
is dense below $b$. Then for any $d\in D$
\[
\Lambda^{\bool{B}}(\delta)[\dot{f}[N_d]]=\dot{g}[N_d]\subseteq\dot{g}[N_c],
\]
as was to be shown.
\end{proof}


\subsection{The separification operator in terms of  strongly separated \'etal\'e bundles}

\begin{proposition}\label{lem:Jsepbindles1}
Let $\bool{B}$ be a complete Boolean algebra. 

Then: 
\begin{enumerate}
\item  \label{lem:Jsepbindles1-2}
there exists a natural transformation 
$\Lambda^{\bool{B}}\circ\Gamma^{0,\bool{B}}\to\Id_{\bundle(\bool{B})}$ which becomes an isomorphism when restricted to $\etale(\bool{B})$;
\item \label{lem:Jsepbindles1-3}
$\Gamma^{0,\bool{B}}$ is right adjoint to $\Lambda^{\bool{B}}:\Presh(\bool{B})\to\bundle(\bool{B})$ and
\[\Xi^{\bool{B}}:=\{\xi^\F\}_{\F\in\Presh(\bool{B})}:\Id_{\Presh(\bool{B})}\to\Gamma^{0,\bool{B}}\circ \Lambda^{\bool{B}}\] is the unit of the adjunction;
\item \label{lem:Jsepbindles1-4}
$\Gamma^{0,\bool{B}}\circ \Lambda^{\bool{B}}$ is the separification operator for the Sup Grothendieck topology on $\bool{B}$. 
\end{enumerate}
\end{proposition}
\begin{proof}
\emph{}

\begin{description}
\item[\ref{lem:Jsepbindles1-2}]
We define the natural transformation of functors by sending, for $b\in H$ and $s\in\Lambda_{(\pi, E, \sigma)}(b)$, the equivalence class $[s]_{ H}$ (which is a point of $\Lambda^{\bool{B}}\circ\Gamma^{0,\bool{B}}(\pi,E,\sigma)$) to
$s(H)\in E$. Note that $[s]_{H}=[t]_{G}$ if and only if $G=H$ and for some $b\in G$
\[
\bigvee\bp{d\leq b: s\restriction N_d=t\restriction N_d}=b
\]
if and only if  $s\restriction N_b=t\restriction N_b$ for some $b\in G$. 
This gives that $s(H)=t(G)$ if $[s]_H=[t]_G$ and also that the converse implication holds when $(\pi,E,\sigma)$ is \'etal\'e.

Therefore whenever $(\pi, E, \sigma)$ is an \'etal\'e space, this map is an open continuous bijection, hence an homeomorphism, between
$\Lambda^1_{\Gamma^{0,\bool{B}}(\pi,E,\sigma)}$ and
$E$.
\item[\ref{lem:Jsepbindles1-3}]
Clearly $\Xi^{\bool{B}}$ is a natural transformation $\Id_{\Presh(\bool{B})}\to \Gamma^{0,\bool{B}}\circ\Lambda^{\bool{B}}$ and the universal property of $\Xi^{\bool{B}}$ given by Proposition \ref{prop:univJsep}, combined with the first item, tells that $\Xi^{\bool{B}}$ is the unit of an adjunction in which the counit is the natural transformation 
$\Gamma^{0,\bool{B}}\circ\Lambda^{\bool{B}}\to\Id_{\bundle(\bool{B})}$ described in item \ref{lem:Jsepbindles1-2}.
\item[\ref{lem:Jsepbindles1-4}]
This is immediate by Proposition \ref{prop:univJsep} and the previous items.

\end{description}

\end{proof}

\subsection{The sheafification operator in terms of gluing completions of \'etal\'e spaces}

We can now address the study of the equivalence between sheaves for the sup topology on $\bool{B}$ and the bundles in $\exetale{\bool{B}}$, by taking into account 
Definitions \ref{def:functxi}, \ref{def:gammapi}, and \ref{def:LambdaB}. 

%
%


To do so, we need to slightly modify the categories of bundles we are dealing with, in particular we need to have a more generous notion of morphism:

\begin{remark}
The functor
$\Lambda^\bool{B}:\Presh(\bool{B})\to\bundle(\bool{B})$ has a unique adjoint, namely the sepaarification functor $\Gamma^{0, \bool{B}}:\exetale(\bool{B})\to\Presh(\bool{B})$. 
We will show that this adjoint  restricts to an equivalence of categories on $\Lambda^\bool{B}[\Sh_\mathrm{sup}(\bool{B})]$. 

Now $\Lambda^\bool{B}$ can also be seen as a functor with range in $\exetale(\bool{B})$ which is a full subcategory of $\bundle(\bool{B})$, hence it is natural to expect that 
the restriction of  $\Gamma^{0,\bool{B}}$ to this category should remain an adjoint of $\Lambda^\bool{B}$ seen as a functor with range contained in this restricted category.
If this were the case the adjunction on this restricted category would give natural isomorphisms
\[
\mathrm{Hom}_{\Presh(\bool{B})}(\F, \Gamma^0_{(\pi,E,\sigma)})\cong
\mathrm{Hom}_{\exetale(\bool{B})}(\Lambda_\F, (\pi,E,\sigma));
\]
however when $\Gamma^0_{(\pi,E,\sigma)}$ is not a sheaf, an arrow $\Phi:\F\to \Gamma^0_{(\pi,E,\sigma)}$ gives rise to a  continuous function defined on $\Lambda_\F$ and range in $\beta(E)$, but possibly not contained in $E$.
Such a function has only a dense open subset on which its range is contained in $E$; thus the request that the above is a natural isomorphism imposes to include such functions among the morphism on \'etal\'e bundles, as we do in Def. \ref{def:partmorphsimsExEtBpresh} below. It will then follw that the requested natural isomorphism 
\[
\mathrm{Hom}_{\Presh(\bool{B})}(\F, \Gamma^0(\pi,E,\sigma))\cong
\mathrm{Hom}_{\stexetale(\bool{B})}(\Lambda_\F, (\pi,E,\sigma)).
\]
exists for the category $\stexetale(\bool{B})$ with the same objects as $\exetale(\bool{B})$ and a wider family of morphisms (to be defined below).
\end{remark}

\begin{definition}\label{def:partmorphsimsExEtBpresh}
We define the category $\stexetale(\bool{B})$ in the following way:
\begin{itemize}
\item objects of $\stexetale(\bool{B})$ are exactly the same objects of $\exetale(\bool{B})$;
\item if $(\pi_1, E_1, \sigma_1), (\pi_2, E_2; \sigma_2)$ are objects in $\stexetale(\bool{B})$, a morphism between them is a continuous open partial map $\varphi:E_1\to E_2$ such that:
\begin{enumerate}
\item $\pi_1=\pi_2\circ\varphi$,
\item the domain of $\varphi$ is open dense in $E_1$.
\end{enumerate}
\item the composition of morphisms is the usual composition of partial functions;
\item the identity of an object $(\pi, E, \sigma)$ is the (total) identity $E\to E$.
\end{itemize}
\end{definition}
Notice that isomorphisms in  $\stexetale(\bool{B})$ and in $\exetale(\bool{B})$ are exactly the same functions.

We can safely think $\Lambda^{\bool{B}}$ as a functor $\Presh(\bool{B})\to\stexetale(\bool{B})$.

\begin{definition}
The functor $\Gamma^{1, \bool{B}}:\stexetale(\bool{B})\to\Presh(\bool{B})$ is defined in the following way:
\begin{description}
\item[on objects] it sends $(\pi, E, \sigma)$ to $\Gamma^{1}_{(\pi, E, \sigma)}$;
\item[on morphisms] it sends $\varphi:(\pi_1, E_1, \sigma_1)\to(\pi_2, E_2, \sigma_2)$ to the natural transformation $\{(\Gamma^1_\varphi)_b\}_{b\in\bool{B}}$ where, if $s\in\Gamma^1_{(\pi_1, E_1, \sigma_1)}(b)$, then
\[
(\Gamma^1_\varphi)_b(s):=\beta(\varphi)\circ s.
\]
\end{description}
\end{definition}
\begin{definition}\label{def:gluingcompletion}
Let $(\pi,E,\sigma)$ be a strongly separated bundle on an extremally disconnected compact Hausdorff space $(X,\tau)$. Its \emph{gluing completion} $(\pi^*,E^*,\sigma^*)$ is obtained by letting:
\begin{itemize}
\item 
$E^*$ consists of those points of $\beta(E)$ which are in the pointwise image of some Stone section $s:U\to \beta(E)$ for $(\pi,E,\sigma)$,
\item $\sigma^*$ is the topology inherited by $E^*$ as a subspace of $(\beta(E),\beta(\sigma))$,
\item $\pi^*$ is the restriction to $E^*$ of $\beta(\pi)$.
\end{itemize}
\end{definition}
For every object $(\pi, E, \sigma)$ in $\stexetale(\bool{B})$ we define a morphism $\epsilon_{(\pi, E, \sigma)}:(\Lambda^{\bool{B}}\circ\Gamma^{1,\bool{B}})(\pi, E, \sigma)\to(\pi, E, \sigma)$ by taking as domain the set
\[
D_E:=\{[s]_G: s:N_b\to\beta(E)\textit{ is a Stone section and }s(G)\in E\}
\]
and letting
\[
\epsilon_{(\pi, E, \sigma)}([s]_G):=s(G)\textit{ for every Stone section }s:N_b\to\beta(E).
\]
Notice that $D_E$ is open dense because the basic open sets of $\Lambda^1_{\Gamma^1_{(\pi, E, \sigma)}}$ are of the form $\dot{s}[N_b]$ for some $s\in\Gamma^1_{(\pi, E, \sigma)}(b)$ and for each $\dot{s}[N_b]$ we can find a $D\subseteq\downarrow b$ with $\bigvee D=b$ such that $s\restriction N_d$ has image in $E$ for every $d\in D$.
\footnote{Here we use that $\stexetale(\bool{B})$ has more morphisms than $\exetale(\bool{B})$: if $E$ is a strongly separated \'etal\'e bundle on $\St(\bool{B})$ and $s:\St(\bool{B}) \to\beta(E)$ is a Stone section on it, then $s^{-1}[E]$ might only be a dense open subset of $\St(\bool{B})$, hence in order to include such an $s$ among the legitimate morphisms of the category from the identity bundle on $\St(\bool{B})$ to $E$ , we are forced to impose that such morphisms  are defined only on a (possibly proper) dense open subset of $\St(\bool{B})$. The inclusion of such Stone sections $s$ among the legitimate morphism of the category is necessary in order to grant that the counit we have just introduced is well defined.}

\begin{proposition}\label{fact:imepsglucomp}
The morphism $\epsilon_{(\pi, E, \sigma)}$ is well defined. 
Moreover, $\epsilon_{(\pi, E, \sigma)}$ has a unique continuous extension $\Lambda^1_{\Gamma^1_{(\pi, E, \sigma)}}\to\beta(E)$ whose image is $E^*$. In particular, $\epsilon_{(\pi, E, \sigma)}$ induces an isomorphism between $(\Lambda^{\bool{B}}\circ\Gamma^{1,\bool{B}})(\pi, E, \sigma)$ and $(\pi^*,E^*,\sigma^*)$.
\end{proposition}
\begin{proof}
To prove that $\epsilon_{(\pi, E, \sigma)}$ is well defined we have to show that $D_E$ in dense and open.
To this end, observe that, if $s:N_b\to E$ a Stone section and $D_s$ is a dense open subset of $N_b$ such that $s\restriction D_s$ has range in $E$, then $\dot{s}[D_s]=D\cap \dot{s}[N_b]$ is open dense in $\dot{s}[N_b]$. Hence to conclude it is enough to notice that $D_E=\bigcup\{D_s: s\textit{ is a Stone section}\}$.

Moreover, $\epsilon_{(\pi, E, \sigma)}$ is continuous. Indeed, if $V\subseteq E$ is open, then 
\[
\epsilon_{(\pi, E, \sigma)}^{-1}[V]=\bigcup\bp{\dot{s}[s^{-1}[V]]: s\in\Gamma^1_{(\pi, E, \sigma)}(b)\textit{ for some }b\in\bool{B}}.
\] 
Furthermore, $\epsilon_{(\pi, E, \sigma)}$ is open: observe that $\epsilon_{(\pi, E, \sigma)}[\dot{s}[N_b]]=s[N_b]$, which is open since each section $s$ is open. 

The fact that $\pi\circ\epsilon_{(\pi, E, \sigma)}=\pi_{(\pi, E, \sigma)}$ is easy to check, and clearly $\epsilon_{(\pi, E, \sigma)}$ can be extended to the whole $\Lambda_{\Gamma^1_{(\pi, E, \sigma)}}$ by taking $\beta(\epsilon_{(\pi, E, \sigma)})\restriction \epsilon_{(\pi, E, \sigma)}$.

Finally, it is immediate from the definition that$\epsilon_{(\pi, E, \sigma)}$ induces an isomorphism between $(\Lambda^{\bool{B}}\circ\Gamma^{1,\bool{B}})(\pi, E, \sigma)$ and $(\pi^*,E^*,\sigma^*)$.
\end{proof}
\begin{corollary}
If $(\pi, E, \sigma)$ is an object in $\stexetale(\bool{B})$, then its gluing completion is complete for gluings.
\end{corollary}
\begin{proof}
Being $(\pi^*, E^*, \sigma^*)$ isomorphic to $(\Lambda^{\bool{B}}\circ\Gamma^{1,\bool{B}})(\pi, E, \sigma)$, its enough to note that (by Theorem \ref{mainthm:sheaf}) $\Gamma^{1,\bool{B}}(\pi, E, \sigma)$ is a sheaf and thus (by Lemma \ref{lem:charbundleissheaf}) $(\Lambda^{\bool{B}}\circ\Gamma^{1,\bool{B}})(\pi, E, \sigma)$ is complete for gluings.
\end{proof}
\begin{corollary}
If $(\pi, E, \sigma)$ is complete for gluings then $\epsilon_{(\pi, E, \sigma)}$ is an isomorphism, and thus its gluing completion is itself.
\end{corollary}
\begin{proof}
If $(\pi, E, \sigma)$ is complete for gluings then every Stone section $s:N_b\to\beta(E)$ has range in $E$, and thus $\epsilon_{(\pi, E, \sigma)}([s]_G)=s(G)\in E$ for every $G\in\St(\bool{B})$. Hence the domain of $\epsilon_{(\pi, E, \sigma)}$ is the whole $(\Lambda^{\bool{B}}\circ\Gamma^{1,\bool{B}})(\pi, E, \sigma)$ and so (by Fact \ref{fact:imepsglucomp}) we have that $E=E^*$.
\end{proof}
We sum up all these observations in the following result. 
\begin{theorem}
The pair $(\Lambda^{\bool{B}}, \Gamma^{1, \bool{B}})$ is an adjuction between $\Presh(\bool{B})$ and $\stexetale(\bool{B})$ with unit $\Xi^{\bool{B}}$ and counit $\{\epsilon_{(\pi, E, \sigma)}\}_{(\pi, E, \sigma)\in\stexetale(\bool{B})}$.\\
It specializes to an equivalence of categories between $\Sh_\mathrm{sup}(\bool{B})$ and the full subcategory of $\stexetale(\bool{B})$ generated by \'etal\'e spaces which are complete for gluings.
\end{theorem}
\begin{proof}
To prove it, we have to show that for every $\F\in\Presh(\bool{B})$ and $(\pi, E, \sigma)\in\stexetale(\bool{B}$ the following hold:
\[
\Id_{\Lambda^1_\F}=\epsilon_{(\pi_\F, \Lambda^1_\F, \sigma_\F)}\circ\Lambda^\bool{B}(\xi^\F)\quad\quad \Id_{\Gamma^1_{(\pi, E, \sigma)}}=\Gamma^{1, \bool{B}}(\epsilon_{(\pi, E, \sigma)})\circ\xi^{\Gamma^1_{(\pi, E, \sigma)}}.
\]
The proof ot these equalities is just a careful application of the definitions, and we leave it to the reader.
\end{proof}
We can now state our main result, proving the fact that $\Gamma^{1,\bool{B}}\circ\Lambda^{\bool{B}}$ is the sheafification functor for the Sup Grothendieck topology on $\bool{B}$.
\begin{theorem}\label{2mainthm:sheaf}
 Let $\bool{B}$ be a complete Boolean algebra. 
The composition $\Gamma^{1,\bool{B}}\circ\Lambda^{\bool{B}}$ is left adjoint to the inclusion of $\Sh_{\mathrm{sup}}(\bool{B})$ into $\Presh(\bool{B})$.
\end{theorem}
\begin{proof}
Let $\F\in\Presh(\bool{B})$ and $\G\in\Sh_{\mathrm{sup}}(\bool{B})$. 
We need to define a bijection
\[
\Phi_{\F, \G}:\Hom_{\Sh_{\mathrm{sup}}(\bool{B})}(\Gamma^{1,\bool{B}}\circ\Lambda^\bool{B}(\F), \G)\to\Hom_{\Presh(\bool{B})}(\F,\G)
\]
which is natural in $\F$ and $\G$. We define a natural transformation $\{\Phi_{\F, \G}\}_{\F\in\Presh(\bool{B}))}$ which is an isomorphism between the two functors
\[
\Hom_{\Sh_{\mathrm{sup}}(\bool{B})}(\Gamma^{1,\bool{B}}\circ\Lambda^{\bool{B}}(-), \G),\Hom_{\Presh(\bool{B})}(-, \G):\Presh(\bool{B})\to\dSet.
\]
Let $\eta=\{\eta_b\}_{b\in\bool{B}}\in\Hom_{\Sh_{\mathrm{sup}}(\bool{B})}(\Gamma^{1,\bool{B}}\circ\Lambda^{\bool{B}}(\F), \G)$ be a natural transformation.
We let $\Phi_{\F, \G}(\eta):=\{\Phi(\eta)_b\}_{b\in\bool{B}}$ be the natural transformation with components 
\[
\begin{aligned}
\Phi(\eta)_{b}:\F(b)&\to\G(b)\\
f&\mapsto\eta_b(\dot{f})
\end{aligned}.
\]
 The fact that $\Phi_{\F, \G}$ is well defined is left to the reader. To see that it is bijective, we provide an inverse
 \[
 \Psi_{\F, \G}:\Hom_{\Presh(\bool{B})}(\F, \G)\to\Hom_{\Sh_{\mathrm{sup}}(\bool{B})}(\Gamma^{1,\bool{B}}\circ\Lambda^{\bool{B}}(\F), \G).
 \]
 Let $\delta=\{\delta_b\}_{b\in\bool{B}}\in\Hom_{\Presh(\bool{B})}(\F, \G)$. We define $\Psi_{\F, \G}(\delta)=\{\Psi(\delta)_b\}_{b\in\bool{B}}$ in the following way. Let $s\in\Gamma^1_{\Lambda^{\bool{B}}(\F)}(b)$. By construction, we can find a $D\subseteq\downarrow b$ with $\bigvee D=b$ and a family $\{f_d\in\F(d): d\in D\}$ such that  $s\restriction N_d=\dot{f}_d$. Being $\G$ a sheaf, let $g\in\G(b)$ be the collation of the matching family \[\{\delta_{d}(f_d)\in\G(d): d\in D\}.\] Then we let $\Psi(\delta)_{b}(s):=g$.
 
 The fact that $\Phi_{\F, \G}$ and $\Psi_{\F, \G}$ are one the inverse of the other is almost immediate. The naturality in $\F$ and $\G$ is a direct consequence of the fact that, if $\G\in\Sh_{\mathrm{sup}}(\bool{B})$, then $(\Gamma^{1,\bool{B}}\circ\Lambda^{\bool{B}})(\G)\cong\G$ via the expected natural transformation. 
\end{proof}

\subsection{Final remarks}

We can also obtain a natural notion of mixification for Boolean valued models:
\begin{corollary}\label{cor:mixbvm}
Let $\M$ be a $\bool{B}$-valued model for a complete boolean algebra $\bool{B}$. Then
\[\M^+:=(R\circ\Gamma^{1,\bool{B}}\circ\Lambda^{\bool{B}}\circ L)(\M)\] is a Boolean valued model with the mixing property and such that the map \[\Phi_\M:\sigma\mapsto (\Gamma^{1,\bool{B}}\circ \Lambda^{\bool{B}})_{1_{\bool{B}}}([\sigma]_{F_{1_\bool{B}}})\] coupled with the identity on $\bool{B}$ into 
 is an elementary embedding\footnote{Recall Definition \ref{def:elembvm}.} of $\M$ into a $\bool{B}$-valued model with the mixing property.
\end{corollary}
\begin{proof}
It is clear that $\M^+$ is a $\bool{B}$-valued model with the mixing property. We leave to the reader to check the remaining items. In essence the key point is that given a  family $\bp{\sigma_i:i\in I}$ such that 
\[
b=\Qp{\exists x \phi(x)}^{\M}_{\bool{B}}=\bigvee_{i\in I}\Qp{\phi(\sigma_i)}^{\M}_{\bool{B}},
\] 
one can find $J\subseteq I$ and a maximal antichain $A=\bp{a_j:j\in J}$ below $b$ such that:
\begin{itemize} 
\item
$\Qp{\phi(\sigma_j)}^{\M}_{\bool{B}}\geq a_j$ for all $j\in J$;
\item
$\bp{[\sigma_j]_{F_{a_j}}:j\in J}$ is a collating family in
$L(\M)$;
\item the collation $\sigma^*\restriction N_b$ of this family is obtained from a Stone section $\sigma^*$ of  
$(\Lambda^{\bool{B}}\circ L)(\M)$.
\end{itemize}
One has then to check that 
\[
N_b=\Qp{\exists x\phi(x)}^{\M^+}_{\bool{B}}=\Qp{\phi(\sigma^*)}^{\M^+}_{\bool{B}}=\bigvee_{\bool{B}} 
\bp{N_{\Qp{\phi(\sigma_j)}^{\M}_{\bool{B}}}:j\in J}
\]
 holds in $\M^+$.
\end{proof}
In particular $R\circ\Gamma^{1,\bool{B}}\circ\Lambda^{\bool{B}}\circ L$ provides the canonical ``mixification'' of a Boolean valued model, canonical in the sense that it preserves the Boolean valued semantics and it embeds  in any other mixing $\bool{B}$-valued model into which  $\M^+$ embeds (since $\F_{\M^+}$ is the sheafification of $\F_\M$ and sheaves correspond to mixing models).

\begin{example}\label{exm:sheafLR}
Let us now outline why the $\bool{MALG}$-names for real numbers correspond to the sheafification of the presheaf 
$\F$ assigning to each $[A]\in\bool{MALG}$ the space $L^\infty(A)$, for $A$ a measurable subset of $\mathbb{R}$.\\
It is not hard to check that the Boolean valued model 
\[
\M=\bp{\tau\in V^\bool{MALG}: \Qp{\tau\in\mathbb{R}}_\bool{MALG}=1_\bool{MALG}}
\]
is a Boolean valued model with the mixing property (for details see \cite{VIAVAC15}). 
Then $L(\M)$ is a sheaf. 
Let also $\beta_0(\mathbb{R})=\mathbb{R}\cup\{\infty\}$ be the one point compactification of $\mathbb{R}$ and let $X=\St(\malg)$.
By the results of \cite{MR779056,MR759416,VIAVAC15} 
\[
L(\M)(1_\malg)\cong C^+(X)=\bp{f:X\to\beta(\mathbb{R}): \, f\text{ is continuous and }f^{-1}[\beta(\mathbb{R})\setminus\mathbb{R}]\text{ is nowhere dense}}.
\]
It is also possible to see that the Gelfand transform of $L^\infty(\mathbb{R})$ extends to a natural isomorphism 
\[
C^+(X)\cong L^{\infty+}(\mathbb{R})=\bp{f:\mathbb{R}\to\beta_0(\mathbb{R}):\,f\text{ is measurable and } \nu(f^{-1}[\bp{\infty}])=0},
\]
where $\nu$ is Lebesgue measure (for details see  \cite{VIAVAC15}, natural here means that it is a $\bool{MALG}$-equivalence according to Definition \ref{def:boolvalmorph} when for $f,g\in C^+(X)$ we set $\Qp{f=g}=\Reg{\bp{G\in X: f(G)=g(G)}}$ and for 
$f^*,g^*\in L^{\infty+}(\mathbb{R})$ $\Qp{f^*=g^*}=[\bp{a\in \mathbb{R}: f^*(a)=g^*(a)}]_{\mathsf{Null}}$).\\
In particular this equivalence shows that for all $A\subseteq\mathbb{R}$ measurable
\[
L(\M)([A]_{\mathsf{Null}})\cong\bp{f:A\to\beta_0(\mathbb{R}): \, f\text{ is continuous and }f^{-1}[\bp{\infty}]\text{ is nowhere dense}}
\]
and that $C^+(X)$ is naturally identified with $\Gamma^1_{\Lambda^1_{\F_\M}}$.

Note that, since we let the functions of interest take values outside of $\mathbb{R}$ on a negligible set, it (can be shown that) it is irrelevant which particular Hausdorff compactification of $\mathbb{R}$ one takes to define $C^+(X)$ or $L^{\infty+}(\mathbb{R})$. However it is important that we choose the Stone Cech compactification of $\mathbb{R}$ to define $C^+(X)$ in order to have a smooth notion of morphism between the associated bundles of which these spaces give the global sections.
\end{example}

\bibliographystyle{plain}
	\bibliography{Biblio.bib}


\end{document}

%% file: macros.tex
\theoremstyle{plain}

	\newtheorem{theorem}{Theorem}[section]
	\newtheorem{proposition}[theorem]{Proposition}
	\newtheorem{lemma}[theorem]{Lemma}
	\newtheorem{corollary}[theorem]{Corollary}
	\newtheorem{fact}[theorem]{Fact}

\theoremstyle{definition}
	\newtheorem{definition}[theorem]{Definition}
	\newtheorem{notation}[theorem]{Notation}
	\newtheorem{example}[theorem]{Example}
	\newtheorem{notation*}{Notation}

\theoremstyle{plain}

\theoremstyle{remark}
	\newtheorem{remark}[theorem]{Remark}
	
\newcommand{\ZFC}{\ensuremath{\mathsf{ZFC}}}

\newcommand{\N}{\mathcal{N}}
\newcommand{\M}{\mathcal{M}}

\newcommand{\dSet}{\mathsf{Set}}

\DeclareMathOperator{\dom}{dom}

\DeclareMathOperator{\ran}{ran}

\DeclareMathOperator{\St}{St}

\DeclareMathOperator{\RO}{\mathsf{RO}}
\DeclareMathOperator{\CLOP}{\mathsf{CLOP}}

\newcommand{\bool}[1]{\mathsf{#1}}

\newcommand{\pow}[1]{\mathcal{P}\left(#1\right)}
\newcommand{\Int}[1]{\mathrm{Int}\left(#1\right)}
\newcommand{\Cl}[1]{\mathrm{Cl}\left(#1\right)}
\newcommand{\Reg}[1]{\mathrm{Reg}\left(#1\right)}

\newcommand{\cp}[1]{\left( #1 \right)}
\newcommand{\qp}[1]{\left[ #1 \right]}
\newcommand{\Qp}[1]{\left\llbracket #1 \right\rrbracket}

\newcommand{\bp}[1]{\left\lbrace #1 \right\rbrace}

\newcommand{\lrao} {\leftrightarrow}

\theoremstyle{definition}

\newcommand{\UnderTilde}[1]{{\setbox1=\hbox{$#1$}\baselineskip=0pt\vtop{\hbox{$#1$}\hbox to\wd1{\hfil$\sim$\hfil}}}{}}

